\documentclass[a4paper, 11pt]{article}
\title{On the equivariant blow-Nash classification of simple invariant Nash germs}
\author{Fabien Priziac}
\date{}

\usepackage{amsfonts}
\usepackage{amsmath}
\usepackage{amsthm}
\usepackage{amssymb}
\usepackage[all]{xy}
\usepackage{graphicx}

\newtheorem{de}{Definition}[section]
\newtheorem{theo}[de]{Theorem}
\newtheorem{prop}[de]{Proposition}
\newtheorem{cor}[de]{Corollary}

\newtheorem{lem}[de]{Lemma}

\theoremstyle{remark}
\newtheorem{rem}[de]{Remark}

\begin{document}

\maketitle

\begin{abstract} We make progress towards the classification of simple Nash germs invariant under the involution changing the sign of the first coordinate, with respect to equivariant blow-Nash equivalence, which is an equivariant Nash version of blow-analytic equivalence, taking advantage of invariants for this relation, the equivariant zeta functions.
\end{abstract}
\footnote{
Keywords : simple invariant Nash germs, equivariant blow-Nash equivalence, $ABCDEF$-singularities, equivariant zeta functions, equivariant virtual Poincar\'e series.
\\
{\it 2010 Mathematics Subject Classification :} 14B05, 14P20, 14P25, 32S15, 57S17, 57S25.
 
Research partially supported by a Japan Society for the Promotion of Science (JSPS) Postdoctoral Fellowship.
}

\section{Introduction}

The classification of real analytic germs requires to choose carefully the used equivalence relation. One may think about the (right) $C^1$-equivalence. However, it is too strong, as illustrated by the example of the Whitney family $f_t(x,y) = xy(y-x)(y-tx)$, $t >1$ ($f_t$ and $f_{t'}$ are $C^1$-equivalent if and only if $t~=~t'$), while the topological equivalence is too rough. In \cite{Kuo}, T.-C. Kuo suggested an equivalence relation for which Whitney family has only one equivalence class : the blow-analytic equivalence. More generally, any analytically parametrized family of isolated singularities has a locally finite classification with respect to blow-analytic equivalence. 

Two real analytic germs are said blow-analytically equivalent if, roughly speaking, they become analytically equivalent after compositions with real modifications, e.g. compositions of blowings-up along smooth centers. From the definition of this equivalence relation, further studies on real analytic germs were stimulated. In particular, invariants have been constructed for blow-analytic equivalence, like the Fukui invariants (\cite{Fukui}) as well as the zeta functions of S. Koike and A. Parusi\'nski (\cite{KoiPar}), inspired by the motivic zeta functions of J. Denef and F. Loeser (\cite{DL-GA}), using the Euler characteristic with compact supports as a motivic measure.

A refinement of blow-analytic equivalence has been defined for Nash germs, that is germs of real analytic functions with a semialgebraic graph, by G. Fichou  in \cite{GF-ZF} : the blow-Nash equivalence, that is Nash equivalence after compositions with Nash modifications. The involved algebraicity allowed him to use the virtual Poincar\'e polynomial (\cite{MCP-VB} and \cite{GF-MI}), which is an additive and multiplicative invariant on $\mathcal{AS}$ sets (\cite{Kur} and \cite{KP}) encoding more information than the Euler characteristic with compact supports, in order to define new zeta functions, invariant for the blow-Nash equivalence of Nash germs. Recently, J.-B. Campesato gave in \cite{JBC} an equivalent alternative definition of blow-Nash equivalence as arc-analytic equivalence, proving that the blow-Nash equivalence of \cite{GF-ZF} was indeed an equivalence relation, and defined a more general invariant for it, the motivic local zeta function. 
 
In \cite{GF-BNT}, G. Fichou used his zeta functions of \cite{GF-ZF} to classify the simple Nash germs (a germ is called simple if sufficiently small perturbations provide only finitely many analytic classes) with respect to blow-Nash equivalence. He showed that this classification actually coincides with the real analytic one, that is the $ADE$-classification of \cite{AGZV}. An analog result for blow-analytic equivalence is not known.

In this paper, we are interested in real analytic germs invariant under right composition with the action of the group $G = \mathbb{Z}/2 \mathbb{Z}$ only changing the sign of the first coordinate (that we will simply call invariant germs). In \cite{Pri-EZF}, we defined the equivariant blow-Nash equivalence for invariant Nash germs, which is, roughly speaking, an equivariant Nash equivalence after compositions with equivariant Nash modifications. Using the equivariant virtual Poincar\'e series (\cite{GF}), which is an additive invariant on $G$-$\mathcal{AS}$ sets, as a motivic measure, we constructed ``equivariant'' zeta functions which are invariants for the equivariant blow-Nash equivalence.

Similarly to the non-equivariant frame, we ask if the equivariant blow-Nash classification of invariant Nash germs could coincide with the equivariant Nash classification for sufficiently ``tame'' invariant singularities. The equivariant analytic classification of simple invariant real analytic germs has been established by V. I. Arnold in \cite{Arn} and recalled in \cite{Gor} by V. V. Goryunov. The representatives for this classification are the invariant singularities $A_k$, $B_k$, $C_k$, $D_k$, $E_6$, $E_7$, $E_8$ and $F_4$ (see theorem \ref{arngor} below). We will first show that a simple invariant Nash germ is $G$-blow-Nash equivalent (and even $G$-Nash equivalent) to one of these germs. The largest part of our study will then consist in trying to distinguish, with respect to $G$-blow-Nash equivalence, the invariant $ABCDEF$-singularities, using notably the equivariant zeta functions.

For some cases, we will be faced with either the equality of the respective equivariant zeta functions of a couple of germs or the fact that they are equal if and only if the respective equivariant virtual Poincar\'e series of specific sets are equal. The former situation is in particular due to the fact that the equivariant virtual Poincar\'e series can not distinguish two different algebraic actions on a same sphere as soon as there is at least one fixed point. As for the latter situation, we do not know if the invariance of the virtual Poincar\'e polynomial under bijection with $\mathcal{AS}$ graph (see \cite{MCP}) ``generalizes'' to an invariance of the equivariant virtual Poincar\'e series under equivariant bijection with $\mathcal{AS}$ graph. If this was proven to be true, it should allow to compute all the coefficients of the considered equivariant zeta functions.
\\

The next section is devoted to the equivariant Nash classification of simple invariant Nash germs : we prove that it coincides with the equivariant real analytic classification of \cite{Arn} and \cite{Gor}. Indeed, two invariant Nash germs are equivariantly Nash equivalent if and only if they are equivariantly analytically equivalent (proposition \ref{equivnasheqiffequivaneq}). This can be deduced from an equivariant Nash approximation theorem of E. Bierstone and P. Milman in \cite{BMinv}.

In section \ref{sectioneqblnequiv}, we justify the fact that a germ $G$-Nash equivalent to a germ of the list $ABCDEF$ is in particular $G$-blow-Nash equivalent to it. On the other hand, one can notice that, forgetting the $G$-action, the invariant singularities $A_k$ and $B_k$, resp. $C_k$ and $D_k$, $E_6$ and $F_4$, are both $A$-, resp. $D$-, $E$-, singularities. Since equivariant blow-Nash equivalence is a particular case of blow-Nash equivalence and because the $ADE$-singularities are not blow-Nash equivalent to one another (\cite{GF-BNT}), we are reduced to compare, with respect to $G$-blow-Nash equivalence, the invariant germs of the families $A_k$ and $B_k$, resp. $C_k$ and $D_k$, $E_6$ and $F_4$.

The section \ref{sectioneqzetafunc} recalls the definition of the tools we are going to use to do so : the equivariant zeta functions. Each of the sections \ref{compAB}, \ref{compCD} and \ref{compEF} is devoted to the comparison of the invariant germs of a specific couple of families ($A_k$ and $B_k$, $C_k$ and $D_k$, and finally $E_6$ and $F_4$). We proceed as follows. We begin by computing the first coefficients of the equivariant zeta functions (that is the coefficients of degree strictly smaller than the degree of the germs) in order to extract first cases of non-G-blow-Nash equivalence. Reducing our study to the remaining cases, we then compute the coefficient of degree equal to the degree of the germs. Finally, for the cases for which it is not sufficient, we compare the last terms of the respective equivariant zeta functions.

These comparisons lead to interesting examples of computations of equivariant virtual Poincar\'e series. The first one, to which is devoted section \ref{sectionypq}, is the computation of the equivariant virtual Poincar\'e series of the fibers over $0$, $-1$ and $+1$ of the quadratic forms $Q_{p,q}(y) := \sum_{i=1}^p y_i^2 - \sum_{j = 1}^q y_{p+j}^2$, equipped with four different actions of $G$.
\\
  
{\bf Acknowledgements.} The author wishes to thank J.-B. Campesato, G. Fichou, T. Fukui, A. Parusi\'nski, G. Rond and M. Shiota for useful discussions and comments.

\section{Equivariant Nash classification of invariant simple Nash germs} 

Consider the affine space $\mathbb{R}^n$ with coordinates $(x_1, \ldots, x_n)$. We denote by $s$ the involution of $\mathbb{R}^n$ changing the sign of the first coordinate $x_1$ : 
$$s : \begin{array}{ccc}
\mathbb{R}^n & \rightarrow & \mathbb{R}^n \\
(x_1,x_2, \ldots , x_n) & \mapsto & (- x_1,x_2, \ldots , x_n) 
\end{array}
$$
This equips $\mathbb{R}^n$ with a linear action of the group $G = \{id_{\mathbb{R}^n}, s\}$.

In this paper, a function germ $f : (\mathbb{R}^n, 0) \rightarrow (\mathbb{R},0)$ will be said invariant if $f$ is invariant under right composition with $s$, that is if $f$ is the germ of an equivariant function (we equip $\mathbb{R}$ with the trivial action of $G$).

In \cite{Arn} and \cite{Gor} is given the classification of invariant simple real analytic germs $(\mathbb{R}^n, 0) \rightarrow (\mathbb{R},0)$ with respect to equivariant analytic equivalence, that is right equivalence via an equivariant analytic diffeomorphism $(\mathbb{R}^n, 0) \rightarrow (\mathbb{R}^n,0)$ :

\begin{theo}[\cite{Arn}, \cite{Gor}] \label{arngor} An invariant simple real analytic function germ $(\mathbb{R}^n,0) \rightarrow (\mathbb{R},0)$ is equivariantly analytically equivalent to one and only one invariant germ of the following list :
$$\begin{array}{ll}
 A_k, k \geq 0 : \pm x_1^2 \pm x_2^{k+1} + Q, & E_6 : \pm x_1^2 + x_2^3 \pm x_3^4 + Q, \\
 B_k, k \geq 2 : \pm x_1^{2k} \pm x_2^2 + Q, &  E_7 : \pm x_1^2 + x_2^3 + x_2 x_3^3 + Q, \\
 C_k, k \geq 3 : x_1^2 x_2 \pm x_2^k + Q, & E_8 : \pm x_1^2 + x_2^3 + x_3^5 + Q, \\
 D_k, k \geq 4 : \pm x_1^2 + x_2^2 x_3 \pm x_3^{k-1} + Q, & F_4 : \pm x_1^4 + x_2^3 + Q, 
\end{array}$$   
where 
$Q = \pm x_s^2 \pm \cdots \pm x_n^2$, with $s=4$ for singularities $D_k$ and $E_k$, and $s = 3$ in the other cases.
\end{theo}

\begin{rem} If we forget the action of the involution $s$ on $\mathbb{R}^n$, notice that the families $A_k$ and $B_k$, resp. $C_k$ and $D_k$, $E_6$ and $F_4$, $E_7$, $E_8$, of Theorem \ref{arngor} are singularities $A$, resp. $D$, $E_6$, $E_7$, $E_8$.
\end{rem}

In this paper, we are interested in the classification of invariant Nash germs $(\mathbb{R}^n,0) \rightarrow (\mathbb{R},0)$, that is germs of equivariant analytic functions with semialgebraic graph. Recall (see for instance \cite{BCR} Corollary 8.1.6) that a Nash germ can be considered as an algebraic power series, via its Taylor series. The above classification is also valid for invariant simple Nash germs with respect to equivariant Nash equivalence, that is right equivalence via an equivariant Nash diffeomorphism $(\mathbb{R}^n, 0) \rightarrow (\mathbb{R}^n,0)$, according to the following proposition :

\begin{prop} \label{equivnasheqiffequivaneq} Let $f, h : (\mathbb{R}^n,0) \rightarrow (\mathbb{R},0)$ be two invariant Nash germs. Then $f$ and $h$ are equivariantly Nash equivalent if and only if they are equivariantly analytically equivalent.
\end{prop}

This property is a particular case of the following result :

\begin{theo} \label{equivnash} Let $G$ be a reductive algebraic group acting linearly on $\mathbb{R}^n$ and $\mathbb{R}^p$. Consider two equivariant Nash germs $f : (\mathbb{R}^n,0) \rightarrow (\mathbb{R}^p,0)$ and $h : (\mathbb{R}^n,0) \rightarrow (\mathbb{R}^p,0)$. If $f$ and $h$ are equivariantly analytically equivalent, then they are equivariantly Nash equivalent.
\end{theo}

\begin{rem}
\begin{itemize}
	\item Since a Nash diffeomorphism is in particular analytic, the converse is obviously true.
	\item Any finite group is reductive.
\end{itemize}
\end{rem}

\begin{proof}[Proof (of Theorem \ref{equivnash})] Suppose there exists an equivariant analytic diffeomorphism $\phi : (\mathbb{R}^n,0) \rightarrow (\mathbb{R}^n,0)$ such that $f \circ \phi = h$. Denote $F(x,y) := f(y) - h(x)$ for $x,y \in \mathbb{R}^n$. Then $F : (\mathbb{R}^{2n}, 0) \rightarrow (\mathbb{R}^p,0)$ is a Nash germ and can be considered as an algebraic power series in $\mathbb{R}_{alg}[[x,y]]^p$, and $\phi(x)$ as an equivariant convergent power series in $\mathbb{R}\{x,y\}$ such that $F(x,\phi(x)) = 0$.

Therefore, by Theorem A of \cite{BMinv} and Example 11.3 of \cite{Rond}, we can approximate $\phi(x)$ by an equivariant algebraic power series $\widetilde{\phi}(x)$ such that $F(x,\widetilde{\phi}(x)) = 0$, and we do the approximation closely enough so that $\widetilde{\phi}(x)$ remains a diffeomorphism. As a consequence, $\widetilde{\phi} : (\mathbb{R}^n,0) \rightarrow (\mathbb{R}^n,0)$ is an equivariant Nash diffeomorphism such that $f \circ \widetilde{\phi} = h$.
\end{proof}

\begin{rem} Actually, Theorem A of \cite{BMinv} is about approximation of equivariant formal solutions of polynomial equations by equivariant algebraic power series but it is also true for algebraic power series equations. Indeed, following G. Rond's ideas, it is possible to reduce to the case of polynomial equations as in \cite{Art} Lemma 5.2 and \cite{CRS} Reduction (2) of the proof of Theorem 1.1, using arguments of the proof of Lemma 8.1 in \cite{Rondloc}, along with the fact that the morphism $\mathbb{R}[x,y]_{(x,y)} \rightarrow \mathbb{R}_{alg}[[x,y]]$ is faithfully flat by \cite{BCR} Corollary 8.7.16.
\end{rem}

\section{Equivariant blow-Nash equivalence} \label{sectioneqblnequiv}

Now, we want to study the classification of invariant simple Nash germs with respect to $G$-blow-Nash equivalence via an equivariant blow-Nash isomorphism : see \cite{Pri-EZF} for the definition of $G$-blow-Nash equivalence via an equivariant blow-Nash isomorphism.

First, we have the following : 

\begin{prop} An invariant simple Nash germ  $(\mathbb{R}^n,0) \rightarrow (\mathbb{R},0)$ is $G$-blow-Nash equivalent via an equivariant blow-Nash isomorphism to an invariant germ of the list of Theorem \ref{arngor}.
\end{prop}

\begin{proof} This comes from the fact that if $f$ and $h$ are equivariantly Nash equivalent invariant Nash germs $(\mathbb{R}^n,0) \rightarrow (\mathbb{R},0)$, then they are $G$-blow-Nash equivalent via an equivariant blow-Nash isomorphism.

Indeed, if $f^{-1}(0)$, resp. $h^{-1}(0)$, has only one irreducible component at $0 \in \mathbb{R}^n$, this is straightforward. If not, we perform a composition $\sigma_f : (M_f , \sigma_f^{-1}(0)) \rightarrow (\mathbb{R}^n,0)$, resp. $\sigma_h : (M_h , \sigma_h^{-1}(0)) \rightarrow (\mathbb{R}^n,0)$, of successive equivariant blowings-up along $G$-invariant smooth Nash centers such that 
\begin{itemize}
	\item the irreducible components of the strict transform of $f^{-1}(0)$ by $\sigma_f$, resp. of $h^{-1}(0)$ by $\sigma_h$, do not intersect,
	\item $f \circ \sigma_f$ and $jac~\sigma_f$, resp. $h \circ \sigma_h$ and $jac~\sigma_h$, have only normal crossings simultaneously,
	\item there exists a finite collection of $G$-invariant affine charts for $\sigma_f$, resp. for $\sigma_h$, such that, on each of these charts, the action of $G$ is of the form 
$$(x_1,x_2, \ldots, x_n) \mapsto (\epsilon_1 x_1, \epsilon_2 x_2, \ldots, \epsilon_n x_n),$$
where $\epsilon_i \in \{\pm 1\}$ (so that the action of $G$ on $M_f$, resp. on $M_h$, can be locally linearized on the normal crossings, in the sense of \cite{Pri-EZF}),
\item after each blowing-up, $f$ and $h$ remain equivariantly Nash equivalent. 
\end{itemize}

%(notice that two irreducible components of the strict transform exchanged by the action will be seen in the same $G$-invariant affine chart)

%Therefore, $f$ and $h$ are $G$-blow-Nash equivalent via an equivariant blow-Nash isomorphism.

\end{proof}

The second step will then consist in understanding the relations, with respect to $G$-blow-Nash equivalence via an equivariant blow-Nash isomorphism, between the invariant Nash germs of the list of Theorem \ref{arngor}.
\\

Equivariant blow-Nash equivalence (resp. equivariant blow-Nash equivalence via an equivariant blow-Nash isomorphism) is a particular case of the blow-Nash equivalence (resp. blow-Nash equivalence via a blow-Nash isomorphism) defined in \cite{GF-ZF}. In \cite{GF-BNT}, Fichou proved that the classification of simple Nash germs $(\mathbb{R}^n,0) \rightarrow (\mathbb{R}, 0)$ with respect to blow-Nash equivalence via a blow-Nash isomorphism is the same as Arnold's $ADE$-classification of real analytic germs with respect to right analytic equivalence.

As a consequence, the $A$, $D$, $E$-singularities, belonging to different blow-Nash classes, cannot be $G$-blow-Nash-equivalent via an equivariant blow-Nash isomorphism either. We are then reduced to try to distinguish the invariant germs of the families $A_k$ and $B_k$, resp. $C_k$ and $D_k$, $E_6$ and $F_4$.

For this purpose, we will use the equivariant zeta functions defined in \cite{Pri-EZF}, which are invariants for equivariant blow-Nash equivalence via an equivariant blow-Nash isomorphism. 

\section{Equivariant zeta functions} \label{sectioneqzetafunc}

Let $f : (\mathbb{R}^n,0) \rightarrow (\mathbb{R},0)$ be an invariant Nash germ. We recall the definition given in \cite{Pri-EZF} of the equivariant zeta functions of $f$.

Denote $\mathcal{L} := \{ \gamma : (\mathbb{R}, 0) \rightarrow (\mathbb{R}^n , 0)~|~\gamma(t) = a_1 t + a_2 t^2 + \ldots, a_i \in \mathbb{R}^n \}$ the space of formal arcs at the origin of $\mathbb{R}^n$. The action of $G$ on $\mathbb{R}^n$ induces naturally an action of $G$ on $\mathcal{L}$, by left composition with $s$. For $m \in \mathbb{N} \setminus \{0\}$, the space 
$$\mathcal{L}_m := \{ \gamma : (\mathbb{R}, 0) \rightarrow (\mathbb{R}^n , 0)~|~\gamma(t) = a_1 t + a_2 t^2 + \ldots + a_m t^m\}$$
of arcs truncated at order $m+1$ is globally stable under this action, as well as the spaces 
$$A_m(f) := \{\gamma \in \mathcal{L}_m~|~f \circ \gamma(t) = c t^m + \ldots, c \neq 0 \},$$ 
$$A_m^+(f) := \{\gamma \in \mathcal{L}_m~|~f \circ \gamma(t) = + t^m + \ldots \} \mbox{ ~and~  } A_m^-(f) := \{\gamma \in \mathcal{L}_m~|~f \circ \gamma(t) = - t^m + \ldots \}.$$

These latter sets are Zariski constructible sets equipped with an algebraic action of $G$ and we define 
$$Z_f^G(u,T) :=  \sum_{m \geq 1} \beta^G(A_m(f)) u^{-mn} T^m \in \mathbb{Z}[u][[u^{-1}]][[T]]$$
and
$$Z_f^{G,\pm}(u,T) :=  \sum_{m \geq 1} \beta^G(A^{\pm}_m(f)) u^{-mn} T^m \in \mathbb{Z}[u][[u^{-1}]][[T]],$$
respectively the naive equivariant zeta function and the equivariant zeta functions with sign of~$f$.

Here, $\beta^G(\cdot)$ denotes the equivariant virtual Poincar\'e series on $G$-$\mathcal{AS}$ sets of \cite{GF} : it is an additive invariant with respect to equivariant isomorphisms, with values in $\mathbb{Z}[[u]]$, such that, if $X$ is a compact nonsingular $G$-$\mathcal{AS}$ set, $\beta^G(X) = \sum_{i \in \mathbb{Z}} dim_{\mathbb{Z}_2} H_i(X ; G) \, u^i$, where $H_*(X ; G)$ denotes the equivariant Borel-Moore homology of $X$ with coefficients in $\mathbb{Z}_2$ defined in \cite{VH}.

\begin{rem} \label{equivpoinc}
\begin{itemize}
	\item By an isomorphism between arc-symmetric sets is meant a birational map containing the arc-symmetric sets in its support.
	\item The equivariant virtual Poincar\'e series of a point is $\frac{u}{u-1}$, the equivariant virtual Poincar\'e series of two fixed points is $2 \frac{u}{u-1}$ and the equivariant virtual Poincar\'e series of two points exchanged by $G$ is $1$ : see \cite{GF} Example 3.12.
	\item If $S^d$ denotes the unit sphere in $\mathbb{R}^d$ then
	$$\beta^G(S^d) = \begin{cases} 1 + u + \ldots + u^d \mbox{ if $G$ acts via the central symmetry of $\mathbb{R}^d$,} \\
	2 \frac{u}{u-1} + u + \ldots + u^d \mbox{ if $G$ acts with a fixed point}
	\end{cases}$$
(see \cite{GF} Example 3.12).
	\item If $\mathbb{R}^d$ is equipped with any algebraic action of $G$, then $\beta^G(\mathbb{R}^d) = \frac{u^{d+1}}{u-1}$ : see \cite{GF} Example 3.12. 
	\item If $X$ is a $G$-$\mathcal{AS}$ set and if the affine space $\mathbb{R}^d$ is equipped with any algebraic action of $G$, then $\beta^G(X \times \mathbb{R}^d) = u^d \beta^G(X)$ (the product $X \times \mathbb{R}^d$ is equipped with the diagonal action of $G$) : see \cite{GF} Proposition 3.13.
	\item If $X$ is a $G$-$\mathcal{AS}$ set and if the affine line $\mathbb{R}$ is equipped with an algebraic action of $G$ stabilizing $0$, then $\beta^G(X \times (\mathbb{R}^*)^d) = (u-1)^d \beta^G(X)$ : see \cite{Pri-EZF} Lemma 3.9.
	\item If $X$ is a $G$-$\mathcal{AS}$, then the coefficients of the negative powers of $u$ in $\beta^G(X)$ are all equal to $\sum_{i \geq 0} \beta_i(X^G)$, where $X^G$ is the fixed point set of $X$ and $\beta_i(\cdot)$ denotes the $i^{th}$ virtual Betti number (\cite{MCP-VB}) : see \cite{GF} Proposition 4.5.
\end{itemize}
\end{rem}

\begin{theo}[Theorem 4.1 of \cite{Pri-EZF}] \label{eqzfinv} Let $f, h : (\mathbb{R}^n, 0) \rightarrow (\mathbb{R}, 0)$ be two invariant Nash germs. If $f$ and $h$ are $G$-blow-Nash equivalent via an equivariant blow-Nash isomorphism, then $Z_f^{G}(u,T) = Z_h^{G}(u,T)$ and $Z_f^{G,\pm}(u,T) = Z_h^{G,\pm}(u,T)$.
\end{theo}  

\begin{rem} In the rest of the paper, we will simply talk about equivariant blow-Nash equivalence to refer to equivariant blow-Nash equivalence via an equivariant blow-Nash isomorphism. 
\end{rem}

In the next parts of the paper, we are then going to use the equivariant zeta functions in order to try to distinguish the families $A_k$ and $B_k$, resp. $C_k$ and $D_k$, $E_6$ and $F_4$, with respect to $G$-blow-Nash equivalence. More precisely, we will show that, in some cases, some terms of the respective equivariant zeta functions of the considered germs are different. 

On the other hand, we will prove that, in some other cases, the equivariant zeta functions are equal.
\\

Before this, in the following section, we compute equivariant virtual Poincar\'e series associated to the quadratic form 
$$Q_{p,q}(y) := \sum_{i=1}^p y_i^2 - \sum_{j = 1}^q y_{p+j}^2,$$
where $p,q \in \mathbb{N}$, $(y_1, \ldots, y_{p+q}) \in \mathbb{R}^{p+q}$. More precisely, we compute the equivariant virtual Poincar\'e series of the algebraic sets
$$Y_{p,q} := \{Q_{p,q} = 0\} \mbox{~~and~} Y_{p,q}^{\xi} := \{Q_{p,q} = \xi\},$$
for $\xi = \pm 1$, in the cases where the action of $G$ on $\mathbb{R}^{p+q}$ is given by 
\begin{enumerate}
\item \label{acty1} $(y_1, \ldots, y_p, y_{p+1}, \ldots , y_{p+q}) \mapsto (- y_1, \ldots, y_p, y_{p+1}, \ldots , y_{p+q})$,
\item \label{actyp1} $(y_1, \ldots, y_p, y_{p+1}, \ldots , y_{p+q}) \mapsto (y_1, \ldots, y_p, - y_{p+1}, \ldots , y_{p+q})$,
\item \label{acttout} $(y_1, \ldots, y_p, y_{p+1}, \ldots , y_{p+q}) \mapsto (- y_1, \ldots, - y_p, -y_{p+1}, \ldots , -y_{p+q})$,
\item \label{triv} or $(y_1, \ldots, y_p, y_{p+1}, \ldots , y_{p+q}) \mapsto (y_1, \ldots, y_p, y_{p+1}, \ldots , y_{p+q})$.
\end{enumerate}
This will reveal useful in the comparisons of the equivariant zeta functions.

\section{Computation of $\beta^G(Y_{p,q})$ and $\beta^G(Y_{p,q}^{\xi})$} \label{sectionypq}

Suppose that $p \leq q$. We have the following result :

\begin{prop} \label{ypq}
\begin{enumerate}
	\item If $0 < p < q$, then
$$\beta^G(Y_{p,q}) = 
\begin{cases}
\frac{u^{p+q} - u^q + u^{p-1}}{u-1} & \mbox{ in the case n\textsuperscript{o}\ref{acty1},} \\
\frac{u^{p+q} - u^q + u^{p+1}}{u-1} & \mbox{ in the three other cases. }
\end{cases}$$
	\item If $p = q \neq 0$, then
$$\beta^G(Y_{p,q}) = 
\begin{cases}
\frac{u^{2p}-u^p + u^{p-1}}{u-1} & \mbox{ in the cases n\textsuperscript{o}\ref{acty1} and n\textsuperscript{o}\ref{actyp1}, } \\
\frac{u^{2p}  - u^p + u^{p+1}}{u-1} & \mbox{ in the two other cases. }
\end{cases}$$
	\item If $p = 0$, then
$$\beta^G(Y_{p,q}) = \frac{u}{u-1}.$$			 
\end{enumerate}
\end{prop}

\begin{rem} If $q \leq p$, just exchange the roles of $p$ and $q$ along with the actions of the cases n\textsuperscript{o}\ref{acty1} and n\textsuperscript{o}\ref{actyp1}.
\end{rem}

\begin{proof}[Proof (of Proposition \ref{ypq})]
If $p = 0$, then $Y_{p,q} = \{0\}$ and $\beta^G(Y_{p,q}) = \frac{u}{u-1}$ by remark \ref{equivpoinc}.
\\

If $0 < p < q$, as in \cite{GF-TC} Proof of Proposition 2.1 and \cite{GF-BNT} Proof of Lemma 3.1, we apply the equivariant change of variables $u_i = y_i + y_{i+p}$, $v_i = y_i - y_{i+p}$ for $i = 2, \ldots, p$ and the equation $Q_{p,q} = 0$ becomes
$$y_1^2 - y_{p+1}^2 + \sum_{i=2}^p u_i v_i - \sum_{j = 2p+1}^{p+q} y_j^2 = 0,$$
the action of $G$ on the new coordinates $u_i$, $v_i$ being trivial in the cases n\textsuperscript{o}\ref{acty1}, n\textsuperscript{o}\ref{actyp1} and n\textsuperscript{o}\ref{triv}, and changing their signs in the case n\textsuperscript{o} \ref{acttout}. 

As in \cite{GF-TC} and \cite{GF-BNT}, we write, by additivity of the equivariant virtual Poincar\'e series,
$$\beta^G(Y_{p,q}) = \beta^G(Y_{p,q} \cap \{u_2 \neq 0\}) + \beta^G(Y_{p,q} \cap \{u_2 = 0\}).$$
Because, if $u_2 \neq 0$, the coordinate $v_2$ is determined via an equivariant isomorphism by $u_2$ and the other variables which are free, we have $\beta^G(Y_{p,q} \cap \{u_2 \neq 0\}) = \beta^G(\mathbb{R}^* \times \mathbb{R}^{p+q-2}) = (u-1) \frac{u^{p+q-1}}{u-1}$ (see remark \ref{equivpoinc}). Furthermore, the equation describing $Y_{p,q} \cap \{u_2 = 0\}$ is 
$$y_1^2 - y_{p+1}^2 + \sum_{i=3}^p u_i v_i - \sum_{j = 2p+1}^{p+q} y_j^2 = 0$$
(notice that the variable $v_2$ is then free) and, by an induction, we obtain
\begin{eqnarray*}
\beta^G(Y_{p,q}) & = & \sum_{i = 2}^{p} u^{p+q+1-i} + u^{p-1} \beta^G\left(\left\{y_1^2 - y_{p+1}^2 - \sum_{j = 2p+1}^{p+q} y_j^2 = 0\right\}\right) \\
& = & u^{q+1} \frac{u^{p-1} - 1 }{u-1} + u^{p-1} \beta^G\left(\left\{y_1^2 - y_{p+1}^2 - \sum_{j = 2p+1}^{p+q} y_j^2 = 0\right\}\right).
\end{eqnarray*}

Now, in order to compute $\beta^G\left(\left\{y_1^2 - y_{p+1}^2 - \sum_{j = 2p+1}^{p+q} y_j^2 = 0\right\}\right)$, we equivariantly blow up the latter algebraic set at the origin of $\mathbb{R}^{q-p+2}$ : in the chart $y_1 = w$, $y_i = w z_i$, $i = p+1, 2p+1, \ldots, p+q$, the blown-up variety is defined by
$$w^2\left(1 - z_{p+1}^2 - \sum_{j = 2p+1}^{p+q} z_j^2\right) = 0,$$
the action of $G$ being given by   
\begin{itemize}
	\item $(w, z_{p+1}, z_{2p+1}, \ldots , z_{p+q}) \mapsto (-w, -z_{p+1}, -z_{2p+1}, \ldots , -z_{p+q})$ in the case n\textsuperscript{o}\ref{acty1},
	\item $(w, z_{p+1}, z_{2p+1}, \ldots , z_{p+q}) \mapsto (w, -z_{p+1}, z_{2p+1}, \ldots , z_{p+q})$ in the case n\textsuperscript{o}\ref{actyp1},
	\item $(w, z_{p+1}, z_{2p+1}, \ldots , z_{p+q}) \mapsto (- w, z_{p+1}, z_{2p+1}, \ldots , z_{p+q})$ in the case n\textsuperscript{o} \ref{acttout},
	\item  $(w, z_{p+1}, z_{2p+1}, \ldots , z_{p+q}) \mapsto (w, z_{p+1}, z_{2p+1}, \ldots , z_{p+q})$ in the case n\textsuperscript{o}\ref{triv}.
\end{itemize}
We have
\begin{eqnarray*}
\beta^G\left(\left\{y_1^2 - y_{p+1}^2 - \sum_{j = 2p+1}^{p+q} y_j^2 = 0\right\} \setminus \{0\}\right) & = &  \beta^G\left(\left\{1 - z_{p+1}^2 - \sum_{j = 2p+1}^{p+q} z_j^2 = 0\right\} \setminus \{w=0\}\right) \\
& = & \beta^G(\mathbb{R}^* \times S^{q-p}) \\
& = & (u-1) \beta^G(S^{q-p}) 
\end{eqnarray*} 

Finally, since the action of $G$ on the sphere $S^{q-p}$ is the central symmetry in the case n\textsuperscript{o}\ref{acty1} and admits a fixed point in the three other cases, we have
$$\beta^G(S^{q-p})  = \begin{cases} \frac{u^{q-p+1} -1}{u-1}  \mbox{ in the case n\textsuperscript{o}\ref{acty1},} \\
\frac{u^{q-p+1} + u}{u-1} \mbox{ in the three other cases}
\end{cases}$$
(see remark \ref{equivpoinc}). Using the additivity relation
$$\beta^G\left(\left\{y_1^2 - y_{p+1}^2 - \sum_{j = 2p+1}^{p+q} y_j^2 = 0\right\} \setminus \{0\}\right) = \beta^G\left(\left\{y_1^2 - y_{p+1}^2 - \sum_{j = 2p+1}^{p+q} y_j^2 = 0\right\} \right) - \beta^G(\{0\})$$
and the equality $\beta^G(\{0\}) = \frac{u}{u-1}$, we obtain the desired result.

\vspace{0.8cm}

If $p = q \in \mathbb{N} \setminus \{0 \, ; 1\}$, we do as before in order to obtain the equality
$$\beta^G(Y_{p,q}) =  u^{p+1} \frac{u^{p-1} - 1 }{u-1} + u^{p-1} \beta^G\left(\left\{y_1^2 - y_{p+1}^2 = 0\right\}\right)$$
(notice that the quantity $\beta^G(\{y_1^2 - y_{p+1}^2 = 0\})$ is the same in the cases n\textsuperscript{o}\ref{acty1} and n\textsuperscript{o}\ref{actyp1}). Now, as above, we equivariantly blow up at the origin of $\mathbb{R}^2$ and look in the chart $y_1 = u_1$, $y_{p+1} = u_1 v_{p+1}$~: the blown-up variety is given by the equation
$$u_1^2(1 - v_{p+1}^2) = 0$$
and the action of $G$ is given by 
\begin{itemize}
	\item $(u_1,v_{p+1}) \mapsto (-u_1, -v_{p+1})$ in the case n\textsuperscript{o}\ref{acty1}, 
	\item $(u_1,v_{p+1}) \mapsto (u_1, -v_{p+1})$ in the case n\textsuperscript{o}\ref{actyp1},
	\item $(u_1,v_{p+1}) \mapsto (- u_1,  v_{p+1})$ in the case n\textsuperscript{o}\ref{acttout}, 
	\item $(u_1,v_{p+1}) \mapsto ( u_1,  v_{p+1})$ in the case n\textsuperscript{o}\ref{triv}. 
\end{itemize}	
As a consequence, 
$$\beta^G(\{1- v^2_{p+1} = 0\}) = \begin{cases} 1 \mbox{ in the cases n\textsuperscript{o}\ref{acty1} and n\textsuperscript{o}\ref{actyp1},} \\
2 \frac{u}{u-1} \mbox{ in the two other cases}
\end{cases}$$
(see remark \ref{equivpoinc}) and we obtain the desired result.
\\

If $p = q = 1$, we have $\beta^G(Y_{p,q}) = \beta^G\left(\left\{y_1^2 - y_{p+1}^2 = 0\right\}\right)$ and we can use the previous computation.
\end{proof}

This proposition can be used to compute the quantities $\beta^G(Y_{p,q}^{\xi})$. Indeed :

\begin{prop} \label{yxi} We have
$$\beta^G(Y_{p,q}^{+1}) = \frac{1}{u-1} \left(\beta^G(Y_{p,q+1}) - \beta^G(Y_{p,q})\right)$$
and
$$\beta^G(Y_{p,q}^{-1}) = \frac{1}{u-1} \left(\beta^G(Y_{p+1,q}) - \beta^G(Y_{p,q})\right)$$
\end{prop}

\begin{rem} We have the same equalities if $q \leq p$.
\end{rem}

\begin{proof}[Proof (of Proposition \ref{yxi})]
We show the first equality, the proof of the second one being similar. 

Denote $Z_{p,q}$ the projective algebraic set
$$\left\{[Y_1: \ldots : Y_{p+q}] \in \mathbb{P}^{p+q-1}(\mathbb{R})~|~\sum_{i=1}^p Y_i^2 - \sum_{j=1}^q Y_{p+j}^2 = 0\right\}$$

As in \cite{GF-CI} Proof of Corollary 2.5, we can equivariantly compactify $Y_{p,q}^{+1}$ into the projective algebraic set $Z_{p,q+1}$, the part at infinity being equivariantly isomorphic to $Z_{p,q}$ (we equip $\mathbb{P}^{p+q}(\mathbb{R})$ and $ \mathbb{P}^{p+q-1}(\mathbb{R})$ with the actions of $G$ naturally induced from the considered action on the variables of $\mathbb{R}^{p+q}$).
\\

Now, we compute $\beta^G(Z_{p,q})$, using, as in \cite{GF-CI} Proof of Proposition 2.1, the fact that the projection 
$$p : \begin{array}{ccc}
         Y_{p,q} \setminus \{0\} & \rightarrow & Z_{p,q} \\
         (y_1, \ldots, y_{p+q}) & \mapsto & [y_1 : \ldots : y_{p+q}]
         \end{array}$$
is a piecewise algebraically trivial fibration, compatible with the respective considered actions of $G$. More precisely, we can cover $Z_{p,q}$ by the globally $G$-invariant open subvarieties 
$$U_i := Z_{p,q} \cap \{Y_i \neq 0\}, i \in \{1, \ldots, p+q\},$$
and, for each $i \in \{1, \ldots, p+q\}$, we can define the isomorphism
$$\varphi_i : \begin{array}{ccc}
p^{-1}(U_i) = Y_{p,q} \cap \{y_i \neq 0\} & \rightarrow & U_i \times \mathbb{R}^* \\
(y_1, \ldots, y_i , \ldots , y_{p+q}) & \mapsto & ([y_1 : \ldots : y_i : \ldots : y_{p+q}], y_i) 
\end{array}$$
For $i \in \{1, \ldots, p+q\}$, if the sign of the coordinate $y_i$ is changed under the action of $G$, we equip $\mathbb{R^*}$ with the action of $G$ given by the involution $z \mapsto -z$. If $y_i$ remains unchanged under the action of $G$, we equip $\mathbb{R^*}$ with the trivial action of $G$. Furthermore equipping the product $U_i \times \mathbb{R}^*$ with the diagonal action, this makes the isomorphism $\varphi_i$ equivariant.

%with inverse
%$$([y_1 : \ldots : y_i : \ldots : y_{p+q}], z) =\left(\left[\frac{y_1}{y_i} : \ldots : 1 : \ldots : \frac{y_n}{y_i}\right] , z\right) \mapsto \left(\frac{y_1}{y_i} \times z , \ldots, z, \ldots , \frac{y_{p+q}}{y_i} \times z\right).$$ 

By the additivity of the equivariant virtual Poincar\'e series, the quantity $\beta^G(Y_{p,q} \setminus \{0\})$ can be written as the alternated sum of the terms 
$$\sum_{J \subset \{1, \ldots, p+q\}, \, Card(J) = r} \beta^G\left(p^{-1}\left(\bigcap_{m \in J} U_m\right)\right), 1 \leq r \leq p+q,$$
and, via the equivariant isomorphisms $\varphi_i$, we have
$$\beta^G\left(p^{-1}\left(\bigcap_{m \in J} U_m\right)\right) = \beta^G\left(\left(\bigcap_{m \in J} U_m\right) \times \mathbb{R}^*\right) = (u-1) \beta^G\left(\bigcap_{m \in J} U_m \right).$$
As a consequence, once again thanks to the additivity of the equivariant virtual Poincar\'e series, 
$$\beta^G(Y_{p,q} \setminus \{0\}) = (u-1) \beta^G(Z_{p,q}).$$

Therefore,
\begin{eqnarray*}
\beta^G(Y_{p,q}^{+1}) & = & \beta^G(Z_{p,q+1}) - \beta^G(Z_{p,q}) \\
				  & = & \frac{1}{u-1} \left(\beta^G(Y_{p,q+1} \setminus \{0\}) - \beta^G(Y_{p,q} \setminus \{0\})\right) \\
				  & = & \frac{1}{u-1} \left(\beta^G(Y_{p,q+1}) - \beta^G(Y_{p,q})\right).
\end{eqnarray*}

\end{proof}

\section{The germs $A_k$ and $B_k$} \label{compAB}

In this section, we want to study the relations with respect to $G$-blow-Nash equivalence between the invariant germs of the families
$$f_k^{\epsilon_k}(x) := \pm x_1^2 + \epsilon_k x_2^{k+1} + Q \mbox{ ~~and~~} g_k^{\epsilon_k}(x) := \epsilon_k x_1^{2k} \pm x_2^2 + Q',$$
where $\epsilon_k \in \{-1 \, ; + 1\}$.

First, if any two invariant Nash germs are $G$-blow-Nash equivalent, they are in particular blow-Nash equivalent and then, according to \cite{GF-CI} Theorem 2.5, they have the same corank and index.

Therefore, if two germs $f_k^{\epsilon_k}$ and $f_l^{\epsilon_l}$ are $G$-blow-Nash equivalent, they have the same quadratic part up to permutation of the variables $x_1, x_3, \ldots , x_n$. Furthermore, we know, by \cite{GF-BNT} Proposition 3.4, that $k = l$ and, if $k = l$ is odd, that $\epsilon_k = \epsilon_l$. If $k$ is even, $f_k^{+1}(x_1,x_2, x_3, \ldots, x_n) = f_k^{-1}(x_1, -x_2,x_3, \ldots, x_n)$ and the (linear) change of variables is equivariant with respect to the involution $s$ on $\mathbb{R}^n$ : $f_k^{+1}$ and $f_k^{-1}$ are then $G$-Nash equivalent, in particular $G$-blow-Nash equivalent.

As a conclusion, inside the family $A_k$, we are reduced to try to distinguish the germs
$$f_k^{\epsilon_k, +}(x) := + x_1^2 + \epsilon_k x_2^{k+1} + Q \mbox{   and   } f_k^{\epsilon_k, -}(x) := - x_1^2 + \epsilon_k x_2^{k+1} + Q'$$
where $\epsilon_k \in \{-1 \, ; + 1\}$ and $+ x_1^2 + Q$ and $- x_1^2 + Q'$ are the same quadratic part up to permutation of the variables $x_1, x_3, \ldots , x_n$.
\\

Similarly, if two germs $g_k^{\epsilon_k}$ and $g_l^{\epsilon_l}$ are $G$-blow-Nash equivalent, they have the same quadratic part, up to permutation of the variables $x_2, \ldots , x_n$, and $k = l$ and $\epsilon_k = \epsilon_l$.
\\

Finally, if two germs $f_k^{\epsilon_k}$ and $g_{k'}^{\epsilon_{k'}}$ are blow-Nash equivalent, then $k = 2k'-1$ and $\epsilon_{k} = \epsilon_{k'}$, and furthermore $\pm x_1^2 + Q$ and $\pm x_2^2 + Q'$ are the same quadratic part up to permutation of all variables. Consequently, it remains to look at the relation between the germs
$$f_{2k-1} = \epsilon x_2^{2k} + \eta x_1^2 + Q \mbox{   and   } g_k = \epsilon x_1^{2k} + \eta' x_2^2 + Q'$$
where $\epsilon, \eta, \eta' \in \{1,-1\}$ and $\eta x_1^2 + Q = \eta' x_2^2 + Q'$ up to permutation of all variables. 
\\

In the following parts of this section, we will compute some terms of the equivariant zeta functions of $f_k$ and $g_k$. In virtue of theorem \ref{eqzfinv}, this will allow us to make further distinctions inside each of the above couples of germs in some cases.  

\subsection{Computation of the first terms of the equivariant zeta functions} \label{firsttermsab} 

If $h$ is an invariant Nash germ $(\mathbb{R}^n, 0) \rightarrow (\mathbb{R},0)$, recall that, for $m \geq 1$,
\begin{eqnarray*}
A_{m}(h) & = & \{\gamma(t) = a_1 t + \cdots + a_{m} t^{m} \in \mathcal{L}_{m}~|~h \circ \gamma(t) = c t^{m} + \cdots , c\neq 0\} \\
	       & = & \{\gamma \in \mathcal{L}_{m}~|~ h \circ \gamma(t) = ct^{m} + \cdots, c \in \mathbb{R} \} \setminus \{\gamma \in \mathcal{L}_{m}~|~ h \circ \gamma(t) = 0 \times t^{m} + \cdots \}
\end{eqnarray*}
Since $h$ is an invariant germ, the latter sets are both globally stable under the action of $G$ on $\mathcal{L}_m$ and, by the additivity of the equivariant virtual Poincar\'e series, the quantity $\beta^G(A_{m}(h))$ is equal to the difference $\beta^G({}^{0}\!A_{m}(h)) - \beta^G(A^0_{m}(h))$, where
$${}^0\!A_{m}(h) := \{\gamma \in \mathcal{L}_{m}~|~ h \circ \gamma(t) = ct^{m} + \cdots, c \in \mathbb{R} \} \mbox{   and   } A_{m}^0(h) := \{\gamma \in \mathcal{L}_{m}~|~ h \circ \gamma(t) = 0 \times t^{m} + \cdots \}.$$
\\

Fix $k \geq 0$ and consider the invariant germ $f_k^{\epsilon, \eta}(x_1, \ldots,x_n) =  \eta x_1^2 + \epsilon x_2^{k+1} + Q$. We denote $x_2 = x$ and $\eta x_1^2 + Q = Q_{p,q} =  \sum_{i=1}^p y_i^2 - \sum_{j=1}^q y_{p+j}^2$ in such a way that $G$ acts on the renamed coordinates via the involution n\textsuperscript{o}\ref{acty1} or n\textsuperscript{o}\ref{actyp1} depending on the sign of $\eta$.
\\

We first compute $\beta^G(A^0_{m}(f_k^{\epsilon, \eta}))$ for $m < k+1$. Notice that the set $A_{1}(f_1^{\epsilon, \eta})$ is empty and, consequently, $\beta^G(A_{1}(f_1^{\epsilon, \eta})) = 0$.

\begin{prop} \label{amf} Suppose $k \geq 2$ and $m < k+1$.
		\begin{enumerate}
			\item If $pq = 0$, then
$$\beta^G(A^0_{m}(f_k^{\epsilon, \eta})) = \begin{cases} \frac{u^{m + (r+1)(p+q) + 1}}{u-1} \mbox{ if $m = 2r+1$,} \\
\frac{u^{m + r(p+q) + 1}}{u-1} \mbox{ if $m = 2r$.}
\end{cases}$$ 
			\item If $(p,q) = (1,1)$, then
$$\beta^G(A^0_{m}(f_k^{\epsilon, \eta})) = \begin{cases} r u^{2m} \beta^G(Y_{1,1} \setminus \{0\}) + \frac{u^{4(r+1)}}{u-1} \mbox{ if $m = 2r+1$,} \\
(r-1) u^{2m} \beta^G(Y_{1,1} \setminus \{0\}) + u^{4r} \beta^G(Y_{1,1})  \mbox{ if $m = 2r$,}
\end{cases}$$
			\item If $pq \neq 0$ and $(p,q) \neq (1,1)$, then
$$\beta^G(A^0_{m}(f_k^{\epsilon, \eta})) = \begin{cases} u^{m} u^{(r+1)(p+q)-1} \frac{u^{r(p+q-2)} - 1}{u^{p+q-2} - 1} \beta^G(Y_{p,q} \setminus \{0\})  + \frac{u^{(r+1)(2+p+q)}}{u-1} \mbox{ if $m = 2r+1$,} \\
 u^m u^{(r+1)(p+q) - 2}  \frac{u^{(r-1)(p+q-2)} - 1}{u^{p+q-2} - 1} \beta^G(Y_{p,q} \setminus \{0\}) + u^{r(2+ p+q)} \beta^G(Y_{p,q}) \mbox{ if $m = 2r$.}
\end{cases}$$
		\end{enumerate}
\end{prop}

\begin{proof} We follow the computation steps of \cite{GF-BNT}, keeping the track of the action of $G$ in our context. 

An arc $\gamma$ of $\mathcal{L}_{m}$ can be written as 
\begin{eqnarray*} \gamma(t) & = & (a_1 t + \cdots + a_{m} t^{m}, c_1^1 t + \cdots + c_{m}^1 t^{m}, \ldots, c_1^{p+q} t + \cdots + c_{m}^{p+q} t^{m}) \\
& = & \left( \begin{array}{c} a_1 \\ c_1^1 \\ \vdots \\ c_1^{p+q} \end{array} \right) t + \cdots + \left( \begin{array}{c} a_{m} \\ c_{m}^1 \\ \vdots \\ c_{m}^{p+q} \end{array} \right) t^{m} = \left( \begin{array}{c} a_1 \\ c_1 \end{array} \right) t + \cdots + \left( \begin{array}{c} a_{m} \\ c_{m}\end{array} \right) t^{m}
\end{eqnarray*}
if $c_i := (c_i^1, \ldots, c_i^{p+q})$. The group $G$ acts on $\mathcal{L}_{m}$ changing the sign of the variables $c_i^1$, resp. $c_i^{p+1}$, in the case n\textsuperscript{o}\ref{acty1}, resp. n\textsuperscript{o}\ref{actyp1}. 
\\
%the first coordinate function of $\gamma$ and notice that $\mathcal{L}_{2k}$ is equivariantly isomorphic to $R^{2k \times (p+q+1)}$ equipped with the diagonal action of $G$ on each variable. 

We begin with the case $pq \neq 0$, $(p,q) \neq (1,1)$ and $m = 2 r + 1$ odd. An arc $\gamma$ of $\mathcal{L}_{m}$ belongs to $A^0_{m}(f_k^{\epsilon, \eta})$ if and only if

$$\begin{cases}
Q_{p,q}(c_1) = 0, \\
\Phi_{p,q}(c_1, c_2) = 0, \\
\cdots \\
Q_{p,q}(c_r) + \sum_{t= 1}^{r-1} \Phi_{p,q}(c_t,c_{2r-t}) = 0, \\
\sum_{t= 1}^{r} \Phi_{p,q}(c_t,c_{2r+1-t}) = 0, \\
\end{cases}$$
where $\Phi_{p,q}$ is the function on $\mathbb{R}^{p+q} \times \mathbb{R}^{p+q}$ defined by $\Phi_{p,q}(u,v) = 2 \sum_{i=1}^p u_i v_i - 2 \sum_{j=1}^q u_{p+j} v_{p+j}$.

The first equality of the system means $c_1 \in Y_{p,q}$ by definition. Now, if $c_1^1 \neq 0$, the variables $c_2^1, \ldots, c_{2r}^1$ are determined by $c_1^1$ and the other (free) variables via an equivariant morphism. Therefore,
\begin{eqnarray*}
\beta^G(A^0_{m}(f_k^{\epsilon, \eta})) & = & \beta^G\left(A^0_{m}(f_k^{\epsilon,\eta}) \cap \{c_1^1 \neq 0\}\right) +  \beta^G\left(A^0_{m}(f_k^{\epsilon, \eta}) \cap \{c_1^1 = 0\}\right) \\
 & = & \beta^G\left(\left(Y_{p,q} \setminus (\{0\} \times Y_{p-1,q})\right) \times \mathbb{R}^{m + (m-1)(p+q-1)+1}\right) + \beta^G\left(A^0_{m}(f_k^{\epsilon, \eta}) \cap \{c_1^1 = 0\}\right) \\
 & = & u^{m + (m-1)(p+q-1)+1}  \beta^G\left(Y_{p,q} \setminus (\{0\} \times Y_{p-1,q})\right) + \beta^G\left(A^0_{m}(f_k^{\epsilon, \eta}) \cap \{c_1^1 = 0\}\right)
 \end{eqnarray*}
Next, we have
$$\beta^G\left(A^0_{m}(f_k^{\epsilon, \eta}) \cap \{c_1^1 = 0\}\right) = u^{m + (m-1)(p+q-1)+1}  \beta^G\left(Y_{p-1,q} \setminus (\{0\} \times Y_{p-2,q})\right) + \beta^G\left(A^0_{m}(f_k^{\epsilon, \eta}) \cap \{c_1^1 = c_1^2 = 0 \}\right)$$
and we obtain by induction
$$\beta^G(A^0_{m}(f_k^{\epsilon, \eta})) = u^{m + (m-1)(p+q-1)+1}  \beta^G(Y_{p,q} \setminus \{0\}) +  \beta^G\left(A^0_{m}(f_k^{\epsilon, \eta}) \cap \{c_1^1 = \ldots =  c_1^p = 0\}\right).$$
If $c_1^1 = \ldots =  c_1^p = 0$ then $c_1^{p+1} = \ldots = c_1^{p+q} = 0$ (since $Q_{p,q}(c_1) = 0$) and the other variables verify the system
$$\begin{cases}
Q_{p,q}(c_2) = 0, \\
\Phi_{p,q}(c_2, c_3) = 0, \\
\cdots \\
Q_{p,q}(c_r) + \sum_{t= 2}^{r-1} \Phi_{p,q}(c_t,c_{2r-t}) = 0, \\
\sum_{t= 2}^{r} \Phi_{p,q}(c_t,c_{2r+1-t}) = 0. \\
\end{cases}$$

Noticing that the vector $c_m$ as well as the variables $a_{m-1}$, $a_m$ are free and renaming the remaining variables, we have
$$\beta^G(A^0_{m}(f_k^{\epsilon, \eta})) = u^{m + (m-1)(p+q-1)+1}  \beta^G(Y_{p,q} \setminus \{0\}) + u^{2+(p+q)} \beta^G(A^0_{m-2}(f_k^{\epsilon, \eta}))$$
and, by an induction,
$$\beta^G(A^0_{m}(f_k^{\epsilon, \eta})) =  \beta^G(Y_{p,q} \setminus \{0\}) \sum_{t=0}^{r-1} u^{t(2+p+q)} u^{m-2t + (m-2t-1)(p+q-1)+1} + u^{(r-1)(2+ p+q)} \beta^G(A^0_{3}(f_k^{\epsilon, \eta}) \cap \{c_1 = 0\})$$
As a conclusion, since the system describing $A^0_{3}(f_k^{\epsilon, \eta}) \cap \{c_1 = 0\}$ is trivial, the variables $a_i$ as well as the vectors $c_2$ and $c_3$ are free and
$$\beta^G(A^0_{m}(f_k^{\epsilon, \eta})) = u^{m} u^{(r+1)(p+q)-1} \frac{u^{r(p+q-2)} - 1}{u^{p+q-2} - 1} \beta^G(Y_{p,q} \setminus \{0\}) + \frac{u^{(r+1)(2+p+q)}}{u-1}.$$
\\

If $m$ is even, $m = 2r$, the system describing $A^0_{m}(f_k^{\epsilon, \eta})$ is
$$\begin{cases}
Q_{p,q}(c_1) = 0, \\
\Phi_{p,q}(c_1, c_2) = 0, \\
\cdots \\
\sum_{t= 1}^{r-1} \Phi_{p,q}(c_t,c_{2r-1-t}) = 0, \\
Q_{p,q}(c_r) + \sum_{t= 1}^{r-1} \Phi_{p,q}(c_t,c_{2r-t}) = 0.
\end{cases}$$
  
Therefore, by similar computations, we obtain
$$\beta^G(A^0_{m}(f_k^{\epsilon, \eta})) =  \beta^G(Y_{p,q} \setminus \{0\}) \sum_{t=0}^{r-2} u^{t(2+p+q)} u^{m-2t + (m-2t-1)(p+q-1)+1} + u^{(r-1)(2+ p+q)} \beta^G(A^0_{2}(f_k^{\epsilon, \eta}))$$
Since $A^0_{2}(f_k^{\epsilon, \eta})$ is described by the equation $Q_{p,q}(c_1) = 0$, the vector $c_2$ as well as the variables $a_1$ and $a_2$ being free,
$$\beta^G(A^0_{m}(f_k^{\epsilon, \eta})) = u^m u^{(r+1)(p+q) - 2}  \frac{u^{(r-1)(p+q-2)} - 1}{u^{p+q-2} - 1} \beta^G(Y_{p,q} \setminus \{0\}) + u^{r(2+ p+q)} \beta^G(Y_{p,q}).$$
\\

Finally, if $(p,q) = (1,1)$, the same process gives 
$$\beta^G(A^0_{m}(f_k^{\epsilon, \eta})) = \begin{cases} \beta^G(Y_{1,1} \setminus \{0\}) \sum_{t=0}^{r-1} u^{2m} + \frac{u^{4(r+1)}}{u-1} \mbox{ if $m = 2r+1$,} \\
\beta^G(Y_{1,1} \setminus \{0\}) \sum_{t=0}^{r-2} u^{2m} + u^{4r} \beta^G(Y_{1,1})  \mbox{ if $m = 2r$.}
\end{cases}$$
If $pq = 0$, since $Y_{p,q} = \{0\}$, the equations $Q_{p,q}(c_1) = \ldots = Q_{p,q}(c_r) = 0$ impose $c_1, \ldots, c_r$ to be zero vectors and, the other variables being free, we have
$$\beta^G(A^0_{m}(f_k^{\epsilon, \eta})) = \begin{cases} \frac{u^{m + (r+1)(p+q) + 1}}{u-1} \mbox{ if $m = 2r+1$,} \\
\frac{u^{m + r(p+q) + 1}}{u-1} \mbox{ if $m = 2r$.}
\end{cases}$$ 
\end{proof}

\begin{rem} We obtain the same quantities for $\beta^G(A^0_{m}(g_l^{\epsilon}))$ with $m < 2l$, providing we equip the set $Y_{p,q}$ with the trivial action of $G$. Indeed, the computation steps above remain equivariant if the group $G$ acts on $\mathcal{L}_m$ changing the sign of the variables $a_i$.
\end{rem}

\begin{prop} \label{oamh} Let $h$ be an invariant Nash germ $(\mathbb{R}^n, 0) \rightarrow (\mathbb{R}, 0)$ and $m \geq 2$. Then 
$$\beta^G({}^0\!A_{m}(h)) = u^{n} \beta^G(A^0_{m-1}(h)).$$
\end{prop}

%\begin{cor} Suppose $k \geq 2$ and $m < k+1$. We have $\beta^G({}^0\!A_{m}(f_k^{\epsilon, \eta})) = u^{p+q+1} \beta^G(A^0_{m-1}(f_k^{\epsilon, \eta}))$ and consequently,
%\begin{enumerate}
%	\item if $pq = 0$, then
%$$\beta^G(A_{m}(f_k^{\epsilon, \eta})) = \begin{cases}
%	u^{p+q+1} \frac{u^{m-1 + r(p+q)+1}}{u-1} -  \frac{u^{m + (r+1)(p+q)+1}}{u-1} \mbox{ if $m = 2r+1$,} \\
%	u^{p+q+1} \frac{u^{m-1 + r(p+q) + 1}}{u-1} -  \frac{u^{m + r(p+q)+1}}{u-1} \mbox{ if $m = 2r$,}
%	\end{cases}$$	
%	\item if $(p,q) = 1$, then
%$$\beta^G(A_{m}(f_k^{\epsilon, \eta})) = \begin{cases}
%	(r-1) u^{2m + 1} \beta^G(Y_{1,1} \setminus \{0\}) + u^{4r+3}  \beta^G(Y_{1,1}) -  r u^{2m} \beta^G(Y_{1,1} \setminus \{0\}) \mbox{ if $m = 2r+1$,} \\
%	(r-1) u^{2m + 1}  \beta^G(Y_{1,1} \setminus \{0\}) + (r-1) u^{2m} \beta^G(Y_{1,1} \setminus \{0\}) - u^{4r}  \beta^G(Y_{1,1})  \mbox{ if $m = 2r$,}
%	\end{cases}$$
%	\item if $pq \neq 0$ and $(p,q) \neq (1,1)$

%\end{enumerate}
%\end{cor}

\begin{proof}
Notice that ${}^0\!A_{m}(h) = \{\gamma \in \mathcal{L}_{m}~|~h \circ \gamma(t) = 0 \times t + \cdots + 0 \times t^{m-1} + c t^m + \cdots\}$. Therefore, the system describing ${}^0\!A_{m}(h)$ is the same as the system describing $A^0_{m-1}(h)$, the last $n$ variables being free.

As a consequence, we have an equivariant isomorphism between ${}^0\!A_{m}(h)$ and the product $\mathbb{R}^{n} \times A^0_{m-1}(h)$ (this set is equipped with the diagonal action of $G$, the first term being equipped with the involution $s$) and consequently
$$\beta^G({}^0\!A_{m}(h)) = u^{n} \beta^G(A^0_{m-1}(h)).$$
\end{proof}

We also compute $\beta^G(A^{\xi}_{m}(f_k^{\epsilon, \eta}))$ for $m < k+1$ :

\begin{prop} \label{amfxi} Suppose $k \geq 2$ and $m < k+1$.
		\begin{enumerate}
			\item If $pq = 0$, then
$$\beta^G(A^{\xi}_{m}(f_k^{\epsilon, \eta})) = \begin{cases} 0 \mbox{ if $m = 2r+1$,} \\
u^{m + r(p+q)} \beta^G(Y^{\xi}_{p,q}) \mbox{ if $m = 2r$.}
\end{cases}$$ 
			\item If $(p,q) = (1,1)$, then
$$\beta^G(A^{\xi}_{m}(f_k^{\epsilon, \eta})) = \begin{cases} r u^{2m} \beta^G(Y_{1,1} \setminus \{0\})\mbox{ if $m = 2r+1$,} \\
(r-1) u^{2m} \beta^G(Y_{1,1} \setminus \{0\}) + u^{4r} \beta^G(Y^{\xi}_{1,1})  \mbox{ if $m = 2r$,}
\end{cases}$$
			\item If $pq \neq 0$ and $(p,q) \neq (1,1)$, then
$$\beta^G(A^{\xi}_{m}(f_k^{\epsilon, \eta})) = \begin{cases} u^{m} u^{(r+1)(p+q)-1} \frac{u^{r(p+q-2)} - 1}{u^{p+q-2} - 1} \beta^G(Y_{p,q} \setminus \{0\}) \mbox{ if $m = 2r+1$,} \\
 u^m u^{(r+1)(p+q) - 2}  \frac{u^{(r-1)(p+q-2)} - 1}{u^{p+q-2} - 1} \beta^G(Y_{p,q} \setminus \{0\}) + u^{r(2+ p+q)} \beta^G(Y^{\xi}_{p,q}) \mbox{ if $m = 2r$.}
\end{cases}$$
		\end{enumerate}
\end{prop}

\begin{proof} We first deal with the case $pq \neq 0$, $(p,q) \neq (1,1)$ and $m = 2r$ even. Keeping the notations of the proof of \ref{amf}, the system describing $A^{\xi}_{m}(f_k^{\epsilon, \eta})$ is
$$\begin{cases}
Q_{p,q}(c_1) = 0, \\
\Phi_{p,q}(c_1, c_2) = 0, \\
\cdots \\
\sum_{t= 1}^{r-1} \Phi_{p,q}(c_t,c_{2r-1-t}) = 0, \\
Q_{p,q}(c_r) + \sum_{t= 1}^{r-1} \Phi_{p,q}(c_t,c_{2r-t}) = \xi.
\end{cases}$$
The computation steps are the same as in the proof of proposition \ref{amf}, and we have
$$\beta^G(A^{\xi}_{m}(f_k^{\epsilon, \eta})) = u^m u^{(r+1)(p+q) - 2}  \frac{u^{(r-1)(p+q-2)} - 1}{u^{p+q-2} - 1} \beta^G(Y_{p,q} \setminus \{0\}) + u^{(r-1)(2+p+q)}\beta^G(A^{\xi}_{2}(f_k^{\epsilon, \eta})).$$
Since the set $A^{\xi}_{2}(f_k^{\epsilon, \eta})$ is described by the equation $Q_{p,q}(c_1) = \xi$ and the other variables being free, we obtain the result.

If $m$ is odd, $m = 2r + 1$, as in the proof of proposition \ref{amf}, we obtain
$$\beta^G(A^{\xi}_{m}(f_k^{\epsilon, \eta})) = u^{m} u^{(r+1)(p+q)-1} \frac{u^{r(p+q-2)} - 1}{u^{p+q-2} - 1} \beta^G(Y_{p,q} \setminus \{0\}) + u^{(r-1)(2+ p+q)} \beta^G(A^{\xi}_{3}(f_k^{\epsilon, \eta}) \cap \{c_1 = 0\})$$
and the set $A^{\xi}_{3}(f_k^{\epsilon, \eta}) \cap \{c_1 = 0\}$ is empty.

Similar considerations provide the results for the cases $(p,q) = (1,1)$ and $pq = 0$.
\end{proof}

\begin{rem} Again, we have the same quantities for $\beta^G(A^{\xi}_{m}(g_l^{\epsilon}))$ with $m < 2l$, providing we equip the sets $Y_{p,q}$ and $Y^{\xi}_{p,q}$ with the trivial action of $G$.
\end{rem}

Now, we are ready to deduce distinctions, with respect to $G$-blow-Nash equivalence, between $f_k^{\epsilon,+}$ and $f_k^{\epsilon, -}$, respectively between $f_{2k-1}$ and $g_k$, in some cases :

\begin{cor} \label{fpmnot} Let $k \geq 1$. Suppose that the invariant germs
$$f_k^{\epsilon, +}(x) := + x_1^2 + \epsilon x_2^{k+1} + Q \mbox{   and   } f_k^{\epsilon, -}(x) := - x_1^2 + \epsilon x_2^{k+1} + Q'$$
have the same quadratic part up to permutation of the variables $x_1, x_3, \ldots , x_n$. Then they are not $G$-blow-Nash equivalent.
\end{cor}

\begin{proof} We begin by assuming $k \geq 2$. We first compare $\beta^G(A_2(f_k^{\epsilon, +}))$ and $\beta^G(A_2(f_k^{\epsilon, -}))$. Since $\beta^G({}^0\!A_2(f_k^{\epsilon, \eta})) = u^{1+p+q} \beta^G(A^0_1(f_k^{\epsilon, \eta}))$ (by proposition \ref{oamh}) and $A^0_1(f_k^{\epsilon, \eta}) = \mathcal{L}_1$, we are reduced to compare $\beta^G(A^0_2(f_k^{\epsilon, +}))$ and $\beta^G(A^0_2(f_k^{\epsilon, -}))$.

Denote $p$ the number of signs $+$ and $q$ the number of signs $-$ in the quadratic part of $f_k^{\epsilon, +}$ and $f_k^{\epsilon, -}$ (notice that $pq \neq 0$). Then, according to proposition \ref{amf},
$$\beta^G(A^0_2(f_k^{\epsilon, \eta})) = u^{2+p+q} \beta^G(Y_{p,q}).$$
  
Therefore, by proposition \ref{ypq}, 
\begin{itemize}
 \item if $p < q$, then
$$\beta^G(A^0_2(f_k^{\epsilon, +})) = u^{2+p+q} \frac{u^{p+q} - u^q + u^{p-1}}{u-1} \mbox{ and } \beta^G(A^0_2(f_k^{\epsilon, -})) = u^{2+p+q} \frac{u^{p+q} - u^q + u^{p+1}}{u-1},$$
\item if $q < p$, then
$$\beta^G(A^0_2(f_k^{\epsilon, +})) = u^{2+p+q} \frac{u^{p+q} - u^p+ u^{q+1}}{u-1} \mbox{ and } \beta^G(A^0_2(f_k^{\epsilon, -})) = u^{2+p+q} \frac{u^{p+q} - u^p + u^{q-1}}{u-1}$$
\end{itemize}
In particular, $\beta^G(A^0_2(f_k^{\epsilon, +})) \neq \beta^G(A^0_2(f_k^{\epsilon, -}))$ if $p \neq q$. Consequently, if $p \neq q$, the naive equivariant zeta functions of $f_k^{\epsilon, +}$ and $f_k^{\epsilon, -}$ are different and, by theorem \ref{eqzfinv}, these germs are not $G$-blow-Nash equivalent.
\\

If $p=q$, $\beta^G(A^0_2(f_k^{\epsilon, +})) = \beta^G(A^0_2(f_k^{\epsilon, -}))$ and we look at the term $\beta^G(A^{+1}_2(f_k^{\epsilon, \eta}))$ of the equivariant zeta functions with sign $+$. According to proposition \ref{amfxi}, 
$$\beta^G(A^{+1}_2(f_k^{\epsilon, \eta})) = u^{2+2p} \beta^G(Y^{+1}_{p,p})$$
and, by \ref{yxi}, $\beta^G(Y_{p,p}^{+1}) = \frac{1}{u-1} \left(\beta^G(Y_{p,p+1}) - \beta^G(Y_{p,p})\right)$. Since the quantity $\beta^G(Y_{p,p})$ is the same in either of the cases n\textsuperscript{o}\ref{acty1} and n\textsuperscript{o}\ref{actyp1}, we are reduced to compare the quantities $\beta^G(Y_{p,p+1})$ in the cases n\textsuperscript{o}\ref{acty1} and n\textsuperscript{o}\ref{actyp1}.

We have
$$\beta^G(Y_{p,p+1}) = \begin{cases} \frac{u^{2p+1} - u^{p+1} + u^{p-1}}{u-1} \mbox{ in the case n\textsuperscript{o}\ref{acty1},} \\
\frac{u^{2p+1} - u^{p+1} + u^{p+1}}{u-1} \mbox{ in the case n\textsuperscript{o}\ref{actyp1},}
\end{cases}$$
and, as a consequence, $\beta^G(A^{+1}_2(f_k^{\epsilon, +})) \neq \beta^G(A^{+1}_2(f_k^{\epsilon, -}))$, so $f_k^{\epsilon, +}$ and $f_k^{\epsilon, -}$ are not $G$-blow-Nash equivalent in the case $p = q$ as well.
\\

If $k = 1$, notice that $f_1^{\epsilon, \eta}(x,y) = \epsilon x^2 + Q_{p,q}(y)$ and we are reduced to compare $\beta^G(A^0_2(f_1^{\epsilon, +}))$ and $\beta^G(A^0_2(f_1^{\epsilon, -}))$ as well. We have $\beta^G(A^0_2(f_1^{\epsilon, \eta})) = u^{1+ p+q} \beta^G(Y_{p+1,q})$ if $\epsilon = +1$, and  $\beta^G(A^0_2(f_1^{\epsilon, \eta})) = u^{1+ p+q} \beta^G(Y_{p,q+1})$ if $\epsilon = -1$. As above, we can show, for instance if $\epsilon = +1$, that $\beta^G(A^0_2(f_1^{\epsilon, +})) \neq \beta^G(A^0_2(f_1^{\epsilon, -}))$ when $p+1 \neq q$, and $\beta^G(A^{+1}_2(f_1^{\epsilon, +})) \neq \beta^G(A^{+1}_2(f_1^{\epsilon, -}))$ if $p+1 = q$.

%If $k = 0$, $f_k^{\epsilon, \eta}(x,y) = \epsilon x + Q_{p,q}(y)$, using the notations of the proof of proposition \ref{amf}, the equation defining $A^0_1(f_k^{\epsilon, \eta})$ is $a_1 = 0$ and the equations describing $A^0_2(f_k^{\epsilon, \eta})$ are $\epsilon a_1 = 0, \epsilon a_2 + Q_{p,q}(c_1) = 0$.  
\end{proof}

\begin{rem} If $k = 0$, $f_0^{\epsilon, \eta}(x,y) = \epsilon x + Q_{p,q}(y)$ and, using the notations of the proof of proposition \ref{amf}, the left members of all the equations describing $A^0_m(f_0^{\epsilon, \eta})$, resp. $A^{\xi}_m(f_0^{\epsilon, \eta})$, for $m \geq 1$, contain a term $\epsilon a_i + \ldots$, so that each of these sets is equivariantly isomorphic to an affine space. As a consequence (see remark \ref{equivpoinc}), the respective equivariant zeta functions of $f_0^{\epsilon, +}$ and $f_0^{\epsilon, -}$ are equal.
\end{rem}

\begin{cor} \label{fgprem} Let $k \geq 2$. Suppose that the invariant germs
$$f_{2k-1} = \epsilon x_2^{2k} + \eta x_1^2 + Q \mbox{   and   } g_k = \epsilon x_1^{2k} + \eta' x_2^2 + Q'$$
have, up to permutation of all variables, the same quadratic part, with $p$ signs $+$ and $q$ signs~$-$. 

If $p \leq q$ and $\eta = + 1$ or $q \leq p$ and $\eta = -1$, then $f_{2k-1}$ and $g_k$ are not $G$-blow-Nash equivalent. 

If $p = q+1$ or $q = p+1$, then $f_{2k-1}$ and $g_k$ are not $G$-blow-Nash equivalent.
\end{cor}

\begin{proof} We first deal with the case $p \leq q$ and $\eta = +1$ (notice that $ p \neq 0$) ; the case $q \leq p$ and $\eta = -1$ is symmetric. 

As in the proof of previous corollary \ref{fpmnot}, we have  
$$\beta^G(A^0_2(f_{2k-1})) = u^{2+p+q} \beta^G(Y_{p,q}) \mbox{ and } \beta^G(A^0_2(g_k)) = u^{2+p+q} \beta^G(Y_{p,q})$$
where, in the left equality, the set $Y_{p,q}$ is equipped with the action n\textsuperscript{o}\ref{acty1} and, in the right one, with the trivial action of $G$. Since the corresponding equivariant virtual Poincar\'e series are different by proposition \ref{ypq}, $\beta^G(A_2(f_{2k-1})) \neq \beta^G(A_2(g_k))$ and the naive equivariant zeta functions of $f_{2k-1}$ and $g_k$ are different. As a consequence, $f_{2k-1}$ and $g_k$ are not $G$-blow-Nash equivalent.
\\

Now we suppose $p = q+1$ (the case $q = p+1$ is symmetric). In particular $q < p$, so we can assume $\eta = +1$. 

We consider $\beta^G(A^{+1}_2(f_{2k-1})) = u^{2+p+q} \beta^G(Y^{+1}_{p,q})$ and $\beta^G(A^{+1}_2(g_k)) = u^{2+p+q} \beta^G(Y^{+1}_{p,q})$ (proposition \ref{amfxi}). Thanks to proposition \ref{yxi}, we know that $\beta^G(Y^{+1}_{p,q}) = \frac{1}{u-1} \left(\beta^G(Y_{p,q+1}) - \beta^G(Y_{p,q})\right)$. By proposition \ref{ypq}, the respective quantities $\beta^G(Y_{p,q})$ for $f_{2k-1}$ and $g_k$ are equal, whereas the quantities $\beta^G(Y_{p,q+1}) = \beta^G(Y_{p,p})$ are different. Consequently, the equivariant zeta functions with sign $+$ of $f_{2k-1}$ and $g_k$ are different and therefore the latter germs are not $G$-blow-Nash equivalent.
\end{proof}

\begin{rem} In the other cases, the quantities $\beta^G(Y_{p,q})$ and $\beta^G(Y^{\xi}_{p,q})$ are the same for $f_{2k-1}$ and $g_k$.
\end{rem}

\subsection{Computation of $\beta^G(A_{2k}(f_{2k-1}))$ and $\beta^G(A_{2k}(g_k))$} \label{a2kfg}

For the continuation of the section, thanks to corollaries \ref{fpmnot} and \ref{fgprem}, we only need to consider the germs 
$$f_{2k-1} = \epsilon x_2^{2k} + \eta x_1^2 + Q \mbox{   and   } g_k = \epsilon x_1^{2k} + \eta' x_2^2 + Q',$$
assumed to have the same quadratic part $Q_{p,q}$, such that $p > q+1$ and $\eta = + 1$ or $q > p+1$ and $\eta = -1$.

In order to prove that the germs $f_{2k-1}$ and $g_k$ are not $G$-blow-Nash equivalent in some of these cases as well, we will compute the coefficients $\beta^G(A_{2k}(f_{2k-1}))$ and $\beta^G(A_{2k}(g_k))$ of their respective naive equivariant zeta functions :

\begin{prop} \label{a02k} Suppose $k \geq 2$.
\begin{enumerate}
	\item If $pq = 0$, then
$$\beta^G(A_{2k}^0(f_{2k-1})) = u^{2k - 1 + k(p+q)} \beta^G(\{f_{2k-1} = 0 \}) \mbox{ and } \beta^G(A_{2k}^0(g_k)) = u^{2k - 1 + k(p+q)} \beta^G(\{g_k = 0 \}).$$	
	\item If $pq \neq 0$, then
$$\beta^G(A^{0}_{2k}(f_{2k-1})) = u^{2k-2} u^{(p+q)(k+1)} \beta^G(Y_{p,q} \setminus \{0\}) \frac{u^{(p+q-2)(k-1)}-1}{u^{p+q-2}-1} + u^{k(p+q) + 2k-1} \beta^G(\{f_{2k-1} = 0\})$$
(the group $G$ acts on $Y_{p,q}$ via the involution n\textsuperscript{o}\ref{acty1} or n\textsuperscript{o}\ref{actyp1} depending on the sign of $\eta$) and
$$\beta^G(A^{0}_{2k}(g_k)) = u^{2k-2} u^{(p+q)(k+1)} \beta^G(Y_{p,q} \setminus \{0\}) \frac{u^{(p+q-2)(k-1)}-1}{u^{p+q-2}-1} + u^{k(p+q) + 2k-1} \beta^G(\{g_k = 0\})$$
(the group $G$ acts trivially on $Y_{p,q}$). 
\end{enumerate}
\end{prop}

\begin{proof} We keep the notations of the proof of proposition \ref{amf} and we proceed as in \cite{GF-BNT} Proof of Lemma 3.3. First suppose that $pq \neq 0$. An arc $\gamma$ of $\mathcal{L}_{2k}$ belongs to $A_{2k}^0(f_{2k-1})$ if and only if
$$\begin{cases}
Q_{p,q}(c_1) = 0, \\
\Phi_{p,q}(c_1, c_2) = 0, \\
\cdots \\
\sum_{t= 1}^{k-1} \Phi_{p,q}(c_t,c_{2k-1-t}) = 0, \\
\epsilon a_1^{2k} + Q_{p,q}(c_k) + \sum_{t= 1}^{k-1} \Phi_{p,q}(c_t,c_{2k-t}) = 0.
\end{cases}$$
We have 
$$\beta^G(A_{2k}^0(f_{2k-1})) = u^{2k + (2k-1)(p+q-1) + 1} \beta^G(Y_{p,q} \setminus \{0\}) + \beta^G(A_{2k}^0(f_{2k-1}) \cap \{c_1^1 = \ldots = c_1^p = 0\}),$$
and $\beta^G(A_{2k}^0(f_{2k-1}) \cap \{c_1^1 = \ldots = c_1^p = 0\}) = u^{2+p+q}\beta^G(C^0_{2k-2})$, if $C^0_{2k-2}$ denotes the algebraic set described by the equations
$$\begin{cases}
Q_{p,q}(c_1) = 0, \\
\Phi_{p,q}(c_1, c_2) = 0, \\
\cdots \\
\sum_{t= 1}^{k-2} \Phi_{p,q}(c_t,c_{2k-3-t}) = 0, \\
\epsilon a_1^{2k} + Q_{p,q}(c_{k-1}) + \sum_{t= 1}^{k-2} \Phi_{p,q}(c_t,c_{2k-2-t}) = 0.
\end{cases}$$

By an induction, we obtain
$$\beta^G(A^0_{2k}(f_{2k-1})) =  \beta^G(Y_{p,q} \setminus \{0\}) \sum_{t=0}^{k-2} u^{t(2+p+q)} u^{2k-2t + (2k-2t-1)(p+q-1)+1} + u^{(k-1)(2+ p+q)} \beta^G(C^0_{2}).$$
Since $C^0_{2}$ is defined by the equation $\epsilon a_1^{2k} + Q_{p,q}(c_1) = 0$ and since the vector $c_2$ and the variable $a_2$ are free, we deduce the desired expression for $\beta^G(A^0_{2k}(f_{2k-1}))$. The steps of computation are the same for $\beta^G(A^0_{2k}(g_{k}))$. 
\\

If $pq = 0$, the vectors $c_1, \ldots, c_{k-1}$ are zero vectors and the system is reduced to the equation $\epsilon a_1^{2k} + Q_{p,q}(c_k) = 0$, the other variables being free.
\end{proof}

Since $p > q+1$ and $\eta = + 1$ or $q > p+1$ and $\eta = -1$, the quantity $\beta^G(Y_{p,q})$ is the same for $f_{2k-1}$ and $g_k$. As a consequence, in order to compare $\beta^G(A_{2k}(f_{2k-1}))$ and $\beta^G(A_{2k}(g_k))$, we are reduced to consider the quantities $\beta^G(\{f_{2k-1} = 0 \})$ and $\beta^G(\{g_{k} = 0 \})$ (notice that $\beta^G({}^0\!A_{2k}(f_{2k-1})) = \beta^G({}^0\!A_{2k}(g_k))$ by the results of the previous paragraph \ref{firsttermsab}). We compute these equivariant virtual Poincar\'e series for all $k \geq 2$, $p,q \in \mathbb{N}$ and $\eta \in \{1, -1\}$ :

\begin{lem} \label{fg2k0} We have
$$\beta^G(\{f_{2k-1} = 0\}) = \beta^G(\{\epsilon x_2^2 + \eta x_1^2 + Q = 0\}) - (k-1) \beta^G(\{\eta x_1^2 + Q = 0\}) + (k-1) \beta^G(\{0\}),$$ 
where the second set in the right member is considered as an algebraic subset of $\mathbb{R}^{n-1}$ and $G$ acts on the considered sets via the involution n\textsuperscript{o}\ref{acty1} or n\textsuperscript{o}\ref{actyp1} depending on the sign of $\eta$, and
$$\beta^G(\{g_k = 0\}) = \beta^G(\{\epsilon x_1^2 + \eta' x_2^2 + Q' = 0\}) - \rho \beta^G(\{\eta' x_2^2 + Q' = 0\}) - \tau \beta^G(\{\eta' x_2^2 + Q' = 0\}) + (k-1) \beta^G(\{0\}),$$ 
where the second and third sets in the right member are considered as algebraic subsets of $\mathbb{R}^{n-1}$, the group $G$ acts on the second set via the involution n\textsuperscript{o}\ref{triv} (trivial action), on the third set via the involution n\textsuperscript{o}\ref{acttout} (change of signs of all coordinates) and
\begin{enumerate}
	\item if $k = 2l + 1$ is odd, then $\rho = \tau = l$ and $G$ acts on the first set in the right member via the involution n\textsuperscript{o}\ref{acty1} or n\textsuperscript{o}\ref{actyp1} depending on the sign of $\epsilon$, 
	\item if $k = 2l$ is even, then $\rho = l$, $\tau = l-1$ and $G$ acts on the first set in the right member via the involution n\textsuperscript{o}\ref{acttout}.
\end{enumerate} 
\end{lem}  

\begin{proof}
We begin with $\beta^G(\{f_{2k-1} = 0\})$. Recall that $f_{2k-1}(x_1,x_2,x_3, \ldots, x_n) = \epsilon x_2^{2k} + \eta x_1^2 + Q(x_3,\ldots,x_n)$. We proceed to an equivariant blowing-up of the algebraic set $\{f_{2k-1} = 0\}$ at the origin of $\mathbb{R}^n$. In the chart $x_2 = u$, $x_i = u v_i$, $i = 1,3, \ldots, n$, the blown-up variety is defined by the equation
$$u^2 f_{2k-3}(v_1,u,v_3, \ldots, v_n) = 0,$$ 
the action of $G$ being given by the involution $(v_1,u,v_3, \ldots, v_n) \mapsto (-v_1,u,v_3, \ldots, v_n)$. We have $\beta^G( \{f_{2k-1} = 0\} \setminus \{0\}) = \beta^G(\{f_{2k-3} = 0 \} \setminus \{u = 0\})$, therefore
$$\beta^G( \{f_{2k-1} = 0\}) = \beta^G(\{f_{2k-3} = 0 \}) - \beta^G(\{\eta v_1^2 + Q(v_3,\ldots,v_n) = 0, u=0\}) + \beta^G(\{0\}).$$
We then obtain the desired result by an induction.
\\

For the computation of $\beta^G(\{g_k = 0\})$, recall that $g_k(x_1,x_2,x_3, \ldots, x_n) = \epsilon x_1^{2k} + \eta' x_2^2 + Q'(x_3, \ldots, x_n)$ and proceed to an equivariant blowing-up of the set $\{g_k = 0\}$ at the origin of $\mathbb{R}^n$, looked at in the chart $x_1 = u$, $x_i = u v_i$, $i = 2,3, \ldots, n$. In this chart, the blown-up variety is defined by 
$$u^2 g_{k-1}(u,v_2,v_3, \ldots, v_n) = 0,$$ 
the action of $G$ being given by the involution $(u,v_2,v_3, \ldots, v_n) \mapsto (-u,-v_2,-v_3, \ldots, -v_n)$, and we have
$$\beta^G( \{g_k = 0\}) = \beta^G(\{g_{k-1} = 0 \}) - \beta^G\{\eta' x_2^2 + Q'(x_3,\ldots,x_n) = 0\}) + \beta^G(\{0\}).$$
One further equivariant blowing-up of $\{g_{k-1} = 0\}$ provides the equation 
$$u^2 g_{k-1}(u,v_2,v_3, \ldots, v_n) = 0,$$ 
the group $G$ acting via the involution $(u,v_2,v_3, \ldots, v_n) \mapsto (-u, v_2, v_3, \ldots, v_n)$. The desired expression is then obtained by an induction.
\end{proof}

\begin{rem} According to proposition \ref{ypq}, the quantity $\beta^G(\{\eta' x_2^2 + Q' = 0\})$ is the same if $G$ acts via the involution n\textsuperscript{o}\ref{triv} or via the involution n\textsuperscript{o}\ref{acttout}. Therefore, in the previous lemma \ref{fg2k0}, we can simply write $\rho \beta^G(\{\eta' x_2^2 + Q' = 0\}) + \tau \beta^G(\{\eta' x_2^2 + Q' = 0\})$ as $(k-1) \beta^G(\{\eta' x_2^2 + Q' = 0\})$ with $G$ acting trivially on the latter set.
\end{rem}

Because $p > q+1$ and $\eta = + 1$ or $q > p+1$ and $\eta = -1$, we have $\beta^G(\{\eta x_1^2 + Q = 0\}) =  \beta^G(\{\eta' x_2^2 + Q' = 0\})$ and we are finally reduced to compare $\beta^G(\{\epsilon x_2^2 + \eta x_1^2 + Q = 0\})$ and $\beta^G(\{\epsilon x_1^2 + \eta' x_2^2 + Q' = 0\})$. The cases where these quantities are different are cases where the germs $f_{2k-1}$ and $g_k$ are not $G$-blow-Nash-equivalent :

\begin{cor} \label{cora2k} 
%If $k$ is even and if $p = q+1$, $\eta = +1$ and $\epsilon = -1$ or $q = p+1$, $\eta = -1$ and $\epsilon = +1$, then the germs $f_{2k-1}$ and $g_k$ are not $G$-blow-Nash-equivalent.

If $k$ is odd and if $p > q +1$, $\eta = +1$ and $\epsilon = -1$ or $q > p+1$, $\eta = -1$ and $\epsilon = +1$, then the germs $f_{2k-1}$ and $g_k$ are not $G$-blow-Nash-equivalent.
\end{cor}

\begin{proof}
%First assume that $k$ is even. We suppose $p = q+1$, $\eta = +1$ and $\epsilon = -1$ (the case $q = p+1$, $\eta = -1$ and $\epsilon = +1$ being symmetric). Then $\{\epsilon x_2^2 + \eta x_1^2 + Q = 0\} = Y_{p,p}$, where $Y_{p,p}$ is equipped with the involution n\textsuperscript{o}\ref{acty1}, and $\{\epsilon x_1^2 + \eta' x_2^2 + Q' = 0\} = Y_{p,p}$, where $Y_{p,p}$ is here equipped with the involution n\textsuperscript{o}\ref{acttout}. Therefore, according to proposition \ref{ypq}, $\beta^G(\{\epsilon x_2^2 + \eta x_1^2 + Q = 0\}) \neq \beta^G(\{\epsilon x_1^2 + \eta' x_2^2 + Q' = 0\})$ and consequently $\beta^G(A_{2k}(f_{2k-1})) \neq \beta^G(A_{2k}(g_k))$. 

Assume that $k$ is odd and suppose that $p > q +1$, $\eta = +1$ and $\epsilon = -1$ (the case $q > p+1$, $\eta = -1$ and $\epsilon = +1$ is symmetric). We have $\{\epsilon x_2^2 + \eta x_1^2 + Q = 0\} = Y_{p,q+1}$, where $Y_{p,q+1}$ is equipped with the involution n\textsuperscript{o}\ref{acty1}, and $\{\epsilon x_1^2 + \eta' x_2^2 + Q' = 0\} = Y_{p,q+1}$, where $Y_{p,q+1}$ is equipped with the involution n\textsuperscript{o}\ref{actyp1}. Then, by proposition \ref{ypq}, $\beta^G(\{\epsilon x_2^2 + \eta x_1^2 + Q = 0\}) \neq \beta^G(\{\epsilon x_1^2 + \eta' x_2^2 + Q' = 0\})$ and $\beta^G(A_{2k}(f_{2k-1})) \neq \beta^G(A_{2k}(g_k))$.
\end{proof}

In the remaining cases, the quantities $\beta^G(\{\epsilon x_2^2 + \eta x_1^2 + Q = 0\})$ and $\beta^G(\{\epsilon x_1^2 + \eta' x_2^2 + Q' = 0\})$ are equal so that $\beta^G(A_{2k}(f_{2k-1})) = \beta^G(A_{2k}(g_k))$. As a consequence, for these cases, we are led to look at the remaining coefficients of the equivariant zeta functions of $f_{2k-1}$ and $g_k$. We begin, in the following paragraph, with the computation of the terms $\beta^G(A_{2k}^{\xi}(f_{2k-1}))$ and $\beta^G(A_{2k}^{\xi}(g_k))$.

\subsection{Computation of $\beta^G(A_{2k}^{\xi}(f_{2k-1}))$ and $\beta^G(A_{2k}^{\xi}(g_k))$} \label{subsectiona2kxi}

We assume we are not in one of the previous cases for which we showed that $f_{2k-1}$ and $g_k$ are not $G$-blow-Nash-equivalent. In particular, we have $\beta^G(A_m(f_{2k-1})) = \beta^G(A_m(g_k))$ and $\beta^G(A^{\xi}_m(f_{2k-1})) = \beta^G(A^{\xi}_m(g_k))$ for $m < 2k$, and $\beta^G(A_{2k}(f_{2k-1})) = \beta^G(A_{2k}(g_k))$.

Now, the same steps of computation as in the proof of proposition \ref{a02k} provide the following formulae for $\beta^G(A_{2k}^{\xi}(f_{2k-1}))$ and $\beta^G(A_{2k}^{\xi}(g_k))$ :

\begin{prop} \label{axi2k} Suppose $k \geq 2$.
\begin{enumerate}
	\item If $pq = 0$, then
$$\beta^G(A_{2k}^{\xi}(f_{2k-1})) = u^{2k - 1 + k(p+q)} \beta^G(\{f_{2k-1} = \xi \}) \mbox{ and } \beta^G(A_{2k}^{\xi}(g_k)) = u^{2k - 1 + k(p+q)} \beta^G(\{g_k = \xi \}).$$	
	\item If $pq \neq 0$, then
$$\beta^G(A^{\xi}_{2k}(f_{2k-1})) = u^{2k-2} u^{(p+q)(k+1)} \beta^G(Y_{p,q} \setminus \{0\}) \frac{u^{(p+q-2)(k-1)}-1}{u^{p+q-2}-1} + u^{k(p+q) + 2k-1} \beta^G(\{f_{2k-1} = \xi\})$$
(the group $G$ acts on $Y_{p,q}$ via the involution n\textsuperscript{o}\ref{acty1} or n\textsuperscript{o}\ref{actyp1} depending on the sign of $\eta$) and
$$\beta^G(A^{\xi}_{2k}(g_k)) = u^{2k-2} u^{(p+q)(k+1)} \beta^G(Y_{p,q} \setminus \{0\}) \frac{u^{(p+q-2)(k-1)}-1}{u^{p+q-2}-1} + u^{k(p+q) + 2k-1} \beta^G(\{g_k = \xi\})$$
(the group $G$ acts trivially on $Y_{p,q}$). 
\end{enumerate}
\end{prop}
 
As in the previous paragraph \ref{a2kfg}, we are reduced to consider the quantities $\beta^G(\{f_{2k-1} = \xi \})$ and $\beta^G(\{g_k = \xi \})$. We give below the first steps of computation of these equivariant virtual Poincar\'e series for all $k \geq 2$, $(p,q) \in \mathbb{N}^2 \setminus \{(0,0)\}$ and $\eta \in \{1, -1\}$. We write $f_{2k-1} = \epsilon x_2^{2k} + \eta x_1^2 + Q =  \epsilon x_2^{2k} + \sum_{i = 1}^p y_i^2 - \sum_{j=1}^q y_{p+j}^2$ and $g_{k} = \epsilon x_1^{2k} + \eta' x_2^2 + Q' = \epsilon x_1^{2k} + \sum_{i = 1}^p y_i^2 - \sum_{j=1}^q y_{p+j}^2$. Then :

\begin{lem} \label{fgxi} We have
$$\beta^G(\{f_{2k-1} =\xi \}) =
\begin{cases}
u^{q+2} \frac{u^{p-1}-1}{u-1} + u^{p-1} \beta^G\left(\left\{\epsilon x_2^{2k} + y_1^2 - y_{p+1}^2 - \sum_{j = 2p+1}^{p+q} y_j^2 = \xi \right\}\right) & \mbox{ if } 0 < p < q, \\
u^{p+2} \frac{u^{q-1}-1}{u-1} + u^{q-1} \beta^G\left(\left\{\epsilon x_2^{2k} + y_1^2 - y_{p+1}^2 + \sum_{j = q+1}^{p} y_j^2 = \xi \right\}\right) & \mbox{ if } 0 < q < p, \\
u^{p+2} \frac{u^{p-1}-1}{u-1} + u^{p-1} \beta^G\left(\left\{\epsilon x_2^{2k} + y_1^2 - y_{p+1}^2 = \xi \right\}\right) & \mbox{ if } p = q, \\
\beta^G\left(\left\{\epsilon x_2^{2k} - y_{1}^2 - \sum_{j = 2}^{q} y_j^2 = \xi \right\}\right) & \mbox{ if } p = 0, \\
\beta^G\left(\left\{\epsilon x_2^{2k} + y_{1}^2 + \sum_{i = 2}^{p} y_i^2 = \xi \right\}\right) & \mbox{ if } q = 0,
\end{cases}
$$
the group $G$ acting only changing the sign of $y_1$ or $y_{p+1}$ depending on the sign of $\eta$, and
$$\beta^G(\{g_k = \xi\}) =
\begin{cases} u^{q+1} \frac{u^p - 1}{u-1} + u^p \beta^G\left(\left\{\epsilon x_1^{2k} - \sum_{j = 2p+1}^{p+q} y_j^2 = \xi \right\}\right) & \mbox{ if } 0 < p < q, \\
u^{p+1} \frac{u^q - 1}{u-1} + u^q \beta^G\left(\left\{\epsilon x_1^{2k} + \sum_{j = q+1}^{p} y_j^2 = \xi \right\}\right) & \mbox{ if } 0 < q < p, \\
u^{p+1} \frac{u^p - 1}{u-1} + u^p  \beta^G\left(\left\{\epsilon x_1^{2k} = \xi \right\}\right) & \mbox{ if } p = q, \\
\beta^G\left(\left\{\epsilon x_1^{2k} - \sum_{j = 1}^{q} y_j^2 = \xi \right\}\right) & \mbox{ if } p = 0, \\
\beta^G\left(\left\{\epsilon x_1^{2k} + \sum_{i = 1}^{p} y_i^2 = \xi \right\}\right) & \mbox{ if } q = 0,
\end{cases}
$$
the group $G$ acting only changing the sign of $x_1$. 
\end{lem}

\begin{proof}

We focus on the case $2 \leq p \leq q$ and proceed as in the proof of proposition \ref{ypq} : in order to compute $\beta^G(\{f_{2k-1} =\xi \})$, we apply the (equivariant) change of variables $u_i = y_i + y_{i+p}$, $v_i = y_i - y_{i+p}$ for $i = 2, \ldots, p$ and the equation $f_{2k-1} = \xi$ becomes 
$$\epsilon x_2^{2k} + y_1^2 - y_{p+1}^2 + \sum_{i=2}^p u_i v_i - \sum_{j = 2p+1}^{p+q} y_j^2 = \xi.$$ 
Then, as in the proof of proposition \ref{ypq}, we use the stratification by the globally $G$-stable subsets $\{f_{2k-1} =\xi \} \cap \{u_2 = \ldots = u_i = 0, u_{i+1} \neq 0\}$, along with the additivity of the equivariant virtual Poincar\'e series, to obtain the desired formula for $\beta^G(\{f_{2k-1} =\xi \})$.  

As for $\beta^G(\{g_k = \xi\})$, we can apply the equivariant change of variables $u_i = y_i + y_{i+p}$, $v_i = y_i - y_{i+p}$ for $i = 1, \ldots, p$ (the strata $\{g_{k} =\xi \} \cap \{u_1 = \ldots = u_i = 0, u_{i+1} \neq 0\}$ are $G$-globally stable).

\end{proof}

\begin{rem} Regarding the equation $f_{2k-1} =\xi$, we could also have applied the change of variables $u_1 = y_1 + y_{p+1}$, $v_1 = y_1 - y_{p+1}$, provided $G$ acts on these new coordinates via the involution $(u_1, v_1) \mapsto (-v_1, -u_1)$ or $(u_1, v_1) \mapsto (v_1, u_1)$ (depending on the sign of $\eta$). However, the stratum $\{f_{2k-1} =\xi \} \cap \{u_1 \neq 0\}$ is not globally stable under this action of $G$.
\end{rem}

From these formulae, among the remaining cases for which we did not establish that the germs $f_{2k-1}$ and $g_k$ are not $G$-blow-Nash equivalent, we first extract the cases for which $\beta^G(A_{2k}^{\xi}(f_{2k-1})) = \beta^G(A_{2k}^{\xi}(g_k))$ :

\begin{prop} \label{eqa2kxi} If $p > q+1$ and $\eta = \epsilon = +1$ or $q  > p+1$ and $\eta = \epsilon = - 1$, we have $\beta^G(A_{2k}^{\xi}(f_{2k-1})) = \beta^G(A_{2k}^{\xi}(g_k))$.
\end{prop}

\begin{proof}
Similarly to the previous proofs, we focus on the case $p > q+1$, $q \neq 0$ and $\eta = \epsilon = + 1$. Then $\beta^G(\{f_{2k-1} =\xi \}) = u^{p+2} \frac{u^{q-1}-1}{u-1} + u^{q-1} \beta^G\left(\left\{+ x_2^{2k} + y_1^2 - y_{p+1}^2 + \sum_{j = q+1}^{p} y_j^2 = \xi \right\}\right)$. On the latter set, the action of $G$ only changes the sign of $y_1$, so that we can use the equivariant change of variables $u = y_{q+1} + y_{p+1}$, $v = y_{q+1} - y_{p+1}$ in order to obtain the equality 
$$\beta^G(\{f_{2k-1} =\xi \}) = u^{p+1} \frac{u^q - 1}{u-1} + u^q  \beta^G\left(\left\{+ x_2^{2k} + y_1^2 + \sum_{j = q+2}^{p} y_j^2 = \xi \right\}\right).$$
Therefore $\beta^G(A_{2k}^{\xi}(f_{2k-1})) = \beta^G(A_{2k}^{\xi}(g_k))$ if and only if $\beta^G\left(\left\{+ x_2^{2k} + y_1^2 + \sum_{j = q+2}^{p} y_j^2 = \xi \right\}\right) = \beta^G\left(\left\{+ x_1^{2k} + \sum_{j = q+1}^{p} y_j^2 = \xi \right\}\right)$ (recall that, on the latter set, the action of $G$ only changes the sign of $x_1$). 

Now, if $\xi = -1$, both sets are empty and if $\xi = +1$, they are compact, nonsingular and equivariantly homeomorphic to spheres having a non-empty fixed point set. As a consequence, for $\xi = \pm 1$, $\beta^G\left(\left\{+ x_2^{2k} + y_1^2 + \sum_{j = q+2}^{p} y_j^2 = \xi \right\}\right) = \beta^G\left(\left\{+ x_1^{2k} + \sum_{j = q+1}^{p} y_j^2 = \xi \right\}\right)$ (see remark \ref{equivpoinc}) and $\beta^G(A_{2k}^{\xi}(f_{2k-1})) = \beta^G(A_{2k}^{\xi}(g_k))$.
\end{proof}

Finally, we give the cases for which the equality $\beta^G(A_{2k}^{\xi}(f_{2k-1})) = \beta^G(A_{2k}^{\xi}(g_k))$ depends on the equality of two equivariant virtual Poincar\'e series :

\begin{prop} \label{a2kxidepend} \begin{itemize} 
 \item If $k$ is even and if $p > q + 1$, $\eta = +1$ and $\epsilon = -1$, the equality $\beta^G(A_{2k}^{\xi}(f_{2k-1})) = \beta^G(A_{2k}^{\xi}(g_k))$ is true if and only if the equivariant virtual Poincar\'e series of the algebraic subsets $\left\{ - x_2^{2k} + y^2 + \sum_{i = 1}^{K-1} y_i^2 = \xi \right\} \subset \mathbb{R}^{K+1}$, $K := p-q$, equipped with the action of $G$ only changing the sign of $y$, and $\left\{ - x_1^{2k} + \sum_{i = 1}^{K} z_i^2 = \xi \right\}\subset \mathbb{R}^{K+1}$, equipped with the action of $G$ only changing the sign of $x_1$, are equal.
 	\item If $k$ is even and if $q > p+1$, $\eta = -1$ and $\epsilon = +1$, we have $\beta^G(A_{2k}^{\xi}(f_{2k-1})) = \beta^G(A_{2k}^{\xi}(g_k))$ if and only if $\beta^G\left(\left\{ x_2^{2k} - y^2 - \sum_{i = 1}^{K-1} y_i^2 = \xi \right\}\right) = \beta^G\left(\left\{ x_1^{2k} - \sum_{i = 1}^{K} z_i^2 = \xi \right\} \right)$.
\end{itemize}
\end{prop}

\begin{proof} If we focus on the case $p > q + 1$, $q \neq 0$, $\eta = +1$ and $\epsilon = -1$, the same computation as in the proof of the previous proposition \ref{eqa2kxi} provides the equivalence $\beta^G(A_{2k}^{\xi}(f_{2k-1})) = \beta^G(A_{2k}^{\xi}(g_k))$ if and only if $\beta^G\left(\left\{- x_2^{2k} + y_1^2 + \sum_{j = q+2}^{p} y_j^2 = \xi \right\}\right) = \beta^G\left(\left\{- x_1^{2k} + \sum_{j = q+1}^{p} y_j^2 = \xi \right\}\right)$.  
\end{proof}

\begin{rem} \label{remevpsas} \begin{enumerate}
	\item Recall that we showed in corollary \ref{cora2k} that the germs $f_{2k-1}$ and $g_k$ are not $G$-blow-Nash equivalent in the case $k$ odd and $p > q +1$, $\eta = +1$, $\epsilon = -1$ or $q > p+1$, $\eta = -1$, $\epsilon = +1$ (notice that in the previous proof of proposition \ref{a2kxidepend}, we did not use the fact that $k$ was even). 
	\item Forgetting the action of $G$, the virtual Poincar\'e polynomials of the algebraic subsets $\left\{ x^{2k} - \sum_{i = 1}^{K} y_i^2 = \xi \right\}$, $\xi = \pm 1$, of $\mathbb{R}^{K+1}$ can be computed using the invariance of the virtual Poincar\'e polynomial under bijection with $\mathcal{AS}$ graph (see \cite{MCP}). However, we do not know if the equivariant virtual Poincar\'e series is invariant under equivariant bijection with $\mathcal{AS}$ graph. 
\end{enumerate}
\end{rem}

As a consequence of the results of this subsection \ref{subsectiona2kxi}, we will then consider the other coefficients $\beta^G(A_M(f_{2k-1}))$ and $\beta^G(A_M(g_k))$, respectively $\beta^G(A^{\xi}_M(f_{2k-1}))$ and $\beta^G(A^{\xi}_M(g_k))$, $M> 2k$, of the equivariant zeta functions of $f_{2k-1}$ and $g_k$, in the cases of propositions \ref{eqa2kxi} and \ref{a2kxidepend}. In the next paragraph, we will show that the comparison of these quantities reduces to the comparison of the equivariant virtual Poincar\'e series of $\{f_{2k-1} =\xi \}$ and $\{g_{k} =\xi \}$ as well.

\subsection{The last terms of the equivariant zeta functions} \label{lasttermsfg}

Suppose $p > q +1$, $\eta = \epsilon = +1$ or $k$ even, $p > q+1$, $\eta = +1$, $\epsilon = -1$. The following results will also be true for the respective symmetric cases.

We first establish the equality between the last coefficients of the naive equivariant zeta functions of $f_{2k-1}$ and $g_k$ (and therefore the equality of $Z_{f_{2k-1}}^G(u,T)$ and $Z_{g_k}^G(u,T)$) : 

\begin{prop} \label{aM02k} For all $M > 2k$, we have $\beta^G(A_M(f_{2k-1})) = \beta^G(A_M(g_k))$.
\end{prop}

\begin{proof} Let $M$ be greater than $2k$. We prove that $\beta^G(A^0_M(f_{2k-1})) = \beta^G(A^0_M(g_k))$ (this will give the desired result because of proposition \ref{oamh} and the additivity of the equivariant virtual Poincar\'e series). 

As in the proofs of propositions \ref{amf} and \ref{a02k}, consider the system of equations defining $A^0_M(f_{2k-1})$. The same computations will bring, in the expression of $\beta^G(A^0_M(f_{2k-1}))$, a contribution of (a multiple in $\mathbb{Z}[u][[u^{-1}]]$ of) $\beta^G(Y_{p,q} \setminus \{0\})$ and a contribution of the equivariant virtual Poincar\'e series of a set defined by a system whose first equation is $\epsilon a_1^{2k} + Q_{p,q}(c_1) = 0$. Stratifying this last algebraic set with the subsets $\{c_1^1 = \ldots = c_1^{i-1} = 0, c_1^i \neq 0\}$, $i = 1, \ldots, p+q$, and $\{c_1 = 0\}$ provides a contribution of $\beta^G(\{f_{2k-1} = 0\} \setminus \{0\})$ and a new system where $c_1 = 0$, $a_1 = 0$ and whose first (non trivial) equations are the ones defining $A^0_m(f_{2k-1})$ for $m = min(M-2k, 2k)$.  

As a consequence, we can repeat the same steps of computations on this system and this will give further contributions of $\beta^G(Y_{p,q} \setminus \{0\})$ (provided by the equations $Q_{p,q}(c_1) = 0$) and $\beta^G(\{f_{2k-1} = 0\} \setminus \{0\})$ (provided by the equations $\epsilon a_j^{2k} + Q_{p,q}(c_1) = 0$).

Since these systems and these operations are also valid for the computation of $\beta^G(A^0_M(g_k))$ and because, in the considered cases, the quantities $\beta^G(Y_{p,q})$ are equal for $f_{2k-1}$ and $g_k$ and $\beta^G(\{f_{2k-1} = 0\}) = \beta^G(\{g_{k} = 0\})$, the expressions of $\beta^G(A^0_M(f_{2k-1}))$ and $\beta^G(A^0_M(g_k))$ are identical.
\end{proof}

Similar considerations bring the following results for the last coefficients of the equivariant zeta functions with signs :

\begin{prop} \label{lasttermszetafsignsfg} \begin{enumerate}
	\item If $p > q+1$, $\eta = \epsilon = + 1$, then, for all $M > 2k$, we have $\beta^G(A_{M}^{\xi}(f_{2k-1})) = \beta^G(A_{M}^{\xi}(g_k))$, and consequently $Z_{f_{2k-1}}^{G,\pm}(u,T) = Z_{g_{k}}^{G,\pm}(u,T)$. 
	\item If $k$ is even and if $p > q + 1$, $\eta = +1$ and $\epsilon = -1$, we have the equality $Z_{f_{2k-1}}^{G,\xi}(u,T) = Z_{g_{k}}^{G,\xi}(u,T)$ if and only if $\beta^G\left(\left\{ - x_2^{2k} + y^2 + \sum_{i = 1}^{K-1} y_i^2 = \xi \right\}\right) = \beta^G\left(\left\{ - x_1^{2k} + \sum_{i = 1}^{K} z_i^2 = \xi \right\}\right)$ (the former set is a subset of $\mathbb{R}^{K+1}$ equipped with the action of $G$ only changing the sign of $y$ and the latter set is a subset of $\mathbb{R}^{K+1}$ equipped with the action of $G$ only changing the sign of $x_1$).
	
%for all $M > 2k$, then we have the equality $\beta^G(A_{M}^{\xi}(f_{2k-1})) = \beta^G(A_{M}^{\xi}(g_k))$ if and only if $\beta^G\left(\left\{ - x_2^{2k} + y^2 + \sum_{i = 1}^{K-1} y_i^2 = \xi \right\}\right) = \beta^G\left(\left\{ - x_1^{2k} + \sum_{i = 1}^{K} z_i^2 = \xi \right\}\right)$ (the former set is a subset of $\mathbb{R}^{K+1}$ equipped with the action of $G$ changing only the sign of $y$ and the latter set is a subset of $\mathbb{R}^{K+1}$ equipped with the action of $G$ changing only the sign of $x_1$).
\end{enumerate}
\end{prop}

\begin{proof} Let $M$ be greater than $2k$. The system defining $A^{\xi}_M(f_{2k-1})$ is obtained by replacing $0$ by $\xi$ in the right member of the last equation of the system defining $A^0_M(f_{2k-1})$. Consequently, the same arguments works as in the proof of previous proposition \ref{aM02k} and $\beta^G(A_{M}^{\xi}(f_{2k-1})) = \beta^G(A_{M}^{\xi}(g_k))$ if and only if the contribution given by the very last equation provided by the computation is the same for $f_{2k-1}$ and $g_k$. 

As in the proof of proposition \ref{amfxi}, if $M$ is odd, this contribution is the equivariant virtual Poincar\'e series of an empty set, and if $M$ is even and not a multiple of $2k$, it is $\beta^G(Y_{p,q}^{\xi})$ : in both cases, $\beta^G(A_{M}^{\xi}(f_{2k-1})) = \beta^G(A_{M}^{\xi}(g_k))$. Finally, if $M$ is a multiple of $2k$, the respective contributions are $\beta^G(\{f_{2k-1} = \xi\})$ and $\beta^G(\{g_k = \xi\})$, hence the result by lemma \ref{fgxi} (see also the proofs of propositions \ref{eqa2kxi} and \ref{a2kxidepend}).
\end{proof}

\subsection{Conclusion}

As a conclusion, we summarize and gather the results of the previous paragraphs in the following theorem :

\begin{theo} Let $k \geq 2$. Suppose that the invariant germs
$$f_{2k-1} = \epsilon x_2^{2k} + \eta x_1^2 + Q \mbox{   and   } g_k = \epsilon x_1^{2k} + \eta' x_2^2 + Q'$$
have, up to permutation of all variables, the same quadratic part, with $p$ signs $+$ and $q$ signs~$-$.

\begin{enumerate}
	\item If  \begin{itemize}
		\item $p \leq q$, $\eta = + 1$ or $q \leq p$, $\eta = -1$,
	        \item $p = q+1$ or $q = p+1$,
	        \item $k$ is odd and if $p > q +1$, $\eta = +1$, $\epsilon = -1$ or $q > p+1$, $\eta = -1$, $\epsilon = +1$,
		      \end{itemize}
		  then $f_{2k-1}$ and $g_k$ are not $G$-blow-Nash equivalent.
	\item If $p > q+1$, $\eta = \epsilon = +1$ or $q  > p+1$, $\eta = \epsilon = - 1$, then $Z_{f_{2k-1}}^{G}(u,T) = Z_{g_{k}}^{G}(u,T)$ and $Z_{f_{2k-1}}^{G,\xi}(u,T) = Z_{g_{k}}^{G,\xi}(u,T)$.
	\item  \begin{itemize} 
		\item If $k$ is even and if $p > q + 1$, $\eta = +1$, $\epsilon = -1$, then $Z_{f_{2k-1}}^{G}(u,T) = Z_{g_{k}}^{G}(u,T)$. Furthermore, $Z_{f_{2k-1}}^{G,\xi}(u,T) = Z_{g_{k}}^{G,\xi}(u,T)$ if and only if $\beta^G\left(\left\{ - x_2^{2k} + y^2 + \sum_{i = 1}^{K-1} y_i^2 = \xi \right\}\right) = \beta^G\left(\left\{ - x_1^{2k} + \sum_{i = 1}^{K} z_i^2 = \xi \right\}\right)$.
	 \item If $k$ is even and if $q > p+1$, $\eta = -1$,  $\epsilon = +1$, then $Z_{f_{2k-1}}^{G}(u,T) = Z_{g_{k}}^{G}(u,T)$. Furthermore, $Z_{f_{2k-1}}^{G,\xi}(u,T) = Z_{g_{k}}^{G,\xi}(u,T)$ if and only if $\beta^G\left(\left\{ x_2^{2k} - y^2 - \sum_{i = 1}^{K-1} y_i^2 = \xi \right\}\right) = \beta^G\left(\left\{ x_1^{2k} - \sum_{i = 1}^{K} z_i^2 = \xi \right\} \right)$.
		\end{itemize}
\end{enumerate}
\end{theo}

\begin{rem}
\begin{enumerate}
	\item As one can notice from the computations, the fact that the equivariant Poincar\'e series of a given sphere is the same for any action of $G$ on it with a non-empty fixed point set (see remark \ref{equivpoinc}) induces equalities between coefficients of the respective equivariant zeta functions of $f_{2k-1}$ and $g_k$.
	\item If the equivariant virtual Poincar\'e series was proved to be an invariant under equivariant bijection with $\mathcal{AS}$ graph, this could allow to compute (and compare) the quantities $\beta^G\left(\left\{ - x_2^{2k} + y^2 + \sum_{i = 1}^{K-1} y_i^2 = \xi \right\}\right)$ and $\beta^G\left(\left\{ - x_1^{2k} + \sum_{i = 1}^{K} z_i^2 = \xi \right\}\right)$.
\end{enumerate}
\end{rem}

\section{The germs $C_k$ and $D_k$} \label{compCD}

In a second time, we plan to make progress towards the classification with respect to $G$-blow-Nash equivalence of the invariant germs of the families
$$h_k^{\epsilon_k}(x) := \pm x_1^2 + x_2^2 x_3 + \epsilon_k x_3^{k-1} + Q \mbox{ ~~and~~} r_k ^{\epsilon_k}(x) := x_1^2 x_2 + \epsilon_k x_2^k + \pm x_3^2+ Q',$$
where $\epsilon_k \in \{-1 ; +1\}$.

By the same arguments as in the introduction of section \ref{compAB}, we know that if two germs $h_k^{\epsilon_k}$ and $h_l^{\epsilon_l}$ are $G$-blow-Nash equivalent, they have the same quadratic part up to permutation of the variables $x_1, x_4, \ldots, x_n$ and, by \cite{GF-BNT} Proposition 3.11, that $k = l$ and $\epsilon_k = \epsilon_l$. Therefore, inside the family $D_k$, it remains to show that the germs
$$h_k^{\epsilon_k, +}(x) := + x_1^2 +  x_2^2 x_3 + \epsilon_k x_3^{k-1} + Q  \mbox{   and   } h_k^{\epsilon_k, -}(x) := - x_1^2 +  x_2^2 x_3 + \epsilon_k x_3^{k-1} + Q',$$
where $\epsilon_k \in \{-1 \, ; + 1\}$ and $+ x_1^2 + Q$ and $- x_1^2 + Q'$ are the same quadratic part up to permutation of the variables $x_1, x_4, \ldots , x_n$, are not $G$-blow-Nash equivalent.

As for the family $C_k$, if two germs $r_k^{\epsilon_k}$ and $r_l^{\epsilon_l}$ are $G$-blow-Nash equivalent, they have the same quadratic part up to permutation of the variables $x_3, \ldots , x_n$, $k = l$ and $\epsilon_k = \epsilon_l$.

On the other hand, if two germs $h_k^{\epsilon_k}$ and $r_{k'}^{\epsilon_{k'}}$ are $G$-blow-Nash equivalent then $k = k'+1$, $\epsilon_{k} = \epsilon_{k'}$ and $\pm x_1^2 + Q$ and $\pm x_2^2 + Q'$ are the same quadratic part up to permutation of all variables. As a consequence, we focus on the comparison of the germs
$$h_{k+1} = x_2^2 x_3 + \epsilon x_3^k + \eta x_1^2 + Q \mbox{   and   } r_k = x_1^2 x_2 + \epsilon x_2^k + \eta' x_3^2 + Q'$$
where $\epsilon, \eta, \eta' \in \{1,-1\}$ and $\eta x_1^2 + Q = \eta' x_3^2 + Q'$ up to permutation of all variables. 
\\

In the following, as we did for the families $A_k$ and $B_k$, we study and compare the respective equivariant zeta functions of $h_k$ and $r_k$ : using theorem \ref{eqzfinv}, this allows to extract further cases of non-G-blow-Nash equivalence.

\subsection{Computation of the first terms of the equivariant zeta functions} \label{firsttermscd}

Fix $k \geq 4$ and consider the invariant germ $h_k^{\epsilon, \eta}(x_1, \ldots, x_n) = \eta x_1^2 + x_2^2 x_3 + \epsilon x_3^{k-1} + Q$. Denote $x_2 = x$, $x_3 = z$ and $\eta x_1^2 + Q = Q_{p,q} =  \sum_{i=1}^p y_i^2 - \sum_{j=1}^q y_{p+j}^2$ ($G$ acts on the renamed coordinates via the involution n\textsuperscript{o}\ref{acty1} or n\textsuperscript{o}\ref{actyp1} depending on the sign of $\eta$), so that $h_k^{\epsilon, \eta}(x,z,y) = x^2z + \epsilon z^{k-1} + Q_{p,q}(y)$.

The following proposition gives the computed expressions for $\beta^G(A_m^0(h_k^{\epsilon, \eta}))$ (see the beginning of paragraph \ref{firsttermsab} for the definition of $A_m^0(h)$ for $h$ an invariant Nash germ) if $m < k-1$. The same expressions can be obtained for $\beta^G(A_m^0(r_{k-1}^{\epsilon}))$, providing $Y_{p,q}$ is equipped with the trivial action in this case.

\begin{prop} \label{a0mhr} Suppose $m < k-1$.

\begin{enumerate}
%	\item If $pq = 0$, then
%$$\beta^G(A^0_{m}(h_k^{\epsilon, \eta})) = 
%\begin{cases} 
%u^{3r+2 + (r+1)(p+q)} \frac{u^{r(p+q-1)} - 1}{u^{p+q-1} - 1} + \frac{u^{3r+3 +(r+1)(p+q)}}{u-1} \mbox{ if $m = 2r+1$,} \\
% u^{3r + (r+1)(p+q)} \frac{u^{(r-1)(p+q-1)} - 1}{u^{p+q-1} - 1}  + \frac{u^{3r+ 2 + r(p+q)}}{u-1} \mbox{ if $m = 2r$.}
%\end{cases}$$
	\item If $p+q = 1$, then
$$\beta^G(A^0_{m}(h_k^{\epsilon, \eta})) = 
\begin{cases} r u^{2m+1} + \frac{u^{4r+4}}{u-1} \mbox{ if $m = 2r+1$,} \\
(r-1) u^{2m+1} + \frac{u^{4r+2}}{u-1}  \mbox{ if $m = 2r$.}
\end{cases}
$$
	\item If $p + q \neq 1$, then
$$\beta^G(A^0_{m}(h_k^{\epsilon, \eta})) = 
\begin{cases} 
u^{3r+2 + (r+1)(p+q)} \frac{u^{r(p+q-1)} - 1}{u^{p+q-1} - 1} \left(\beta^G(Y_{p,q} \setminus \{0\}) + 1\right) + \frac{u^{3r+3 +(r+1)(p+q)}}{u-1} \mbox{ if $m = 2r+1$,} \\
 u^{3r + (r+1)(p+q)} \frac{u^{(r-1)(p+q-1)} - 1}{u^{p+q-1} - 1} \left(\beta^G(Y_{p,q} \setminus \{0\}) + 1\right)  + u^{3r+ 1 + r(p+q)} \beta^G(Y_{p,q}) \mbox{ if $m = 2r$.}
\end{cases}$$
\end{enumerate}

\end{prop}

\begin{proof} As in subsection \ref{firsttermsab}, we follow the computations of \cite{GF-BNT}, paying attention to their equivariance with respect to the considered actions of $G$.

Here, we write an arc $\gamma$ of $\mathcal{L}_m$ as
\begin{eqnarray*} \gamma(t) & = & (a_1 t + \cdots + a_{m} t^{m}, b_1 t + \cdots + b_{m} t^{m}, c_1^1 t + \cdots + c_{m}^1 t^{m}, \ldots, c_1^{p+q} t + \cdots + c_{m}^{p+q} t^{m}) \\
& = & \left( \begin{array}{c} a_1 \\ b_1 \\ c_1^1 \\ \vdots \\ c_1^{p+q} \end{array} \right) t + \cdots + \left( \begin{array}{c} a_{m} \\ b_m \\c_{m}^1 \\ \vdots \\ c_{m}^{p+q} \end{array} \right) t^{m} = \left( \begin{array}{c} a_1 \\ b_1\\ c_1 \end{array} \right) t + \cdots + \left( \begin{array}{c} a_{m} \\ b_1 \\ c_{m}\end{array} \right) t^{m}
\end{eqnarray*}
(the group $G$ acts on $\mathcal{L}_{m}$ changing the sign of the variables $c_i^1$, resp. $c_i^{p+1}$, in the case n\textsuperscript{o}\ref{acty1}, resp. n\textsuperscript{o}\ref{actyp1}).

We focus on the generic case $pq \neq 0$, $p+q \neq 1$. First suppose that $m$ is odd, $m = 2 r + 1$. Then an arc $\gamma$ of $\mathcal{L}_{m}$ belongs to $A^0_{m}(h_k^{\epsilon, \eta})$ if and only if
$$\begin{cases}
Q_{p,q}(c_1) = 0, \\
a_1^2 b_1 + \Phi_{p,q}(c_1, c_2) = 0, \\
a_1^2 b_2 + 2 a_1 a_2 b_1 + Q_{p,q} (c_2) + \Phi_{p,q} (c_1, c_3) = 0,\\
\cdots \\
\sum_{t=1}^{r-1} a_t^2 b_{2r-2t} + 2 \sum_{t=1}^{r-1} a_t \sum_{\delta=t+1}^{2r-(t+1)} a_{\delta} b_{2r-\delta-t} + Q_{p,q}(c_r) + \sum_{t= 1}^{r-1} \Phi_{p,q}(c_t,c_{2r-t}) = 0, \\
\sum_{t = 1}^{r} a_t^2 b_{2r+1-2t} + 2 \sum_{t=1}^{r-1} a_s \sum_{\delta=t+1}^{2r+1-(t+1)} a_{\delta} b_{2r+1-\delta-t} +\sum_{t= 1}^{r} \Phi_{p,q}(c_t,c_{2r+1-t}) = 0.
\end{cases}$$

Stratifying $A^0_{m}(h_k^{\epsilon, \eta})$ with the $G$-globally invariant subsets $\{c_1^1 = \ldots = c_1^{i-1} = 0, c_1^i \neq 0\}$, $i = 1, \ldots, p$, and $\{c_1^1 = \ldots = c_1^p = 0\} = \{c_1 = 0\}$, as we did in the proof of proposition \ref{amf}, we obtain, by additivity of the equivariant virtual Poincar\'e series, 
$$\beta^G(A^0_{m}(h_k^{\epsilon, \eta})) = u^{2 \times (2r+1) + 2r (p-1) + 2r q + 1} \beta^G(Y_{p,q} \setminus \{0\}) + \beta^G(A^0_{m}(h_k^{\epsilon, \eta}) \cap \{c_1 = 0\}),$$ 
the algebraic set $A^0_{m}(h_k^{\epsilon, \eta}) \cap \{c_1 = 0\}$ being described by the system
$$\begin{cases}
a_1^2 b_1 = 0, \\
a_1^2 b_2 + 2 a_1 a_2 b_1 + Q_{p,q} (c_2) = 0,\\
\cdots \\
\sum_{t=1}^{r-1} a_t^2 b_{2r-2t} + 2 \sum_{t=1}^{r-1} a_t \sum_{\delta=t+1}^{2r-(t+1)} a_{\delta} b_{2r-\delta-t} + Q_{p,q}(c_r) + \sum_{t= 2}^{r-1} \Phi_{p,q}(c_t,c_{2r-t}) = 0, \\
\sum_{t = 1}^{r} a_t^2 b_{2r+1-2t} + 2 \sum_{t=1}^{r-1} a_s \sum_{\delta=t+1}^{2r+1-(t+1)} a_{\delta} b_{2r+1-\delta-t} +\sum_{t= 2}^{r} \Phi_{p,q}(c_t,c_{2r+1-t}) = 0.
\end{cases}$$

Now, if $a_1 \neq 0$, then $b_1 = 0$ and the coordinates $b_2, \ldots, b_{2r-1}$ are determined by $a_1$ and the other variables (via an equivariant morphism), and thus
$$\beta^G(A^0_{m}(h_k^{\epsilon, \eta}) \cap \{c_1 = 0\}) = (u-1) \frac{u^{[2r + 2 + 2r (p+q)]+ 1}}{u-1} + \beta^G(A^0_{m}(h_k^{\epsilon, \eta}) \cap \{c_1 = 0, a_1 = 0\}).$$
If $c_1 = 0$ and $a_1 = 0$, the remaining coordinates verify the system
$$\begin{cases}
Q_{p,q} (c_2) = 0,\\
a_2^2 b_1 + \Phi_{p,q}(c_2, c_3) = 0, \\
a_2^2 b_2 + 2 a_2 a_3 b_1 + Q'_{p,q} (c_3) + \Phi_{p,q} (c_2, c_4) = 0,\\
\cdots \\
\sum_{t=2}^{r-1} a_t^2 b_{2r-2t} + 2 \sum_{t=2}^{r-1} a_t \sum_{\delta=t+1}^{2r-(t+1)} a_{\delta} b_{2r-\delta-t} + Q_{p,q}(c_r) + \sum_{t= 2}^{r-1} \Phi_{p,q}(c_t,c_{2r-t}) = 0, \\
\sum_{t = 2}^{r} a_t^2 b_{2r+1-2t} + 2 \sum_{t=2}^{r-1} a_s \sum_{\delta=t+1}^{2r+1-(t+1)} a_{\delta} b_{2r+1-\delta-t} +\sum_{t= 2}^{r} \Phi_{p,q}(c_t,c_{2r+1-t}) = 0.
\end{cases}$$
Notice that the vector $c_m$ as well as the variables $a_m$, $b_{m-1}$ and $b_m$ are free and that, if we rename the variables, these equations define the set $A^0_{m-2}(h_k^{\epsilon, \eta})$, so that 
$$\beta^G(A^0_{m}(h_k^{\epsilon, \eta}) \cap \{c_1 = 0, a_1 = 0\}) = u^{3 + p + q} \beta^G(A^0_{m-2}(h_k^{\epsilon, \eta})).$$

By an induction process, we then obtain  
\begin{eqnarray*}
\beta^G(A^0_{m}(h_k^{\epsilon, \eta})) & = & \beta^G(Y_{p,q} \setminus \{0\}) \left[ \sum_{t = 0}^{r-1} u^{t(3+p+q)} u^{2 \times (m - 2t) + (m-2t - 1)(p+q-1)+1} \right] \\
& & + \left[ \sum_{t = 0}^{r-1} u^{t(3+p+q)} u^{(m-1-2t) (p+q + 1)+ 3} \right] + u^{(r-1)(3+p+q)} \beta^G(A^0_{3}(h_k^{\epsilon, \eta}) \cap \{c_1 = 0, a_1 = 0\}),
\end{eqnarray*}
the equations for $A^0_{3}(h_k^{\epsilon, \eta}) \cap \{c_1 = 0, a_1 = 0\}$ becoming trivial. As a consequence (notice that for all $t = 0, \ldots, r-1$, $2 \times (m - 2t) + (m-2t - 1)(p+q-1)+1 = (m-1-2t) (p+q + 1)+ 3$),
$$\beta^G(A^0_{m}(h_k^{\epsilon, \eta})) = u^{3r+2 + (r+1)(p+q)} \frac{u^{r(p+q-1)} - 1}{u^{p+q-1} - 1} \left(\beta^G(Y_{p,q} \setminus \{0\}) + 1\right) + \frac{u^{3r+3 +(r+1)(p+q)}}{u-1}.$$ 
\\

%\begin{eqnarray*}
%\beta^G(A^0_{m}(h_k^{\epsilon, \eta})) & = & u^{3r+2 + (r+1)(p+q)}  \beta^G(Y_{p,q} \setminus \{0\}) \frac{u^{r(p+q-1)} - 1}{u^{p+q-1} - 1} \\
%& & + u^{3r+2 + (r+1)(p+q)} \frac{u^{r(p+q-1)} - 1}{u^{p+q-1} - 1} + \frac{u^{3r+3 +(r+1)(p+q)}}{u-1}
%\end{eqnarray*}

If $m$ is even, $m = 2r$, the system defining $A^0_{m}(h_k^{\epsilon, \eta})$ is 
$$\begin{cases}
Q_{p,q}(c_1) = 0, \\
a_1^2 b_1 + \Phi_{p,q}(c_1, c_2) = 0, \\
\cdots \\
\sum_{t = 1}^{r-1} a_t^2 b_{2r-1-2t} + 2 \sum_{t=1}^{r-2} a_s \sum_{\delta=t+1}^{2r-1-(t+1)} a_{\delta} b_{2r-1-\delta-t} +\sum_{t= 1}^{r-1} \Phi_{p,q}(c_t,c_{2r-1-t}) = 0, \\
\sum_{t=1}^{r-1} a_t^2 b_{2r-2t} + 2 \sum_{t=1}^{r-1} a_t \sum_{\delta=t+1}^{2r-(t+1)} a_{\delta} b_{2r-\delta-t} + Q_{p,q}(c_r) + \sum_{t= 1}^{r-1} \Phi_{p,q}(c_t,c_{2r-t}) = 0,
\end{cases}$$
and we have
\begin{eqnarray*}
\beta^G(A^0_{m}(h_k^{\epsilon, \eta})) & = & \beta^G(Y_{p,q} \setminus \{0\}) \left[ \sum_{t = 0}^{r-2} u^{t(3+p+q)} u^{2 \times (m - 2t) + (m-2t - 1)(p+q-1)+1} \right] \\
& & + \left[ \sum_{t = 0}^{r-2} u^{t(3+p+q)} u^{(m-1-2t) (p+q + 1)+ 3} \right] + u^{(r-1)(3+p+q)} \beta^G(A^0_{2}(h_k^{\epsilon, \eta}))
\end{eqnarray*}
Because $A^0_{2}(h_k^{\epsilon, \eta})$ is described by the equation $Q_{p,q}(c_1) = 0$ and since the vector $c_2$ as well as the variables $a_1,a_2, b_1$ and $b_2$ are free, we obtain
$$\beta^G(A^0_{m}(h_k^{\epsilon, \eta}))  =  u^{3r + (r+1)(p+q)} \frac{u^{(r-1)(p+q-1)} - 1}{u^{p+q-1} - 1} \left(\beta^G(Y_{p,q} \setminus \{0\}) + 1\right)  + u^{3r+ 1 + r(p+q)} \beta^G(Y_{p,q}).$$

%\begin{eqnarray*}
%\beta^G(A^0_{m}(h_k^{\epsilon, \eta}))  & = & u^{3r + (r+1)(p+q)}  \beta^G(Y_{p,q} \setminus \{0\}) \frac{u^{(r-1)(p+q-1)} - 1}{u^{p+q-1} - 1} \\
%& & u^{3r + (r+1)(p+q)} \frac{u^{(r-1)(p+q-1)} - 1}{u^{p+q-1} - 1}  + u^{3r+ 1 + r(p+q)} \beta^G(Y_{p,q})
%\end{eqnarray*}
 
\end{proof}

As for $\beta^G(A^{\xi}_{m}(h_k^{\epsilon, \eta}))$, we have the following expressions if $m< k+1$ :

\begin{prop} \label{amxihr} Suppose $m < k-1$.

\begin{enumerate}
	\item If $(p,q) = (0,1)$, then
$$\beta^G(A^{\xi}_{m}(h_k^{\epsilon, \eta})) = 
\begin{cases} r u^{2m+1} \mbox{ if $m = 2r+1$,} \\
(r-1) u^{2m+1} + u^{4r+1} \beta^G(Y^{\xi}_{0,1}) \mbox{ if $m = 2r$.}
\end{cases}
$$
	\item If $(p,q)= (1,0)$, then
$$\beta^G(A^{\xi}_{m}(h_k^{\epsilon, \eta})) = 
\begin{cases} r u^{2m+1} \mbox{ if $m = 2r+1$,} \\
(r-1) u^{2m+1} + u^{4r+1} \beta^G(Y^{\xi}_{1,0})  \mbox{ if $m = 2r$.}
\end{cases}
$$
	\item If $p + q \neq 1$, then
$$\beta^G(A^{\xi}_{m}(h_k^{\epsilon, \eta})) = 
\begin{cases} 
u^{3r+2 + (r+1)(p+q)} \frac{u^{r(p+q-1)} - 1}{u^{p+q-1} - 1} \left(\beta^G(Y_{p,q} \setminus \{0\}) + 1\right) \mbox{ if $m = 2r+1$,} \\
 u^{3r + (r+1)(p+q)} \frac{u^{(r-1)(p+q-1)} - 1}{u^{p+q-1} - 1} \left(\beta^G(Y_{p,q} \setminus \{0\}) + 1\right)  + u^{3r+ 1 + r(p+q)} \beta^G(Y_{p,q}^{\xi}) \mbox{ if $m = 2r$.}
\end{cases}$$
\end{enumerate}

\end{prop}

\begin{proof} If we keep the notations of the proof of proposition \ref{a0mhr}, the system defining $A^{\xi}_{m}(h_k^{\epsilon, \eta})$ is obtained by replacing $0$ by $\xi$ in the right member of the last of the equations describing $A^{0}_{m}(h_k^{\epsilon, \eta})$. Furthermore, the system for $A^{\xi}_{3}(h_k^{\epsilon, \eta}) \cap \{c_1 = 0, a_1 = 0\}$ has no solution, whereas $A^{\xi}_{2}(h_k^{\epsilon, \eta})$ is described by the equation $Q_{p,q}(c_1) = \xi$.
\end{proof}

We are now able to show that the germs $h_k^{\epsilon, +}$ and $h_k^{\epsilon, -}$ are not $G$-blow-Nash equivalent :

\begin{cor} \label{hk+hk-dif}
Let $k \geq 4$. Suppose that the invariant Nash germs
$$h_k^{\epsilon, +}(x) := + x_1^2 + x_2^2 x_3 + \epsilon x_3^{k-1} + Q \mbox{   and   } h_k^{\epsilon, -}(x) := - x_1^2 + x_2^{2}x_3 + \epsilon x_3^{k-1}  + Q'$$
have the same quadratic part up to permutation of the variables $x_1, x_4, \ldots , x_n$. Then they are not $G$-blow-Nash equivalent.
\end{cor}

\begin{proof} We compare $\beta^G(A_2(h_k^{\epsilon, +}))$ and $\beta^G(A_2(h_k^{\epsilon, -}))$. Because $\beta^G(A_2(h_k^{\epsilon, \eta})) = \beta^G({}^0\!A_2(h_k^{\epsilon, \eta})) - \beta^G(A^0_2(h_k^{\epsilon, \eta}))$, $\beta^G({}^0\!A_2(h_k^{\epsilon, \eta})) = u^{n} \beta^G(A^0_1(h_k^{\epsilon, \eta}))$ (by proposition \ref{oamh}) and $A^0_1(h_k^{\epsilon, \eta}) = \mathcal{L}_1$, we are reduced to compare the quantities $\beta^G(A^0_2(h_k^{\epsilon, +}))$ and $\beta^G(A^0_2(h_k^{\epsilon, -}))$.

Now, if $p$ denotes the number of signs $+$ and $q$ the number of signs $-$ in the quadratic part of $h_k^{\epsilon, +}$ and $h_k^{\epsilon, -}$ (notice that $p+q \neq 1$), we have, by proposition \ref{a0mhr},
$$\beta^G(A^0_2(h_k^{\epsilon, \eta}))) = u^{4+p+q} \beta^G(Y_{p,q})$$
Consequently, if $p \neq q$, we can use the same arguments as in the proof of corollary \ref{fpmnot} to conclude that the naive equivariant zeta functions of $h_k^{\epsilon, +}$ and $h_k^{\epsilon, -}$ are different and therefore that the germs $h_k^{\epsilon, +}$ and $h_k^{\epsilon, -}$ are not $G$-blow-Nash equivalent. 

If $p = q$, we compare $\beta^G(A_2^{+1}(h_k^{\epsilon, +}))$ and $\beta^G(A_2^{+1}(h_k^{\epsilon, -}))$. Since, by proposition \ref{amxihr}, 
$$\beta^G(A_2^{+1}(h_k^{\epsilon, +})) = u^{4+2p} \beta^G(Y_{p,q}^{+1}),$$
we can in this case as well use the arguments of the proof of corollary \ref{fpmnot}, in order to conclude that the equivariant zeta functions with signs $+$ of $h_k^{\epsilon, +}$ and $h_k^{\epsilon, -}$ are different.

\end{proof}

Using again the formulae of propositions \ref{a0mhr} and \ref{amxihr}, we then extract cases for which the germs $h_{k+1}$ and $r_k$ are not $G$-blow-Nash equivalent :

\begin{cor} Let $k \geq 3$. Suppose that the invariant germs
$$h_{k+1} = x_2^2 x_3 + \epsilon x_3^k + \eta x_1^2 + Q \mbox{   and   } r_k = x_1^2 x_2 + \epsilon x_2^k + \eta' x_3^2 + Q'$$
have, up to permutation of all variables, the same quadratic part, with $p$ signs $+$ and $q$ signs~$-$.

If $p \leq q$ and $\eta = +1$ or $q \leq p$ and $\eta = -1$, then $h_{k+1}$ and $r_k$ are not $G$-blow-Nash equivalent.

If $p = q+1$ or $q = p+1$, then $h_{k+1}$ and $r_k$ are not $G$-blow-Nash equivalent. 
\end{cor}

\begin{proof}  For the first point, focus for instance on the case $p \leq q$ and $\eta = +1$. As in the proof of corollary \ref{hk+hk-dif}, we consider $\beta^G(A^0_2(h_{k+1})) = u^{4+p+q} \beta^G(Y_{p,q})$ and $\beta^G(A^0_2(r_{k})) = u^{4+p+q} \beta^G(Y_{p,q})$. Since the action of $G$ on the former set $Y_{p,q}$ is the action n\textsuperscript{o}\ref{acty1} and the action on the latter set $Y_{p,q}$ is the trivial action, we obtain $\beta^G(A_2(h_{k+1})) \neq \beta^G(A_2(r_{k}))$.

For the second point, assume for instance $p = q+1$. Suppose furthermore that $\eta = +1$ and consider the quantities $\beta^G(A^{+1}_2(h_{k+1})) = u^{4+p+q} \beta^G(Y^{+1}_{p,q})$ and $\beta^G(A^{+1}_2(r_{k})) = u^{4+p+q} \beta^G(Y^{+1}_{p,q})$. By proposition \ref{yxi}, $\beta^G(Y_{p,q}^{+1}) = \frac{1}{u-1} \left(\beta^G(Y_{p,p}) - \beta^G(Y_{p,q})\right)$. Since $q < p$ and $\eta = +1$, the quantity $\beta^G(Y_{p,q})$ is the same for $h_{k+1}$ and $r_k$, while the quantities $\beta^G(Y_{p,p})$ are different (see proposition \ref{ypq}). As a consequence, $\beta^G(A^{+1}_2(h_{k+1})) \neq \beta^G(A^{+1}_2(r_{k}))$.
\end{proof}

From now, we are going to study the other coefficients of the equivariant zeta functions of $h_{k+1}$ and $r_k$ in the remaining cases, that is if $p > q+1$ and $\eta = +1$ or $q > p+1$ and $\eta = -1$. Notice that, in these cases, the quantities $\beta^G(Y_{p,q})$ and $\beta^G(Y^{\xi}_{p,q})$ are identical for $h_{k+1}$ and $r_k$.

\subsection{Computation of $\beta^G(A_k(h_{k+1}))$ and $\beta^G(A_k(r_{k}))$} \label{akhr}
 
Assuming that the Nash germs $h_{k+1}$ and $r_k$ have the same quadratic part $Q_{p,q}$, with $p > q+1$ and $\eta = +1$ or $q > p+1$ and $\eta = -1$, we first compute the coefficients $\beta^G(A_k(h_{k+1}))$ and $\beta^G(A_k(r_{k}))$ of their respective naive equivariant zeta functions. Having in mind proposition \ref{oamh}, we actually give formulae for $\beta^G(A^0_k(h_{k+1}))$ and $\beta^G(A^0_k(r_{k}))$ :

\begin{prop} \label{a0khr} Suppose $k \geq 3$. Then
$$\beta^G(A^0_{k}(h_{k+1})) = \begin{cases} u^{3l+2 + (l+1)(p+q)} \frac{u^{l(p+q-1)} - 1}{u^{p+q-1} - 1} \left(\beta^G(Y_{p,q} \setminus \{0\})\right) \\ + u^{3l+1 + (l+2)(p+q)} \frac{u^{(l-1)(p+q-1)} - 1}{u^{p+q-1} - 1} + u^{3l+1 + (l+1)(p+q)} \beta^G(\{h_{k+1}(x_2,x_3,0) = 0\}) \mbox{ if $k = 2l+1$,}\\
u^{3l + (l+1)(p+q)} \frac{u^{(l-1)(p+q-1)} - 1}{u^{p+q-1} - 1} \left(\beta^G(Y_{p,q} \setminus \{0\}) + 1\right) \\ ~~~~~ + u^{3l + l(p+q)} \beta^G(\{h_{k+1}(0,x_3,y) = 0\}) \mbox{ if $k = 2l$,}
\end{cases}$$
and
$$\beta^G(A^0_{k}(r_{k})) = \begin{cases} u^{3l+2 + (l+1)(p+q)} \frac{u^{l(p+q-1)} - 1}{u^{p+q-1} - 1} \left(\beta^G(Y_{p,q} \setminus \{0\})\right) \\ ~~~~ + u^{3l+1 + (l+2)(p+q)} \frac{u^{(l-1)(p+q-1)} - 1}{u^{p+q-1} - 1} + u^{3l+1 + (l+1)(p+q)} \beta^G(\{r_{k}(x_1,x_2,0) = 0\}) \mbox{ if $k = 2l+1$,}\\
u^{3l + (l+1)(p+q)} \frac{u^{(l-1)(p+q-1)} - 1}{u^{p+q-1} - 1} \left(\beta^G(Y_{p,q} \setminus \{0\}) + 1\right) + u^{3l + l(p+q)} \beta^G(\{r_{k}(0,x_2,y) = 0\}) \mbox{ if $k = 2l$.}
\end{cases}$$
\end{prop}

\begin{proof} We do the computations for $\beta^G(A^0_{k}(h_{k+1}))$.

First suppose $k$ to be odd, $k = 2l+1$. Keeping the notations of the proof of proposition \ref{a0mhr}, the set $A^0_k(h_{k+1})$ is defined by the system
$$\begin{cases}
Q_{p,q}(c_1) = 0, \\
a_1^2 b_1 + \Phi_{p,q}(c_1, c_2) = 0, \\
a_1^2 b_2 + 2 a_1 a_2 b_1 + Q_{p,q} (c_2) + \Phi_{p,q} (c_1, c_3) = 0,\\
\cdots \\
\sum_{t=1}^{l-1} a_t^2 b_{2l-2t} + 2 \sum_{t=1}^{l-1} a_t \sum_{\delta=t+1}^{2l-(t+1)} a_{\delta} b_{2l-\delta-t} + Q_{p,q}(c_l) + \sum_{t= 1}^{l-1} \Phi_{p,q}(c_t,c_{2l-t}) = 0, \\
\epsilon b_1^{2l+1}+ \sum_{t = 1}^{l} a_t^2 b_{2l+1-2t} + 2 \sum_{t=1}^{l-1} a_s \sum_{\delta=t+1}^{2l+1-(t+1)} a_{\delta} b_{2l+1-\delta-t} +\sum_{t= 1}^{l} \Phi_{p,q}(c_t,c_{2l+1-t}) = 0.
\end{cases}$$

Proceeding as in the proof of proposition \ref{a0mhr} (see also the proof of proposition \ref{a02k}), we obtain
\begin{eqnarray*}
\beta^G(A^0_{k}(h_{k+1})) & = & u^{3l+2 + (l+1)(p+q)} \frac{u^{l(p+q-1)} - 1}{u^{p+q-1} - 1} \left(\beta^G(Y_{p,q} \setminus \{0\})\right) \\
 &&+ \sum_{t = 0}^{l-2} u^{t(3+p+q)} u^{(k-1-2t) (p+q + 1)+ 3} + u^{(l-1)(3+p+q)} \beta^G(S^0_3),
\end{eqnarray*}
if $S^0_3$ denotes the algebraic set defined by the equation $\epsilon b_1^{2l+1} + a_1^2 b_1 = 0$, the variables $a_2, a_3$, $b_2,b_3$ as well as the vectors $c_2,c_3$ being free. Hence the desired result.
\\

If we suppose $k$ even, $k = 2l$, the set $A^0_k(h_{k+1})$ is described by the system
$$\begin{cases}
Q_{p,q}(c_1) = 0, \\
a_1^2 b_1 + \Phi_{p,q}(c_1, c_2) = 0, \\
\cdots \\
\sum_{t = 1}^{l-1} a_t^2 b_{2l-1-2t} + 2 \sum_{t=1}^{l-2} a_s \sum_{\delta=t+1}^{2l-1-(t+1)} a_{\delta} b_{2l-1-\delta-t} +\sum_{t= 1}^{l-1} \Phi_{p,q}(c_t,c_{2l-1-t}) = 0, \\
\epsilon b_1^{2l} + \sum_{t=1}^{l-1} a_t^2 b_{2l-2t} + 2 \sum_{t=1}^{l-1} a_t \sum_{\delta=t+1}^{2l-(t+1)} a_{\delta} b_{2l-\delta-t} + Q_{p,q}(c_l) + \sum_{t= 1}^{l-1} \Phi_{p,q}(c_t,c_{2l-t}) = 0,
\end{cases}$$
and we have
$$\beta^G(A^0_{k}(h_{k+1}))  =  u^{3l + (l+1)(p+q)} \frac{u^{(l-1)(p+q-1)} - 1}{u^{p+q-1} - 1} \left(\beta^G(Y_{p,q} \setminus \{0\}) + 1\right)  + u^{(l-1)(3+p+q)} \beta^G(S_2^0),$$
where $S_2^0$ is the algebraic set given by the equation $\epsilon b_1^{2l} + Q_{p,q}(c_1) = 0$, with free variables $a_1, a_2,b_2$ and free vector $c_2$.
\end{proof}
 
Since, in our present framework, the quantity $\beta^G(Y_{p,q})$ is the same for $h_{k+1}$ and $g_k$, we are reduced to study the equivariant virtual Poincar\'e series of the $G$-algebraic sets $\{h_{k+1}(x_2,x_3,0) = 0\}$ and $\{r_{k}(x_1,x_2,0) = 0\}$ if $k$ is odd, resp. $\{h_{k+1}(0,x_3,y) = 0\}$ and $\{r_{k}(0,x_2,y) = 0\}$ if $k$ is even.

Notice that, if $k$ is even, $\beta^G(\{h_{k+1}(0,x_3,y) = 0\})$ and $\beta^G(\{r_{k}(0,x_2,y) = 0\})$ have been already computed in lemma \ref{fg2k0} : if $k$ is even and if $p > q+1$ and $\eta = +1$ or $q > p+1$ and $\eta = -1$, the equivariant virtual Poincar\'e series $\beta^G(\{h_{k+1}(0,x_3,y) = 0\})$ and $\beta^G(\{r_{k}(0,x_2,y) = 0\})$ are equal and therefore $\beta^G(A_k(h_{k+1}))= \beta^G(A_k(r_{k}))$.

Now, in the next lemma, we compute $\{h_{k+1}(x_2,x_3,0) = 0\}$ and $\{r_{k}(x_1,x_2,0) = 0\}$ if $k$ is odd, $k = 2l+1$ : 
 
\begin{lem}
We have
$$\beta^G(\{h_{2l+2}(x_2,x_3,0) = 0\}) =  \beta^G(\{x_2^2 + \epsilon x_3^{2} = 0\}) - \frac{u}{u-1} +\frac{u^2}{u-1},$$
where the latter set is considered as an algebraic subset of $\mathbb{R}^2$ on which the group $G$ acts trivially, and
$$\beta^G(\{r_{2l+1}(x_1,x_2,0) = 0\}) = \beta^G(\{x_1^2 + \epsilon x_2^{2} = 0\}) - \frac{u}{u-1} + \frac{u^2}{u-1},$$ 
where the latter set is considered as an algebraic subset of $\mathbb{R}^2$ on which the group $G$ acts only changing the sign of the coordinate $x_1$.
\end{lem} 
 
\begin{proof} We make the computation for $\beta^G(\{r_{2l+1}(x_1,x_2,0) = 0\})$.

Consider the equation $x_1^2 x_2 + \epsilon x_2^{2l+1} = 0$. If $x_2 \neq 0$, it is equivalent to $x_1^2 + \epsilon x_2^{2l} = 0$, and if $x_2 = 0$, it becomes trivial. Consequently,
$$\beta^G(\{r_{2l+1}(x_1,x_2,0) = 0\}) = \beta^G(\{x_1^2 + \epsilon x_2^{2l} = 0 \} \setminus \{(0,0)\}) + \frac{u^2}{u-1},$$
and we use lemma \ref{fg2k0} to write $\beta^G(\{x_1^2 + \epsilon x_2^{2l} = 0 \}) = \beta^G(\{x_1^2 + \epsilon x_2^{2} = 0 \}) - (l-1) \beta^G(\{x_1^2 = 0\}) + (l-1) \beta^G(\{(0,0)\}) = \beta^G(\{x_1^2 + \epsilon x_2^{2} = 0 \})$ (recall also that the equivariant virtual Poincar\'e series of a point is $\frac{u}{u-1}$).
\end{proof} 

If $\epsilon = +1$, the sets $\{x_2^2 + x_3^{2} = 0\})$ and $\{x_1^2 + x_2^{2} = 0\}$ are both reduced to a single point. On the other hand, if $\epsilon = -1$,  we have $\beta^G(\{x_2^2 - x_3^{2} = 0\}) = \frac{2u^2-u}{u-1}$ whereas $\beta^G(\{x_1^2 - x_2^{2} = 0\}) = \frac{u^2-u+1}{u-1}$ (see proposition \ref{ypq}). As a consequence :
 
\begin{cor} If $k$ is odd and if $\epsilon = -1$, the germs $h_{k+1}$ and $r_k$ are not $G$-blow-Nash equivalent.
\end{cor} 
 
If $p > q+1$ and $\eta = +1$ or $q > p+1$ and $\eta = -1$, and if $k$ is even or $k$ is odd and $\epsilon = +1$, the coefficients $\beta^G(A_k(h_{k+1}))$ and $\beta^G(A_k(r_{k}))$ of the respective naive equivariant zeta functions of $h_{k+1}$ and $r_k$ are equal. We are then led to look at the coefficients $\beta^G(A^{\xi}_k(h_{k+1}))$ and $\beta^G(A^{\xi}_k(r_{k}))$ of their respective equivariant zeta functions with signs.

\subsection{Computation of $\beta^G(A^{\xi}_k(h_{k+1}))$ and $\beta^G(A^{\xi}_k(r_{k}))$}
 
For the cases listed above, we consider the quantities $\beta^G(A^{\xi}_k(h_{k+1}))$ and $\beta^G(A^{\xi}_k(r_{k}))$, expressed by the following formulae (just follow the steps of computation of the proof of proposition \ref{a0khr})~:

\begin{prop} Suppose $k \geq 3$. Then
$$\beta^G(A^{\xi}_{k}(h_{k+1})) = \begin{cases} u^{3l+2 + (l+1)(p+q)} \frac{u^{l(p+q-1)} - 1}{u^{p+q-1} - 1} \left(\beta^G(Y_{p,q} \setminus \{0\})\right) \\ + u^{3l+1 + (l+2)(p+q)} \frac{u^{(l-1)(p+q-1)} - 1}{u^{p+q-1} - 1} + u^{3l+1 + (l+1)(p+q)} \beta^G(\{h_{k+1}(x_2,x_3,0) = \xi\}) \mbox{ if $k = 2l+1$,}\\
u^{3l + (l+1)(p+q)} \frac{u^{(l-1)(p+q-1)} - 1}{u^{p+q-1} - 1} \left(\beta^G(Y_{p,q} \setminus \{0\}) + 1\right) \\ ~~~~~ + u^{3l + l(p+q)} \beta^G(\{h_{k+1}(0,x_3,y) = \xi\}) \mbox{ if $k = 2l$,}
\end{cases}$$
and
$$\beta^G(A^{\xi}_{k}(r_{k})) = \begin{cases} u^{3l+2 + (l+1)(p+q)} \frac{u^{l(p+q-1)} - 1}{u^{p+q-1} - 1} \left(\beta^G(Y_{p,q} \setminus \{0\})\right) \\ ~~~~ + u^{3l+1 + (l+2)(p+q)} \frac{u^{(l-1)(p+q-1)} - 1}{u^{p+q-1} - 1} + u^{3l+1 + (l+1)(p+q)} \beta^G(\{r_{k}(x_1,x_2,0) = \xi\}) \mbox{ if $k = 2l+1$,}\\
u^{3l + (l+1)(p+q)} \frac{u^{(l-1)(p+q-1)} - 1}{u^{p+q-1} - 1} \left(\beta^G(Y_{p,q} \setminus \{0\}) + 1\right) + u^{3l + l(p+q)} \beta^G(\{r_{k}(0,x_2,y) = \xi\}) \mbox{ if $k = 2l$.}
\end{cases}$$
\end{prop}

If $k$ is even, we can use the formulae of lemma \ref{fgxi} and the same arguments as in the proofs of propositions \ref{eqa2kxi} and \ref{a2kxidepend} in order to establish the following facts :

\begin{prop} \label{akxihreven} Suppose $k$ is even, $k = 2l \geq 4$.
\begin{enumerate}
\item If $p > q+1$ and $\eta = \epsilon = + 1$ or $q > p+1$ and $\eta = \epsilon = - 1$, then $\beta^G(A^{\xi}_{k}(h_{k+1})) = \beta^G(A^{\xi}_{k}(r_{k}))$.
\item \begin{itemize}
	\item If $p > q + 1$, $\eta = +1$ and $\epsilon = -1$, the equality $\beta^G(A_{k}^{\xi}(h_{k+1})) = \beta^G(A_{k}^{\xi}(r_k))$ is true if and only if the equivariant virtual Poincar\'e series of the algebraic subsets $\left\{ - x_3^{2l} + y^2 + \sum_{i = 1}^{K-1} y_i^2 = \xi \right\} \subset \mathbb{R}^{K+1}$, $K := p-q$, equipped with the action of $G$ changing only the sign of $y$, and $\left\{ - x_2^{2l} + \sum_{i = 1}^{K} z_i^2 = \xi \right\}\subset \mathbb{R}^{K+1}$, equipped with the trivial action of $G$, are equal.
 	\item If $q > p+1$, $\eta = -1$ and $\epsilon = +1$, we have $\beta^G(A_{k}^{\xi}(h_{k+1})) = \beta^G(A_{k}^{\xi}(r_k))$ if and only if $\beta^G\left(\left\{ x_3^{2l} - y^2 - \sum_{i = 1}^{K-1} y_i^2 = \xi \right\}\right) = \beta^G\left(\left\{ x_2^{2l} - \sum_{i = 1}^{K} z_i^2 = \xi \right\} \right)$.
	\end{itemize}
\end{enumerate}
\end{prop}

If $k$ is odd, $k = 2l+1$ and $\epsilon = +1$, and if $p > q+1$ and $\eta = +1$ or $q > p+1$ and $\eta = -1$, we are reduced to compare $\beta^G(\{h_{k+1}(x_2,x_3,0) = \xi\}) = \beta^G(\{x_2^2 x_3 + x_3^{2l+1} = \xi\})$ and $\beta^G(\{r_{k}(x_1,x_2,0) = \xi\}) = \beta^G(\{x_1^2 x_2 + x_2^{2l+1} = \xi \})$. We are going to show that these two quantities are equal and therefore :

\begin{prop} \label{axihkkoddepsplus} If $p > q+1$ and $\eta = +1$ or $q > p+1$ and $\eta = -1$, and if $k$ is odd and $\epsilon = +1$, then $\beta^G(A^{\xi}_{k}(h_{k+1})) = \beta^G(A^{\xi}_{k}(r_{k}))$. 
\end{prop}
 
\begin{proof} We compute the equivariant virtual Poincar\'e series of the nonsingular curve $C := \{x_0^2 y_0 + y_0^{2l+1} = \xi \}$ of $\mathbb{R}^2$, on which the group $G$ acts only changing the sign of the first coordinate $x_0$, resp. trivially. 
\\

First suppose the action of $G$ is the former one. Suppose also $l \geq 2$. We equivariantly compactify $C$ in the projective space $\mathbb{P}^2(\mathbb{R})$ with homogeneous coordinates $[X : Y : Z]$, on which $G$ acts via the involution $[X : Y : Z] \mapsto [-X : Y : Z] = [X : -Y : -Z]$. We denote $\Gamma := \{X^2 Y Z^{2l-2} + Y^{2l+1} = \xi Z^{2l+1} \}$ this compactification and $p := [1 : 0 : 0]$ the point at infinity.

The equivariant compactification $\Gamma$ is singular at the fixed point $p$ as one can see in the globally invariant chart $X \neq 0$. If $(y_0,z_0)$ are the coordinates in this chart, the group $G$ acting via the involution $(y_0,z_0) \mapsto (-y_0,-z_0)$, we denote by $C'$ the curve $\Gamma \cap \{X \neq 0\} = \{y_0 z_0^{2l-2} + y_0^{2l+1} = \xi z_0^{2l+1}\}$ (the point at infinity is the fixed point $q = [0 : \xi : 1]$ of $C$).

Equivariantly blowing-up $\Gamma$ at $p$ resolves the singularity : in the chart $y_0 = u_0 v_0$, $z_0 = v_0$, where the action of $G$ is given by $(u_0,v_0) \mapsto (u_0, - v_0)$, the equation of the strict transform is $u_0 + u_0^{2l+1} v_0^2 - \xi v_0^2 = 0$ and it intersects the exceptional divisor at the single point $p_0$ with coordinates $(u_0, v_0) = (0,0)$. The resolved compact $G$-variety, denoted by $\widetilde{\Gamma}$, is equivariantly homeomorphic to a circle equipped with an action of $G$ fixing the two points $p_0$ and $q$.

As a conclusion, we have
$$\beta^G(C) = \beta(\Gamma \setminus \{p\}) = \beta^G(\widetilde{\Gamma} \setminus \{p_0\}) = \beta^G(\widetilde{\Gamma}) - \beta^G(\{p_0\}) = u + 2 \frac{u}{u-1} - \frac{u}{u-1} = \frac{u^2}{u-1}$$
(see remark \ref{equivpoinc}). 

If $l =1$, the point $p$ of $\Gamma$ is not singular and $\Gamma$ is a compact nonsingular $G$-variety equivariantly homeomorphic to a circle with two fixed points $p$ and $q$.
\\
 
If now we suppose that the affine space $\mathbb{R}^2$ with coordinates $(x_0,y_0)$ is equipped with the trivial action of $G$, we will obtain the same expression for $\beta^G(C)$ since the equivariant homology of a circle is the same as soon as there is at least one fixed point.
\\

Consequently, the equivariant virtual Poincar\'e series $\beta^G(\{h_{k+1}(x_2,x_3,0) = \xi\}) = \beta^G(\{x_2^2 x_3 + x_3^{2l+1} = \xi\})$ and $\beta^G(\{r_{k}(x_1,x_2,0) = \xi\}) = \beta^G(\{x_1^2 x_2 + x_2^{2l+1} = \xi \})$ are equal and then $\beta^G(A^{\xi}_{k}(h_{k+1})) = \beta^G(A^{\xi}_{k}(r_{k}))$.

\end{proof} 
 
In the next paragraph, we will look at the last part of the respective equivariant zeta functions of $h_{k+1}$ and $r_k$. Still supposing that $p > q+1$ and $\eta = +1$ or $q > p+1$ and $\eta = -1$, we will show that, if $k$ is even or $k$ is odd and $\epsilon = +1$, their comparison reduces as in proposition \ref{akxihreven}.
 
\subsection{The last terms of the equivariant zeta functions} 
 
Suppose $p > q+1$ and $\eta = +1$ or $q > p+1$ and $\eta = -1$. Suppose that $k$ is even or that $k$ is odd and $\epsilon = +1$. The naive equivariant zeta functions of $h_{k+1}$ and $r_k$ are equal :

\begin{prop} \label{aMMhk} For all $M > k$, we have $\beta^G(A_M(h_{k+1})) = \beta^G(A_M(r_k))$.
\end{prop}
 
\begin{proof} Let $M$ be greater than $k$. We prove that $\beta^G(A^0_M(h_{k+1})) = \beta^G(A^0_M(r_k))$. 
\\

Suppose $k$ to be even, $k = 2l$. Consider the system of equations describing $A^0_M(h_{k+1})$ and $A^0_M(r_{k})$. The same computations as in the proofs of propositions \ref{a0mhr} and \ref{a0khr} bring, in both expressions of $\beta^G(A^0_M(h_{k+1}))$ and $\beta^G(A^0_M(r_{k}))$, an equal contribution of $\beta^G(Y_{p,q} \setminus \{0\})$ and a contribution of the equivariant virtual Poincar\'e series of a set defined by a system whose first equation is $\epsilon b_1^{k} + Q_{p,q}(c_1) = 0$. Stratifying this last algebraic set with the subsets $\{c_1^1 = \ldots = c_1^{i-1} = 0, c_1^i \neq 0\}$, $i = 1, \ldots, p+q$, and $\{c_1 = 0\}$ provides a contribution of $\beta^G(\{h_{k+1}(0,x_3,y) = 0\} \setminus \{0\})$, resp. $\beta^G(\{r_k(0,x_2,y) = 0\}\setminus \{0\})$ (it is the same quantity in our hypothesis) and we are led to the further condition $c_1 = 0$, and then $b_1 = 0$, in the previous system. 

Now, stratify with the subsets $\{a_1 \neq 0\}$ (this will provide an equal contribution for $h_{k+1}$ and $r_k$) and $\{a_1 = 0\}$. If $a_1 = 0$, shifting by $-1$ the indices of the remaining variables $a_i$ and $c_i$, we obtain a new system whose first equations are, if $M \geq 2k$ :
$$\begin{cases}
Q_{p,q}(c_1) = 0, \\
\Phi_{p,q}(c_1,c_2) = 0, \\
a_1^2 b_2 +  Q_{p,q}(c_2) + \Phi_{p,q}(c_1,c_3) = 0, \\
a_1^2 b_3 + 2 a_1 a_2 b_2 + \Phi_{p,q}(c_1, c_4) + \Phi_{p,q}(c_2,c_3) = 0, \\ 
\cdots \\
\sum_{t = 1}^{l-2} a_t^2 b_{2l-1-2t} + 2 \sum_{t=1}^{l-2} a_s \sum_{\delta=t+1}^{2l-2-(t+1)} a_{\delta} b_{2l-1-\delta-t} +\sum_{t= 1}^{l-1} \Phi_{p,q}(c_t,c_{2l-1-t}) = 0, \\
\epsilon b_2^{2l} + \sum_{t=1}^{l-1} a_t^2 b_{2l-2t} + 2 \sum_{t=1}^{l-2} a_t \sum_{\delta=t+1}^{2l-1-(t+1)} a_{\delta} b_{2l-\delta-t} + Q_{p,q}(c_l) + \sum_{t= 1}^{l-1} \Phi_{p,q}(c_t,c_{2l-t}) = 0.
\end{cases}$$
These equations can be obtained from the system defining $A_{k}^0(h_{k+1})$, by replacing the term $\epsilon b_1^{2l}$ with $\epsilon b_2^{2l}$ in the last equation and imposing $b_1$ to be $0$ in the other ones. 
%shifting by $+1$ the indices of the variables $b_i$ in the terms $\epsilon \prod b_j$.

%$$\begin{cases}
%a_1^2 b_2 + Q_{p,q}(c_2) = 0, \\
%a_1^2 b_3 + 2 a_1 a_2 b_2 + \Phi_{p,q}(c_2,c_3) = 0, \\
%a_1^2 b_4 + a_2^2 b_2 + 2 a_1 a_2 b_3 + 2 a_1 a_3 b_2 +  Q_{p,q}(c_3) + \Phi_{p,q}(c_2,c_4) = 0, \\
%\cdots \\
%\sum_{t = 1}^{l-1} a_t^2 b_{2l-1-2t} + 2 \sum_{t=1}^{l-2} a_s \sum_{\delta=t+1}^{2l-1-(t+1)} a_{\delta} b_{2l-1-\delta-t} +\sum_{t= 1}^{l-1} \Phi_{p,q}(c_t,c_{2l-1-t}) = 0, \\
%\epsilon b_1^{2l} + \sum_{t=1}^{l-1} a_t^2 b_{2l-2t} + 2 \sum_{t=1}^{l-1} a_t \sum_{\delta=t+1}^{2l-(t+1)} a_{\delta} b_{2l-\delta-t} + Q_{p,q}(c_l) + \sum_{t= 1}^{l-1} \Phi_{p,q}(c_t,c_{2l-t}) = 0,
%\end{cases}$$
%If $a_1 \neq 0$ then $b_2, b_3,\ldots, b_?$ are determined by $a_1$ and the other variables and if $a_1 = 0$, we obtain the system
%$$\begin{cases}
%Q_{p,q}(c_2) = 0, \\
%\Phi_{p,q}(c_2,c_3) = 0, \\
%a_2^2 b_2 +  Q_{p,q}(c_3) + \Phi_{p,q}(c_2,c_4) = 0, \\
%\cdots \\
%\sum_{t = 1}^{l-1} a_t^2 b_{2l-1-2t} + 2 \sum_{t=1}^{l-2} a_s \sum_{\delta=t+1}^{2l-1-(t+1)} a_{\delta} b_{2l-1-\delta-t} +\sum_{t= 1}^{l-1} \Phi_{p,q}(c_t,c_{2l-1-t}) = 0, \\
%\epsilon b_1^{2l} + \sum_{t=1}^{l-1} a_t^2 b_{2l-2t} + 2 \sum_{t=1}^{l-1} a_t \sum_{\delta=t+1}^{2l-(t+1)} a_{\delta} b_{2l-\delta-t} + Q_{p,q}(c_l) + \sum_{t= 1}^{l-1} \Phi_{p,q}(c_t,c_{2l-t}) = 0,
%\end{cases}$$
%Rewriting the variables $a_i$ and $c_i$, we obtain the system

Therefore, a similar process as above can be applied and provides further equal contributions for $\beta^G(A^0_M(h_{k+1}))$ and $\beta^G(A^0_M(r_{k}))$. In any case, the final equation will be either $Q_{p,q} (c_1) = 0$, $\epsilon b_{j}^{2l} + Q_{p,q}(c_1) = 0$ or trivial, so that the induced respective contributions are equal as well.

As a consequence, $\beta^G(A^0_M(h_{k+1})) = \beta^G(A^0_M(r_k))$ and $\beta^G(A_M(h_{k+1})) = \beta^G(A_M(r_k))$.
\\

If $k$ is odd, $k =2l+1$, and $\epsilon = +1$, from the initial system of equations defining $A^0_M(h_{k+1})$ and $A^0_M(r_{k})$, we are reduced to consider a system whose first equation is $b_1^{2l+1} + a_1^2 b_1 = 0$ (see the proof of proposition \ref{a0khr}). Therefore $b_1 = 0$. Stratifying with the subsets $\{a_1 \neq 0\}$ and $\{a_1 = 0 \}$, we then get, after a renaming of the variables, a system whose first non trivial equations are, if $M \geq 2k$ :
$$\begin{cases}
Q_{p,q}(c_1) = 0, \\
\Phi_{p,q}(c_1,c_2) = 0, \\
a_1^2 b_2 +  Q_{p,q}(c_2) + \Phi_{p,q}(c_1,c_3) = 0, \\
a_1^2 b_3 + 2 a_1 a_2 b_2 + \Phi_{p,q}(c_1, c_4) + \Phi_{p,q}(c_2,c_3) = 0, \\ 
\cdots \\
 \sum_{t = 1}^{l-1} a_t^2 b_{2l+1-2t} + 2 \sum_{t=1}^{l-1} a_s \sum_{\delta=t+1}^{2l-(t+1)} a_{\delta} b_{2l+1-\delta-t} +\sum_{t= 1}^{l} \Phi_{p,q}(c_t,c_{2l+1-t}) = 0,\\
b_2^{2l+1} + \sum_{t=1}^{l} a_t^2 b_{2l+2-2t} + 2 \sum_{t=1}^{l-1} a_t \sum_{\delta=t+1}^{2l+1-(t+1)} a_{\delta} b_{2l+2-\delta-t} + Q_{p,q}(c_{l+1}) + \sum_{t= 1}^{l} \Phi_{p,q}(c_t,c_{2l+2-t}) = 0.
\end{cases}$$  

Repeating the process provides further equal contributions for $\beta^G(A^0_M(h_{k+1}))$ and $\beta^G(A^0_M(r_{k}))$ and, if $M \geq 2k$, we are led to a new system whose first equation is $ b_2^{2l+1} + Q_{p,q}(c_1) = 0$. We will show in lemma \ref{boddq0} below that the respective induced contributions are equal.

In any case, these repeated steps of computations will eventually allow us to consider a single equation, which will be either $Q_{p,q} (c_1) = 0$, $b_j^{2l+1} + a_1^2 b_j = 0$, $b_{j}^{2l+1} + Q_{p,q}(c_1) = 0$ or trivial.

Consequently, if $k$ is odd, $\beta^G(A^0_M(h_{k+1})) = \beta^G(A^0_M(r_k))$ and $\beta^G(A_M(h_{k+1})) = \beta^G(A_M(r_k))$ as well.
 
\end{proof}

\begin{lem} \label{boddq0} Suppose that $k$ is odd, $k = 2l+1$, $p > q+1$ and $\eta = +1$ (the property will also be true if $q > p+1$ and $\eta = -1$). Then
$$\beta^G(\{h_{k+1}(0, x_3, y) = 0 \}) = \beta^G(\{r_{k}(0, x_2, y) = 0 \}).$$
\end{lem} 
 
\begin{proof} Applying successive blowings-up as in the proof of lemma \ref{fg2k0}, we obtain
\begin{eqnarray*}
\beta^G(\{h_{k+1}(0, x_3, y) = 0 \}) & = & \beta^G(\{\epsilon x_3 + Q_{p,q}(y) = 0\}) - k \beta^G( \{Q_{p,q}(y) = 0\}) + k \beta^G(\{0\}) \\
& = & \beta^G(\mathbb{R}^{p+q}) - k \beta^G( \{Q_{p,q}(y) = 0\}) + k \beta^G(\{0\}).
\end{eqnarray*}
We have the same expression for $\beta^G(\{r_{k}(0, x_2, y) = 0 \})$ and therefore, since $p > q+1$ and $\eta = +1$, the two quantities are equal.
\end{proof} 
 
As for the last part of the equivariant zeta functions with signs of $h_{k+1}$ and $r_k$, adapting the computations of the proof of proposition \ref{aMMhk}, we obtain the following (still under our current hypothesis) :

\begin{prop} \label{lasttermszetafsignshr}

\begin{enumerate} 
	\item Suppose $k$ is even. \begin{itemize} \item If $\eta = \epsilon$, then, for all $M > k$,  
$\beta^G(A^{\xi}_M(h_{k+1})) = \beta^G(A^{\xi}_M(r_k))$ and consequently the respective equivariant zeta functions with signs of $h_{k+1}$ and $r_k$ are equal.
	\item If $\eta = +1$, $\epsilon = -1$, we have the equality $Z_{h_{k+1}}^{G,\xi}(u,T) = Z_{r_{k}}^{G,\xi}(u,T)$ if and only if $\beta^G\left(\left\{ - x_3^{k} + y^2 + \sum_{i = 1}^{K-1} y_i^2 = \xi \right\}\right) = \beta^G\left(\left\{ - x_2^{k} + \sum_{i = 1}^{K} z_i^2 = \xi \right\}\right)$ (the former set is a subset of $\mathbb{R}^{K+1}$ equipped with the action of $G$ only changing the sign of $y$ and the latter set is a subset of $\mathbb{R}^{K+1}$ equipped with the trivial action of $G$).
	\item If $\eta = -1$, $\epsilon = +1$, we have the equality $Z_{h_{k+1}}^{G,\xi}(u,T) = Z_{r_{k}}^{G,\xi}(u,T)$ if and only if $\beta^G\left(\left\{ x_3^{k} -  y^2 - \sum_{i = 1}^{K-1} y_i^2 = \xi \right\}\right) = \beta^G\left(\left\{ x_2^{k} - \sum_{i = 1}^{K} z_i^2 = \xi \right\}\right)$.
\end{itemize}
	\item Suppose $k$ is odd and $\epsilon = +1$. \begin{itemize} \item If $\eta = +1$, the equality $Z_{h_{k+1}}^{G,\xi}(u,T) = Z_{r_{k}}^{G,\xi}(u,T)$ is true if and only if the quantities $\beta^G\left(\left\{ x_3^{k} + y^2 + \sum_{i = 1}^{K-1} y_i^2 = \xi \right\}\right)$ and $\beta^G\left(\left\{ x_2^{k} + \sum_{i = 1}^{K} z_i^2 = \xi \right\}\right)$ are equal.
	\item If $\eta = -1$, the equality $Z_{h_{k+1}}^{G,\xi}(u,T) = Z_{r_{k}}^{G,\xi}(u,T)$ is true if and only if the quantities $\beta^G\left(\left\{ x_3^{k} - y^2 - \sum_{i = 1}^{K-1} y_i^2 = \xi \right\}\right)$ and $\beta^G\left(\left\{ x_2^{k} - \sum_{i = 1}^{K} z_i^2 = \xi \right\}\right)$ are equal.
\end{itemize}
\end{enumerate}

\end{prop}
 
\begin{proof} Let $M$ be greater than $k$. Since the system describing $A^{\xi}_M(h_{k+1})$ and $A^{\xi}_M(r_k)$ is obtained from the one defining $A^{0}_M(h_{k+1})$ and $A^{0}_M(r_k)$ by replacing $0$ by $\xi$ in the right member of the last equation, we are reduced, as in the proof of proposition \ref{aMMhk}, to consider a single equation.

If $k$ is even, $k = 2l$, this equation is either $Q_{p,q} (c_1) = \xi$, $\epsilon b_{j}^{2l} + Q_{p,q}(c_1) = \xi$ or an equation with no solution. Under our current hypothesis, the quantity $\beta^G(Y_{p,q}^{\xi})$ is the same for $h_{k+1}$ and $r_k$. If $\epsilon = \eta$, we can show, as in the proof of proposition \ref{eqa2kxi}, using the formulae of lemma \ref{fgxi}, that $\beta^G(\{h_{k+1}(0,x_3,y) = \xi\}) = \beta^G(\{r_k(0,x_2,y) = \xi\})$. If $\epsilon = - \eta$, we also use lemma \ref{fgxi} to obtain the desired equivalences. 

If $k$ is odd and $\epsilon = +1$, the final equation is either $Q_{p,q} (c_1) = \xi$, $b_j^{2l+1} + a_1^2 b_j = \xi$, $b_{j}^{2l+1} + Q_{p,q}(c_1) = \xi$ or an equation with no solution. The quantity $\beta^G(Y_{p,q}^{\xi})$ is the same for $h_{k+1}$ and $r_k$ and we showed in proposition \ref{axihkkoddepsplus} that $\beta^G(\{h_{k+1}(x_2,x_3,0) = \xi\}) = \beta^G(\{r_k(x_1,x_2,0) = \xi\})$. Finally, we can obtain formulae similar to the ones in lemma \ref{fgxi} for $\beta^G(\{h_{k+1}(0,x_3,y) = \xi\})$ and $\beta^G(\{r_k(0,x_2,y) = \xi\})$ if $k$ is odd and this provides the desired equivalences.
\end{proof} 
 
\subsection{Conclusion}

We gather the obtained results in the following statement :

\begin{theo} Let $k \geq 3$. Suppose that the invariant germs
$$h_{k+1} = x_2^2 x_3 + \epsilon x_3^k + \eta x_1^2 + Q \mbox{   and   } r_k = x_1^2 x_2 + \epsilon x_2^k + \eta' x_3^2 + Q'$$
have, up to permutation of all variables, the same quadratic part, with $p$ signs $+$ and $q$ signs~$-$.

\begin{enumerate}
	\item If \begin{itemize}
			\item $p \leq q$, $\eta = +1$ or $q \leq p$, $\eta = -1$, 
			\item $p = q+1$ or $q = p+1$,
			\item $k$ is odd, $\epsilon = -1$, 
		\end{itemize}
then $h_{k+1}$ and $r_k$ are not $G$-blow-Nash equivalent.
	\item If $k$ is even and if $p > q+1$, $\eta = +1$, $\epsilon = +1$ or $q > p+1$, $\eta = -1$, $\epsilon = -1$, then $Z_{h_{k+1}}^{G}(u,T) = Z_{r_{k}}^{G}(u,T)$ and $Z_{h_{k+1}}^{G,\xi}(u,T) = Z_{r_{k}}^{G,\xi}(u,T)$.
	\item \begin{itemize}
			\item If $k$ is even and if $p > q+1$, $\eta = +1$, $\epsilon = -1$, then $Z_{h_{k+1}}^{G}(u,T) = Z_{r_{k}}^{G}(u,T)$. Furthermore, $Z_{h_{k+1}}^{G,\xi}(u,T) = Z_{r_{k}}^{G,\xi}(u,T)$ if and only if $\beta^G\left(\left\{ - x_3^{k} + y^2 + \sum_{i = 1}^{K-1} y_i^2 = \xi \right\}\right) = \beta^G\left(\left\{ - x_2^{k} + \sum_{i = 1}^{K} z_i^2 = \xi \right\}\right)$.
			\item If $k$ is even and if $q > p+1$, $\eta = -1$, $\epsilon = +1$, then $Z_{h_{k+1}}^{G}(u,T) = Z_{r_{k}}^{G}(u,T)$. Furthermore, $Z_{h_{k+1}}^{G,\xi}(u,T) = Z_{r_{k}}^{G,\xi}(u,T)$ if and only if $\beta^G\left(\left\{ x_3^{k} -  y^2 - \sum_{i = 1}^{K-1} y_i^2 = \xi \right\}\right) = \beta^G\left(\left\{ x_2^{k} - \sum_{i = 1}^{K} z_i^2 = \xi \right\}\right)$.
			\item If $k$ is odd and if $p > q+1$, $\eta = +1$, $\epsilon = +1$, then $Z_{h_{k+1}}^{G}(u,T) = Z_{r_{k}}^{G}(u,T)$. Furthermore, $Z_{h_{k+1}}^{G,\xi}(u,T) = Z_{r_{k}}^{G,\xi}(u,T)$ if and only if $\beta^G\left(\left\{ x_3^{k} + y^2 + \sum_{i = 1}^{K-1} y_i^2 = \xi \right\}\right) = \beta^G\left(\left\{ x_2^{k} + \sum_{i = 1}^{K} z_i^2 = \xi \right\}\right)$.
			\item If $k$ is odd and if $q > p+1$, $\eta = -1$, $\epsilon = +1$, then $Z_{h_{k+1}}^{G}(u,T) = Z_{r_{k}}^{G}(u,T)$. Furthermore, $Z_{h_{k+1}}^{G,\xi}(u,T) = Z_{r_{k}}^{G,\xi}(u,T)$ if and only if $\beta^G\left(\left\{ x_3^{k} - y^2 - \sum_{i = 1}^{K-1} y_i^2 = \xi \right\}\right) = \beta^G\left(\left\{ x_2^{k} - \sum_{i = 1}^{K} z_i^2 = \xi \right\}\right)$. 
		\end{itemize}
\end{enumerate}
\end{theo} 
 
\begin{rem} \label{remforgacthr} If we forget the $G$-actions, the virtual Poincar\'e polynomials of the algebraic subsets $\left\{ x^{2l+1} + \sum_{i = 1}^{K} y_i^2 = \xi \right\}$ and $\left\{ x^{2l+1} - \sum_{i = 1}^{K} y_i^2 = \xi \right\}$ of $\mathbb{R}^{K+1}$, $\xi = \pm 1$, can also be computed using the invariance of the virtual Poincar\'e polynomial under bijection with $\mathcal{AS}$ graph (see remark \ref{remevpsas}).
%but we do not know if the equivariant virtual Poincar\'e series is invariant under equivariant bijection with $\mathcal{AS}$ graph. 
\end{rem} 
 
\section{The germs $E_6$ and $F_4$} \label{compEF}

Finally, we study the classification with respect to $G$-blow-Nash equivalence of the families
$$\varphi^{\epsilon}(x) := \pm x_1^2 + x_2^3 + \epsilon x_3^4 + Q$$
and
$$\omega^{\epsilon}(x) := \epsilon x_1^4 + x_2^3 + \pm x_3^3 + Q'$$
where $\epsilon \in \{-1 ; +1\}$. 

If two germs $\varphi^{\epsilon}$ and $\varphi^{\epsilon'}$ are $G$-blow-Nash equivalent, they have the same quadratic part up to permutation of the variables $x_1, x_4, \ldots, x_n$ and, by \cite{GF-BNT} Proposition 3.14, $\epsilon = \epsilon'$. Furthermore, we will show in corollary \ref{firstdistphiom} below that the germs
$$\varphi^{\epsilon,+}(x) := + x_1^2 + x_2^3 + \epsilon x_3^4 + Q \mbox{ and } \varphi^{\epsilon,-}(x) := - x_1^2 + x_2^3 + \epsilon x_3^4 + Q',$$
where $\epsilon\in \{-1 \, ; + 1\}$ and $+ x_1^2 + Q$ and $- x_1^2 + Q'$ are the same quadratic part up to permutation of the variables $x_1, x_4, \ldots , x_n$, are not $G$-blow-Nash equivalent.

If two germs $\omega^{\epsilon}$ and $\omega^{\epsilon'}$ are $G$-blow-Nash equivalent, they also have the same quadratic part, up to permutation of the variables $x_3, \ldots, x_n$, and $\epsilon = \epsilon'$ as well.

If now two germs $\varphi^{\epsilon}$ and $\omega^{\epsilon'}$ are $G$-blow-Nash equivalent, then $\epsilon = \epsilon'$ and $\pm x_1^2 + Q$ and $\pm x_3^2 + Q'$ are the same quadratic part up to permutation of all variables, so that we will intend to compare the germs
$$\varphi(x) =  x_2^3 + \epsilon x_3^4 + \eta x_1^2 + Q \mbox{ and } \omega(x) =  x_2^3 + \epsilon x_1^4 + \eta' x_3^3 + Q'$$
where $\epsilon, \eta, \eta' \in \{1,-1\}$ and $\eta x_1^2 + Q = \eta' x_3^3 + Q'$ up to permutation of all variables.
\\

As in the previous two parts, we will consider the respective equivariant zeta functions of $\varphi$ and $\omega$, along with theorem \ref{eqzfinv}, to try to distinguish these invariant germs with respect to $G$-blow-Nash equivalence.

We begin with the computation of the first coefficients $\beta^G(A_2(\varphi))$, $\beta^G(A_3(\varphi))$, $\beta^G(A_4(\varphi))$ of the naive equivariant zeta function of $\varphi$ (notice that the set $A_1(\varphi)$ is empty so that $\beta^G(A_1(\varphi)) = 0$). Thanks to proposition \ref{oamh}, we can focus on the quantities $\beta^G(A^0_m(\varphi))$, $m \leq 4$. The corresponding expressions for $\omega$ are similar, in this case equipping the set $Y_{p,q}$ with the trivial action of $G$.

\begin{prop} \label{firsttermsphiom} Write $\varphi = \varphi(x,z,y) = x^3 + \epsilon z^4 + Q_{p,q}(y)$. We have $\beta^G(A_2^0(\varphi)) = u^{4+p+q} \beta^G(Y_{p,q})$, $\beta^G(A_3^0(\varphi)) = u^{2(p+q) + 5} \beta^G(Y_{p,q} \setminus \{0\}) + \frac{u^{2(p+q) + 6}}{u-1}$ and $\beta^G(A_4^0(\varphi)) = u^{3(p+q) + 6} \beta^G(Y_{p,q} \setminus \{0\})  + u^{2(p+q) + 6} \beta^G(\{\varphi(0, z,y) = 0\})$.
\end{prop}

\begin{proof} If $m \geq 1$, we write an arc $\gamma$ of $\mathcal{L}_m$ as
\begin{eqnarray*} \gamma(t) & = & (a_1 t + \cdots + a_{m} t^{m}, b_1 t + \cdots + b_{m} t^{m}, c_1^1 t + \cdots + c_{m}^1 t^{m}, \ldots, c_1^{p+q} t + \cdots + c_{m}^{p+q} t^{m}) \\
& = & \left( \begin{array}{c} a_1 \\ b_1 \\ c_1^1 \\ \vdots \\ c_1^{p+q} \end{array} \right) t + \cdots + \left( \begin{array}{c} a_{m} \\ b_m \\c_{m}^1 \\ \vdots \\ c_{m}^{p+q} \end{array} \right) t^{m} = \left( \begin{array}{c} a_1 \\ b_1\\ c_1 \end{array} \right) t + \cdots + \left( \begin{array}{c} a_{m} \\ b_1 \\ c_{m}\end{array} \right) t^{m}
\end{eqnarray*}
(the group $G$ acts only changing the sign of the coordinates $c_i^1$, resp. $c_i^{p+1}$, in the case n\textsuperscript{o}\ref{acty1}, resp. n\textsuperscript{o}\ref{actyp1}).

The set $A_2^0(\varphi)$ is described by the single equation $Q_{p,q}(c_1) = 0$, the other variables remaining free. The set $A_3^0(\varphi)$ is defined by the system
$$\begin{cases}
Q_{p,q}(c_1) = 0, \\
a_1^3 + \Phi_{p,q}(c_1,c_2) = 0,
\end{cases}$$
and, stratifying with the $G$-globally invariant subsets $\{c_1^1 = \ldots = c_1^{i-1} = 0, c_1^i \neq 0\}$, $i = 1, \ldots, p$, and $\{c_1^1 = \ldots = c_1^p = 0\} = \{c_1 = 0\}$, we obtain
$$\beta^G(A_3^0(\varphi)) = u^{6 + 2(p-1) + 2q + 1} \beta^G(Y_{p,q} \setminus \{0\}) + \beta^G(A_3^0(\varphi) \cap \{c_1 = 0\}).$$ 
If $c_1 = 0$, then $a_1 = 0$ and the other variables are free, hence the desired expression.

Finally, $A_4^0(\varphi)$ is described by the system of equations
$$\begin{cases}
Q_{p,q}(c_1) = 0, \\
a_1^3 + \Phi_{p,q}(c_1,c_2) = 0,\\
\epsilon b_1^4 + 3 a_1^2 a_2 + Q_{p,q}(c_2) + \Phi(c_1,c_3) = 0.
\end{cases}$$
Equivariantly stratifying $A_4^0(\varphi)$ as we did for $A_3^0(\varphi)$, we get the equality
$$\beta^G(A_4^0(\varphi)) = u^{8 + 3(p-1) + 3q + 1} \beta^G(Y_{p,q} \setminus \{0\}) + \beta^G(A_4^0(\varphi) \cap \{c_1 = 0, a_1 = 0\}),$$ 
the set $A_4^0(\varphi) \cap \{c_1 = 0, a_1 = 0\}$ being given by the equation $\epsilon b_1^4 + Q_{p,q}(c_2) = 0$.
\end{proof}

%\begin{rem} In the computation for $\omega$, the group $G$ acts only changing the sign of the coordinates $b_i$.
%\end{rem}

With the same way of computation, we obtain the following expressions for the first terms of the equivariant zeta functions with signs of $\varphi$ :

\begin{prop} We have $\beta^G(A_2^{\xi}(\varphi)) = u^{4+p+q} \beta^G(Y^{\xi}_{p,q})$, $\beta^G(A_3^{\xi}(\varphi)) = u^{2(p+q) + 5} \beta^G(Y_{p,q} \setminus \{0\}) + \frac{u^{2(p+q) + 6}}{u-1}$ and $\beta^G(A_4^{\xi}(\varphi)) = u^{3(p+q) + 6} \beta^G(Y_{p,q} \setminus \{0\})  + u^{2(p+q) + 6} \beta^G(\{\varphi(0, x_3,y) = \xi\})$. 
\end{prop}

As we did in sections \ref{firsttermsab}, \ref{a2kfg} and \ref{firsttermscd}, \ref{akhr}, we deduce the following distinctions :

\begin{cor} \label{firstdistphiom} \begin{enumerate}
	\item The germs $\varphi^{\epsilon,+}$ and $\varphi^{\epsilon,-}$ are not $G$-blow-Nash equivalent.
	\item If $p \leq q$ and $\eta = +1$ or $q \leq p$ and $\eta = -1$, then the germs $\varphi$ and $\omega$ are not $G$-blow-Nash equivalent.
	\item If $p = q+1$ or $q = p+1$, then $\varphi$ and $\omega$ are not $G$-blow-Nash equivalent.
\end{enumerate}
\end{cor}

If $p > q+1$ and $\eta = +1$ or $q > p+1$ and $\eta = -1$, the respective quantities $\beta^G(Y_{p,q})$ and $\beta^G(Y_{p,q}^{\xi})$ are identical for $\varphi$ and $\omega$. Furthermore, notice that, equivariantly, $\{\varphi(0, x_3,y) = 0\} = \{f_3(x_2,y) = 0\}$, resp. $\{\varphi(0, x_3,y) = \xi\} = \{f_3(x_2,y) = \xi\}$, and $\{\omega(0,x_1,y) = 0 \} = \{g_2(x_1,y) = 0\}$, resp. $\{\omega(0,x_1,y) = \xi \} = \{g_2(x_1,y) = \xi\}$. Therefore, thanks to the computations of paragraph \ref{a2kfg}, we can state the following :

\begin{prop} Suppose that $p > q+1$ and $\eta = +1$ or $q > p+1$ and $\eta = -1$.
\begin{enumerate}
	\item For $m \leq 4$, $\beta^G(A_m(\varphi)) = \beta^G(A_m(\omega))$.
	\item For $m \leq 3$, $\beta^G(A^{\xi}_m(\varphi)) = \beta^G(A^{\xi}_m(\omega))$.
\item \begin{itemize}	\item If $\eta = \epsilon$, then $\beta^G(A^{\xi}_4(\varphi)) = \beta^G(A^{\xi}_4(\omega))$.
	\item If $\eta = +1$ and $\epsilon = -1$, then $\beta^G(A_{4}^{\xi}(\varphi)) = \beta^G(A_{4}^{\xi}(\omega))$ if and only if the equivariant virtual Poincar\'e series of the algebraic subsets $\left\{ - x_3^{4} + y^2 + \sum_{i = 1}^{K-1} y_i^2 = \xi \right\} \subset \mathbb{R}^{K+1}$, $K := p-q$, equipped with the action of $G$ only changing the sign of $y$, and $\left\{ - x_1^{4} + \sum_{i = 1}^{K} z_i^2 = \xi \right\}\subset \mathbb{R}^{K+1}$, equipped with the action of $G$ only changing the sign of $x_1$, are equal.
 	\item If $\eta = -1$ and $\epsilon = +1$, then we have $\beta^G(A_{4}^{\xi}(\varphi)) = \beta^G(A_{4}^{\xi}(\omega))$ if and only if $\beta^G\left(\left\{ x_3^{4} - y^2 - \sum_{i = 1}^{K-1} y_i^2 = \xi \right\}\right) = \beta^G\left(\left\{ x_1^{4} - \sum_{i = 1}^{K} z_i^2 = \xi \right\} \right)$.
\end{itemize}
\end{enumerate}
\end{prop}

For these cases, we have then to look at the other coefficients of the equivariant zeta functions of $\varphi$ and $\omega$. We begin by showing that, under this hypothesis $p > q+1$, $\eta = +1$ or $q > p+1$, $\eta = -1$, the respective naive equivariant zeta functions of $\varphi$ and $\omega$ are equal :

\begin{prop} \label{lasttermsphiom} If $p > q+1$ and $\eta = +1$ or $q > p+1$ and $\eta = -1$, then, for all $M > 4$, we have $\beta^G(A_M(\varphi)) = \beta^G(A_M(\omega))$.
\end{prop}

\begin{proof} Let $M$ be greater than $4$, we prove that $\beta^G(A^0_M(\varphi)) = \beta^G(A^0_M(\omega))$. If we consider the system defining the two latter sets, the same computations as in proposition \ref{firsttermsphiom} provide an equal (under our current hypothesis) contribution of $\beta^G(Y_{p,q} \setminus \{0\})$ and we are reduced to consider a system whose first condition is $a_1 = 0$ and next equation is (after a shift of indices) $\epsilon b_1^4 + Q_{p,q}(c_1) = 0$. This equation induces equal contributions for $\beta^G(A^0_M(\varphi))$ and $\beta^G(A^0_M(\omega))$ as well (recall that $\{\varphi(0, x_3,y) = 0\} = \{f_3(x_2,y) = 0\}$ and $\{\omega(0,x_1,y) = 0 \} = \{g_2(x_1,y) = 0\}$).

We then stratify with respect to the coordinates of $c_1$ as we did in the proofs of propositions \ref{aM02k} and \ref{aMMhk}, and we obtain the further condition $b_1 = 0$. The first subsequent equations become, if $M \geq 8$,
$$\begin{cases}
a_2^3 + Q_{p,q}(c_1) = 0, \\
3 a_2^2 a_3 + \Phi_{p,q}(c_1,c_2) = 0,\\
\epsilon b_2^4 + 3 a_2 a_3^2 +  3 a_2^2 a_4 + Q_{p,q}(c_2) + \Phi(c_1,c_3) = 0.
\end{cases}$$
  
Another stratification with respect to the vector $c_1$ provides an equal (by lemma \ref{boddq0}) contribution of $\beta^G(\{\varphi(x_2,0,y) = 0 \}) = \beta^G(\{h_4(0,x_3,y) =0\})$, respectively $\beta^G(\{\omega(x_2, 0,y) = 0\}) = \beta^G(\{r_3(0,x_2,y) = 0\})$, and the condition $a_2 = 0$. 

Carrying on with the computation, we obtain the equivalence $\beta^G(A^0_M(\varphi)) =\beta^G(A^0_M(\omega))$ if and only if $\beta^G(\{\varphi = 0 \}) = \beta^G(\{\omega = 0\})$, from the equations of the form $\epsilon b_j^4 + a_{j'}^3 + Q_{p,q}(c_1) = 0$ with $4 j = 3 j'$. We prove in lemma \ref{betaphi0eqom0} below that $\beta^G(\{\varphi = 0 \}) = \beta^G(\{\omega = 0\})$.
\end{proof}

\begin{lem} \label{betaphi0eqom0} Suppose that $p > q+1$ and $\eta = +1$ or $q > p+1$ and $\eta = -1$. Then
$$\beta^G(\{\varphi = 0 \}) = \beta^G(\{\omega = 0\}).$$
\end{lem}

\begin{proof} Suppose that $p> q+ 1$ and $\eta = +1$. Considering an equivariant resolution of singularities of the $G$-algebraic set $\{\omega = 0\}$, we will compare the quantities $\beta^G(\{\omega = 0 \})$ and $\beta^G(\{\varphi = 0 \})$. 

Write $\omega(x) =  x^3 + \epsilon z^4 + Q_{p,q}(y)$ (the group $G$ acts via the involution $(x,z,y) \mapsto (x,-z,y)$). Using an equivariant change of coordinates as in the proof of proposition \ref{ypq}, we can assume $q = 0$. We then equivariantly blow-up the $G$-algebraic set $\{\omega = 0 \}$ at the origin of $\mathbb{R}^n$ : 
\begin{itemize}
	\item in the chart $x = u$, $z = uv$, $y_i = u w_i$, with $G$-action $(u,v,w_i) \mapsto (u,-v,w_i)$, the equation of the blown-up variety is
$$u^2[u + \epsilon u^2 v^4 + Q_{p,q}(w)] = 0,$$
	\item in the chart $x = v u$, $z = v$, $y_i = v w_i$, with $G$-action $(u,v,w_i) \mapsto (-u,-v,-w_i)$, it is
$$v^2[v u^3 + \epsilon v^2 + Q_{p,q}(w)]= 0,$$
	\item in the respective charts $x = w_j u$, $z = w_j v$, $y_j = w_j$ , $y_i = w_j w_i$ for $i \neq j$, with $G$-action $(u,v,w_i) \mapsto (u,v,w_i)$, it is
$$w_j^2 [w_j u^3 + \epsilon w_j^2 v^4 + 1 + Q(\widehat{w})] = 0.$$
\end{itemize}

The set of points of the strict transform of $\{\omega = 0 \}$ which are in the first chart but not in the second one is given by $v = 0,  u + Q_{p,q}(w) = 0$, therefore it is equivariantly isomorphic to an affine space : the respective induced contributions for $\beta^G(\{\omega = 0 \})$ and $\beta^G(\{\varphi = 0 \})$ are equal. Now, the set of points of the strict transform which are in one of the last charts but not in the second and the first ones is given by $v = 0, u = 0, 1+  Q(\widehat{w}) = 0$ : it is the empty set ($q = 0$).

Furthermore, notice that the intersection of the strict transform of $\{\omega = 0 \}$ with the exceptional divisor is a circle with a nonempty fixed point set.
\\

Consequently, we are reduced to consider the equivariant virtual Poincar\'e series of the algebraic set of $\mathbb{R}^n$ defined by the equation $z x^3 + \epsilon z^2 + Q_{p,q}(y) = 0$, $G$ acting via $(x,z,y) \mapsto (-x,-z,-y)$ (for $\varphi$, the involution would have been $(x,z,y_1,y_i) \mapsto (x,z,-y_1,y_i)$). We equivariantly blow-up this $G$-algebraic set at the origin of $\mathbb{R}^n$ as well :
\begin{itemize}
	\item in the chart $x = u$, $z = uv$,$y_i = u w_i$, with $G$-action $(u,v,w_i) \mapsto (-u,v,w_i)$, the equation of the blown-up variety is
$$u^2[v u^2 + \epsilon v^2 + Q_{p,q}(w)] = 0,$$
	\item in the chart $x = v u$, $z = v$, $y_i = v w_i$, with $G$-action $(u,v,w_i) \mapsto (u,-v,w_i)$, it is
$$v^2[v^2 u^3 + \epsilon + Q_{p,q}(w)] = 0,$$
	\item in the respective charts $x = w_j u$, $z = w_j v$, $y_j = w_j$ , $y_i = w_j w_i$ for $i \neq j$, with $G$-action $(u,v,w_j, w_i) \mapsto (u,v,-w_j, w_i)$, it is
$$w_j^2 [v w_j^2 u^3 + \epsilon v^2 + 1 + Q(\widehat{w})] =0.$$
\end{itemize} 

The set of points of the strict transform which are in the second chart but not in the first one is given by $u = 0, \epsilon + Q_{p,q}(w) = 0$ : it is the cartesian product of an affine line and the set $Y_{p,q}^{-\epsilon}$, and therefore it induces an equal contribution for $\beta^G(\{\omega = 0 \})$ and $\beta^G(\{\varphi = 0 \})$ under our current hypothesis. As for the set of points of the strict transform which are in one of the last charts but not in the first and the second ones, it is given by $u = 0, v = 0, 1 + Q(\widehat{w}) = 0$, thus it is empty.

% : according to the sign of $\epsilon$, it is equivariantly isomorphic to an affine line times either $Y_{p,0}^{-1}$ or $Y_{p-1,1}^{-1}$. Since $p \neq 2$, the respective induced contributions for $\beta^G(\{\omega = 0 \})$ and $\beta^G(\{\varphi = 0 \})$ are equal. 

On the other hand, the intersection of the strict transform with the exceptional divisor provides equal contributions for $\beta^G(\{\omega = 0 \})$ and $\beta^G(\{\varphi = 0 \})$ as well.
\\

As a consequence, we can focus on the equation $zx^2 + \epsilon z^2 + Q_{p,q}(y) = 0$ in $\mathbb{R}^n$, the group $G$ acting via $(x,z,y) \mapsto (-x,z,y)$ for $\omega$ (respectively via $(x,z,y_1,y_i) \mapsto (x,z,-y_1,y_i)$ for $\varphi$). We equivariantly blow-up once again :
\begin{itemize}
	\item in the chart $x = u$, $z = uv$,$y_i = u w_i$, with $G$-action $(u,v,w_i) \mapsto (-u,-v,-w_i)$, the equation of the blown-up variety is
$$u^2[uv + \epsilon v^2 + Q_{p,q}(w) ]=0,$$
	\item in the chart $x = v u$, $z = v$, $y_i = v w_i$, with $G$-action $(u,v,w_i) \mapsto (-u,v,w_i)$, it is
$$v^2[v u^2+ \epsilon + Q_{p,q}(w) ]= 0,$$
	\item in the respective charts $x = w_j u$, $z = w_j v$, $y_j = w_j$ , $y_i = w_j w_i$ for $i \neq j$, with $G$-action $(u,v,w_i) \mapsto (-u,v,w_i)$, it is
$$w_j^2 [v w_j u^2 + \epsilon v^2 + 1 + Q(\widehat{w})] =0.$$
\end{itemize}
 
By similar arguments as above, we are reduced to consider the equation $uv + \epsilon v^2 + Q_{p,q}(w) = 0$. We can then stratify with respect to $v$ and show that the induced respective contributions for $\beta^G(\{\varphi = 0 \})$ and $\beta^G(\{\omega = 0\})$ are also the same. This finally proves the equality $\beta^G(\{\varphi = 0 \}) = \beta^G(\{\omega = 0\})$.
%\\

%If $p = 2$ and $\epsilon = +1$, it goes the same as above (the set $Y_{p,0}^{-1}$ is empty). If $p = 2$ and $\epsilon = -1$, after the second blowing-up considered above, the set of points of the strict transform which are in one of the last charts but not in the first one is given by $u = 0,  \epsilon v^2 + 1 + Q(\widehat{w}) = 0$, while the intersection of the strict transform with the exceptional divisor is given by $w_j = 0, \epsilon v^2 + 1 + Q(\widehat{w}) = 0$

\end{proof}

Similarly to what we did in the proofs of propositions \ref{lasttermszetafsignsfg} and \ref{lasttermszetafsignshr}, we can adapt the proof of proposition \ref{lasttermsphiom} in order to state a sufficient and necessary condition for the equality of the respective equivariant zeta functions with signs of $\varphi$ and $\omega$ to be true :

\begin{prop} 
	\begin{enumerate}
		\item Suppose $p > q+1$ and $\eta = +1$. Then $Z_{\varphi}^{G,\xi}(u,T) = Z_{\omega}^{G,\xi}(u,T)$ if and only if we have the equalities 
		\begin{itemize}
			\item $\beta^G(\{x_2^3 + y^2 + \sum_{i=1}^{K-1} y_i^2 = \xi\}) = \beta^G(\{x_2^3 + \sum_{i=1}^{K} z_i^2 = \xi\})$, 
			\item $\beta^G(\{\epsilon x_3^4 + y^2 + \sum_{i=1}^{K-1} y_i^2 = \xi\}) = \beta^G(\{\epsilon x_1^4 + \sum_{i=1}^{K} z_i^2 = \xi\})$,
			\item and $\beta^G(\{x_2^3 + \epsilon x_3^4 + y^2 + \sum_{i=1}^{K-1} y_i^2 = \xi\}) = \beta^G(\{x_2^3 + \epsilon x_1^4 + \sum_{i=1}^{K} z_i^2 = \xi\})$,
		\end{itemize}
where, in the left members of the equalities, the considered sets are algebraic subsets of $\mathbb{R}^{K+2}$ equipped with the action of $G$ only changing the sign of $y$, and, in the right members, the sets are subsets of $\mathbb{R}^{K+2}$ equipped with the action of $G$ only changing the sign of $x_1$.
		\item Suppose $q > p+1$ and $\eta = -1$. Then $Z_{\varphi}^{G,\xi}(u,T) = Z_{\omega}^{G,\xi}(u,T)$ if and only if we have the equalities 
		\begin{itemize}
			\item $\beta^G(\{x_2^3 - y^2 - \sum_{i=1}^{K-1} y_i^2 = \xi\}) = \beta^G(\{x_2^3 - \sum_{i=1}^{K} z_i^2 = \xi\})$, 
			\item $\beta^G(\{\epsilon x_3^4 - y^2 - \sum_{i=1}^{K-1} y_i^2 = \xi\}) = \beta^G(\{\epsilon x_1^4 - \sum_{i=1}^{K} z_i^2 = \xi\})$,
			\item and $\beta^G(\{x_2^3 + \epsilon x_3^4 - y^2 - \sum_{i=1}^{K-1} y_i^2 = \xi\}) = \beta^G(\{x_2^3 + \epsilon x_1^4 - \sum_{i=1}^{K} z_i^2 = \xi\})$.
		\end{itemize}
	\end{enumerate}
\end{prop} 

\begin{rem} As we showed in the proof of proposition \ref{eqa2kxi}, we have $\beta^G(\{ + x_3^4 + y^2 + \sum_{i=1}^{K-1} y_i^2 = \xi\}) = \beta^G(\{ + x_1^4 + \sum_{i=1}^{K} z_i^2 = \xi\})$ and $\beta^G(\{ - x_3^4 - y^2 - \sum_{i=1}^{K-1} y_i^2 = \xi\}) = \beta^G(\{ - x_1^4 - \sum_{i=1}^{K} z_i^2 = \xi\})$ for $\xi = \pm 1$.
\end{rem}

We finally gather the results of this section in one theorem :

\begin{theo} Suppose that the invariant germs
$$\varphi(x) =  x_2^3 + \epsilon x_3^4 + \eta x_1^2 + Q \mbox{ and } \omega(x) =  x_2^3 + \epsilon x_1^4 + \eta' x_3^3 + Q'$$
have, up to permutation of all variables, the same quadratic part, with $p$ signs $+$ and $q$ signs~$-$.

\begin{enumerate}
	\item If \begin{itemize}
			\item $p \leq q$, $\eta = +1$ or $q \leq p$, $\eta = -1$, 
			\item $p = q+1$ or $q = p+1$,
		\end{itemize}
then $\varphi$ and $\omega$ are not $G$-blow-Nash equivalent.
	\item \begin{itemize}
			\item If $p > q+1$, $\eta = +1$, then $Z_{\varphi}^{G}(u,T) = Z_{\omega}^{G}(u,T)$. Furthermore, $Z_{\varphi}^{G,\xi}(u,T) = Z_{\omega}^{G,\xi}(u,T)$ if and only if $\beta^G(\{x_2^3 + y^2 + \sum_{i=1}^{K-1} y_i^2 = \xi\}) = \beta^G(\{x_2^3 + \sum_{i=1}^{K} z_i^2 = \xi\})$, $\beta^G(\{\epsilon x_3^4 + y^2 + \sum_{i=1}^{K-1} y_i^2 = \xi\}) = \beta^G(\{\epsilon x_1^4 + \sum_{i=1}^{K} z_i^2 = \xi\})$ and $\beta^G(\{x_2^3 + \epsilon x_3^4 + y^2 + \sum_{i=1}^{K-1} y_i^2 = \xi\}) = \beta^G(\{x_2^3 + \epsilon x_1^4 + \sum_{i=1}^{K} z_i^2 = \xi\})$.
			\item If $q > p+1$, $\eta = -1$, then $Z_{\varphi}^{G}(u,T) = Z_{\omega}^{G}(u,T)$. Furthermore, $Z_{\varphi}^{G,\xi}(u,T) = Z_{\omega}^{G,\xi}(u,T)$ if and only if $\beta^G(\{x_2^3 - y^2 - \sum_{i=1}^{K-1} y_i^2 = \xi\}) = \beta^G(\{x_2^3 - \sum_{i=1}^{K} z_i^2 = \xi\})$, $\beta^G(\{\epsilon x_3^4 - y^2 - \sum_{i=1}^{K-1} y_i^2 = \xi\}) = \beta^G(\{\epsilon x_1^4 - \sum_{i=1}^{K} z_i^2 = \xi\})$ and $\beta^G(\{x_2^3 + \epsilon x_3^4 - y^2 - \sum_{i=1}^{K-1} y_i^2 = \xi\}) = \beta^G(\{x_2^3 + \epsilon x_1^4 - \sum_{i=1}^{K} z_i^2 = \xi\})$.
		\end{itemize}
\end{enumerate}
\end{theo}

\begin{rem} Forgetting the $G$-action, the respective virtual Poincar\'e polynomials of the algebraic subsets $\{x^3 + \epsilon z^4 + \sum_{i=1}^{K} y_i^2 = \xi\}$ and $\{x^3 + \epsilon z^4 - \sum_{i=1}^{K} y_i^2 = \xi\}$, $\epsilon = \pm 1$, $\xi = \pm 1$, of $\mathbb{R}^{K+1}$ can be computed using the invariance of the virtual Poincar\'e polynomial under bijection with $\mathcal{AS}$ graph (see also remarks \ref{remevpsas} and \ref{remforgacthr}). If the equivariant virtual Poincar\'e series was shown to be an invariant under equivariant $\mathcal{AS}$ bijection, it should be possible to compute the above considered quantities.  
\end{rem}

%\newpage

%We want to classify the simple Nash germs invariant under $\sigma$ with respect to $G$-blow-Nash equivalence if $G := \{1, \sigma\}$. First, we want to show that the different Nash germs of the above list are not $G$-blow-Nash equivalent.

%Notice that, for $k \geq 2$, the germ $B_k$ is Nash equivalent to the germ $A_{2k-1}$, by an exchange of the variables $x_1$ and $x_2$. Similarly, for $k \geq 4$, the germ $D_k$ is Nash equivalent to the germ $C_{k-1}$, by the permutation of variables $(x_1,x_2,x_3) \mapsto (x_3,x_1,x_2)$, and $F_4$ is Nash equivalent to $E_6$ by the permutation $(x_1,x_2,x_3) \mapsto (x_3,x_2,x_1)$. Furthermore, we know by \cite{GF-BNT} that the $A,D,E$ singularities all belong to different blow-Nash equivalence classes so a fortiori they belong to different $G$-blow-Nash equivalence classes (equivariant blow-Nash equivalence implies in particular blow-Nash equivalence). In the same idea, in each family of the list, if two elements are not blow-Nash equivalent, they are not $G$-blow-Nash equivalent either.  

%As a consequence, we will study these germs by packs : first, $A_k$ and $B_k$, then $C_k$ and $D_k$ and finally $E_k$ and $F_4$. 

 \vspace{0.5cm}
Fabien PRIZIAC
\\
Institut de Math\'ematiques de Marseille (UMR 7373 du CNRS)
\\
Aix-Marseille Universit\'e
\\
39, rue Fr\'ed\'eric Joliot-Curie
\\
13453 Marseille Cedex 13
\\
FRANCE
\\
fabien.priziac@univ-amu.fr

\end{document}